\documentclass[preprint]{imsart}

\RequirePackage{amsthm,amsmath,amsfonts,amssymb}
\RequirePackage[authoryear]{natbib}%% uncomment this for author-year 
\usepackage{xr-hyper}
\makeatletter
\newcommand*{\addFileDependency}[1]{
\typeout{(#1)}
\@addtofilelist{#1}
\IfFileExists{#1}{}{\typeout{No file #1.}}
}\makeatother

\usepackage{comment}
\usepackage[babel,german=quotes]{csquotes}
\usepackage[utf8]{inputenc}
\usepackage{color}
\usepackage{nicefrac}
\usepackage{xcolor}
\usepackage{graphicx}
\usepackage{enumerate}
\usepackage{natbib}
\usepackage{dsfont}
\usepackage{a4wide}

\usepackage{hyperref}
\hypersetup{pdfstartview={FitH},
colorlinks=true,
linkcolor=black, 
urlcolor=black,
citecolor=black,
bookmarksopen,
bookmarksopenlevel={1},
bookmarksnumbered}
\usepackage[colorinlistoftodos,textsize=tiny]{todonotes}
\newcommand{\Comments}{1}
\newcommand{\mynote}[2]{\ifnum\Comments=1\textcolor{#1}{#2}\fi}
\newcommand{\mytodo}[2]{\ifnum\Comments=1%
\todo[linecolor=#1!80!black,backgroundcolor=#1,bordercolor=#1!80!black]{#2}\fi}

\ifnum\Comments=1               % fix margins for todonotes
\fi

\startlocaldefs
\numberwithin{equation}{section}
\newtheorem{theorem}{Theorem}

\newtheorem{lemma}{Lemma}[section]
\theoremstyle{remark}

\newtheorem{remark}{Remark}

\newtheorem{ass}{Assumption}

\providecommand{\U}[1]{\protect\rule{.1in}{.1in}}
\DeclareMathOperator{\E}{\mathbb{E}}

\DeclareMathOperator{\argmin}{argmin}

\newcommand{\dd}{\mathrm{d}}
\newcommand{\cO}{\mathcal{O}}
\newcommand{\cJ}{\mathcal{J}}
\newcommand{\cG}{\mathcal{G}}

\newcommand{\rd}{\, \mathrm{d}}
\newcommand{\cK}{\, \mathcal{K}}
\newcommand{\cS}{\, \mathcal{S}}
\renewcommand{\cJ}{\, \mathcal{J}}

\newcommand{\R}{\mathbb{R}}
\newcommand{\rP}{\mathrm{P}}
\newcommand{\CH}{C_\mathrm{H}}
\newcommand{\Cf}{C_f}

\newcommand{\CV}{C_{\mathrm{V}}}

\newcommand{\ft}{\mathrm{ft}}
\newcommand{\N}{\mathbb{N}}

\renewcommand{\Pr}{\mathbb{P}}

\endlocaldefs

\begin{document}

% \title{Multivariate root-n-consistent smoothing parameter free matching estimators 
	% and estimators of inverse density weighted expectations }
% \author{Hajo Holzmann\footnote{Corresponding author. Prof.~Dr.~Hajo Holzmann, Fachbereich Mathematik und Informatik, Philipps-Universit\"at Marburg, Hans-Meerweinstr.~6, 35043 Marburg, Germany}\\
	% \small{Fachbereich Mathematik und Informatik}  \\
	% \small{Philipps-Universit\"at Marburg} \\
	% \small{holzmann@mathematik.uni-marburg.de}
	% \and
	% Alexander Meister \\
	% \small{Institut f\"ur Mathematik}  \\
	% \small{Universit\"at Rostock} \\
	% \small{alexander.meister@uni-rostock.de}}
% %

%\date{\today }

%\maketitle

\begin{frontmatter}
	\title{Multivariate root-n-consistent smoothing parameter free matching estimators and estimators of inverse density weighted expectations }
	%\title{A sample article title with some additional note\thanksref{T1}}
	\runtitle{Multivariate root-n-consistent matching estimators}
	%\thankstext{T1}{A sample of additional note to the title.}
	
	\begin{aug}
		%%%%%%%%%%%%%%%%%%%%%%%%%%%%%%%%%%%%%%%%%%%%%%%
		%% Only one address is permitted per author. %%
		%% Only division, organization and e-mail is %%
		%% included in the address.                  %%
		%% Additional information can be included in %%
		\author[A]{\fnms{Hajo}~\snm{Holzmann}\ead[label=e1]{holzmann@mathematik.uni-marburg.de}}
		\and
		\author[B]{\fnms{Alexander}~\snm{Meister}\ead[label=e2]{alexander.meister@uni-rostock.de}}
		%%%%%%%%%%%%%%%%%%%%%%%%%%%%%%%%%%%%%%%%%%%%%%
		%% Addresses                                %%
		%%%%%%%%%%%%%%%%%%%%%%%%%%%%%%%%%%%%%%%%%%%%%%
		\address[A]{Philipps-Universit{\"a}t Marburg\printead[presep={,\ }]{e1}}
		
		\address[B]{Universit{\"a}t Rostock\printead[presep={,\ }]{e2}}
	\end{aug}
	
	\begin{abstract}
		Expected values weighted by the inverse of a multivariate density or, equivalently, Lebesgue integrals of regression functions with multivariate regressors occur in various areas of applications, including estimating average treatment effects, nonparametric estimators in random coefficient regression models or deconvolution estimators in Berkson errors-in-variables models. The frequently used nearest-neighbor and matching estimators suffer from bias problems in multiple dimensions. % and do not attain the parametric rate of convergence in dimensions $d \geq 3$, while having a non-negligible bias in dimension $d=2$. 
		By using polynomial least squares fits on each cell of the $K^{\text{th}}$-order Voronoi tessellation for sufficiently large $K$, we develop novel modifications of nearest-neighbor and matching estimators which again converge at the parametric $\sqrt n $-rate under mild smoothness assumptions on the unknown regression function and without any smoothness conditions on the unknown density of the covariates. We stress that in contrast to competing methods for correcting for the bias of matching estimators, our estimators do not involve nonparametric function estimators and in particular do not rely on sample-size dependent smoothing parameters.  We complement the upper bounds with appropriate lower bounds derived from information-theoretic arguments, which show that some smoothness of the regression function is indeed required to achieve the parametric rate.  Simulations illustrate the practical feasibility of the proposed methods. 
	\end{abstract}
	
	\begin{keyword}[class=MSC]
		\kwd[Primary ]{62H12}
		\kwd[; secondary ]{62G05}
	\end{keyword}
	
	\begin{keyword}
		\kwd{Average treatment effects}
		\kwd{Berkson errors in variables models}
		\kwd{bias correction}
		\kwd{inverse density weighted expectations}
		\kwd{matching estimators}
		\kwd{random coefficients}
		\kwd{Voronoi tessellation}
	\end{keyword}
	
\end{frontmatter}

\section{Introduction}

We consider estimation of inverse density weighted expectations, or, equivalently, of Lebesgue integrals of regression functions, with multivariate regressors.  As we shall discuss in detail, such functionals occur in various  areas of applications, including estimating average treatment effects \citep{abadie2006large, abadie2011bias}, random coefficient regression models \citep{ hoderlein2017triangular, gaillac2022adaptive}, Berkson errors in variables models \citep{zbMATH05071790} or transfer learning under covariate shift \citep{kouw2018introduction, portier2023scalable}.  Frequently used nearest-neighbor as well as matching estimators involve a bias or order $n^{-1/d}$ in $d$-dimensions \citep{abadie2006large}. Therefore, methods for bias correction have been proposed in the treatment effect literature \citep{abadie2011bias, Lin2023Econometrica}. These however involve nonparametric estimation and rates of convergence for   regression functions and their derivatives, potentially under the strong smoothness assumptions.

In this paper we introduce simple modifications of nearest-neighbor and matching estimators for such functionals  which do not involve nonparametric function estimators. In particular our estimators do not rely on sample-size dependent smoothing parameters. However, we show that, under mild smoothness assumptions on the unknown regression function, they converge at the parametric $\sqrt n $-rate again. The upper bounds are complemented by a discussion of lower bounds which show that some smoothness is required for $\sqrt n$-consistent estimation to be possible. 

More precisely, consider a $d$-vector of covariates $Z$ supported on some compact and convex set $S \subseteq \mathbb{R}^d$ and having a  strictly positive  density $f_Z$ on $S$. Further let $S^* \subseteq S$ be some given non-empty Borel set. We consider two closely related estimation problems, in which we focus on the multivariate setting $d \geq 2$.  
%Consider a random vector $(U,Z)$, together with a given real-valued measurable function $g(u,z)$. 
%From $(U,Z)$ we form the real-valued response variable $Y=g(U,Z)$. 
%

First we intend to estimate from i.i.d.~data $(U_j,Z_j)$, $j=1,\ldots,n$, distributed as $(U,Z)$, for some  given real-valued measurable function $g(u,z)$ the functional
\begin{align}\label{eq:thefunctional}
	\Psi \,   := \, \E\Big[ g(U,Z)\, \frac{1_{S^*}(Z)}{f_Z(Z)} \Big] \,  = \, \int_{S^*} G(z) \rd z\,,
\end{align} 
where 
\begin{equation}\label{eq:regressionfct}
	G(z)  := \E\big[g(U,Z) \mid Z=z \big]
\end{equation}
is the regression function of $g(U,Z)$  on $Z$. The issue in estimating $\Psi$ in \eqref{eq:thefunctional} in the representation as expected value is the unknown density $f_Z$ which occurs in the denominator. 

To illustrate the relevance of the functional $\Psi$, consider, for real-valued $U$, series estimators of a regression function $h(z)=\mathbb{E}[U|Z=z]$ in a given orthogonal basis $(\phi_n)_n$ or, more generally, a frame in $L_2(\mu)$ for some measure $\mu$. These require estimation of the inner products
	$ \langle h, \phi_n \rangle_\mu.$ 
	If it is simply assumed that $\mu$ is the distribution of the covariates $Z$ then $\langle h, \phi_n \rangle_\mu = \E[U\,\phi_n(Z)]$. However, if $\mu$ is a given measure such as the Lebesgue measure $\lambda$ on some compact domain determined by an a-priori choice of $(\phi_n)_n$ as e.g.~the Fourier basis, wavelets \citep{kerkyacharian2004regression} or Gabor frames from time-frequency analysis \citep{dahlke2022statistically} this may not be a realistic assumption for random $Z$. Then 
	%, assuming $X$ has Lebesgue density $f_X$, 
	we rather have
	\begin{equation}\label{eq:FCoeff}
		\langle h, \phi_n \rangle_\lambda = \E\Big[\frac{U\, \phi_n(Z)}{f_Z(Z)} \Big],
	\end{equation} 
	leading to the estimation problem considered in \eqref{eq:thefunctional} when putting $g(u,z)= u \cdot \phi_n(z)$. We shall analyze such series estimators with Fourier basis in detail for the Berkson errors-in-variables model \citep{zbMATH05071790}. Further examples involving functionals of the form $\Psi$ include random coefficient regression models \citep{hoderlein2017triangular, gaillac2022adaptive, Gautier2013, masten2018random}.     

Second, from independent and identically distributed observations 
$(Y_i, Z_i)$, $i=1, \ldots, n$, where $Y_i$ is real-valued, we consider estimation of the functional
\begin{equation}\label{eq:estPhi}
	\Phi = \E\Big[Y \, 1_{S^*}(Z)\, \frac{f(Z)}{f_{Z}(Z)} \Big].% = \int_{S^*}\, G(z)\, f(z)\, \dd z,
\end{equation}
Here $f$ is an additional Lebesgue density in $\R^d$. 
%
%where again  
%
%$G(z) = \E[Y | Z = z]$ 
%
If $f$ is known, by putting $g(y,z) = y\, f(z)$, this is a special case of the functional $\Psi$ with $Y$ taking the role of $U$. We devise an estimator in case of unknown $f$, from which an additional sample $X_1, \ldots, X_m$, independent of the $(Y_i, Z_i)$,  is supposed to be available. 
  As we describe in detail, the functional  $\Phi$ in turn arises in the estimation of average treatment effects \citep{abadie2011bias}. Moreover, $\Phi$ also occurs in transfer learning under covariate shift, when we have, at our disposal, labeled data from a source distribution but only unlabeled data from the target distribution. In such settings  our method allows for a novel approach to estimate importance-weighted averages.

In dimension $d=1$, estimating $\Psi$ in  \eqref{eq:thefunctional} has been studied in \citet{lewbel2007simple}. Based on the theory of order statistics and spacings they construct estimators which do not require nonparametric tuning parameters and converge at the parametric $\sqrt n$-rate. Further, they even show how to make their estimator asymptotically efficient by letting the order of the spacings diverge. However, as noted in \citet{lewbel2007simple} their method does not easily extend to multiple dimensions.  

%Further models which require estimating expected values with unknown densities in the denominator, but for which the function $g$ has to be estimated as well include average derivative estimates in single index regression models \citep{zbMATH04153712, zbMATH00446727}, or reweighting in transfer learning under covariate shift \citep{portier2023scalable}.   
%
%
%
%The aim of this paper is to provide and analyze a method to estimate \eqref{eq:thefunctional} in dimensions $d \geq 2$.  
%More specifically, 
We construct an estimator of $\Psi$ in \eqref{eq:thefunctional}  which achieves the parametric $\sqrt n$-rate with neither  dividing by a nonparametric density estimator nor estimating the regression function $G$  nonparametrically. 
Our methods involve least squares polynomial fits on cells of the $K^{\text{th}}$-order Voronoi tessellation,  that is on the $K$-nearest neighbor partition induced by the sample $Z_1, \ldots, Z_n$. We stress that our method does not have tuning parameters which must be chosen depending on the sample size and on the unknown smoothness of $G$, which would be typical of nonparametric methods. 
Compared to the analysis of the one-dimensional case $d=1$ in \citet{lewbel2007simple}, higher-order smoothness assumptions on $G$ are required to achieve the parametric rate, however, no smoothness conditions on the density $f_Z$ of the covariates are needed in our analysis. 
Moreover, using information-theoretic arguments we show that some smoothness of the regression function $G$  is indeed required in order to achieve the parametric rate.

Furthermore we extend the methodology to construct a $\sqrt n$-consistent estimator for the functional $\Phi$ in \eqref{eq:estPhi} in case of unknown density $f$. We discuss how this estimator can be applied for average treatment effects, and  show that for the special case of constant approximations it coincides with the classical matching estimator from the treatment effect literature. Using polynomial least squares fits of appropriate order depending on the dimension, we achieve a bias correction resulting in $\sqrt n$-consistent estimators under much milder smoothness assumptions than in \citet{abadie2011bias} and without invoking nonparametric regression estimators.    

In our proofs we obtain results on the moments of the volume of higher-order Voronoi cells, as well as on properties of design matrices for multivariate polynomial regression under random design which might be of some independent interest. 
Let us briefly discuss some further methodologically related work. \citet{devroye2017measure} study the asymptotic distribution for the volume of a first-order Voronoi cells, and \citet{sharpnack2023} analyzes matching estimators by relying on properties of first-order Voronoi cells. \citet{kallus2020generalized} investigate weighted matching estimators with emphasis on weights from reproducing kernels, and  \citet{samworthnearest} provides comprehensive analysis of weighted nearest neighbor estimators. Finally, \citet{hugreentibsh2022} study nonparametric estimators of a function of bounded total variation by functions which are constant on elements of the first-order Voronoi tessellation. 

The structure of the paper is as follows. In Section \ref{sec:motestbiascorr} we first motivate our method and the need for higher-order bias correction by considering and briefly analyzing a simple special case. Then in Section \ref{sec:estmainresult} we introduce our general method and state our main theoretical result for estimating $\Psi$, that is, the parametric rate of convergence under smoothness assumptions on the regression function $G$. This is extended in Section \ref{sec:matchest} to a $\sqrt n$-consistent estimator of $\Phi$. Section \ref{sec:lowerbounds} complements these upper bounds by providing lower bounds for estimating $\Psi$, which show that some smoothness of $G$ (or $f_Z$) is generally required to  estimate at the parametric rate. In Section \ref{sec:applcompute} we discuss in detail how to apply the methodology to series estimators and in particular the Berkson errors-in-variables problem, to random coefficient regression models, to the estimation of average treatment effects and to estimating importance-weighted averages in transfer learning under covariate shift. Section \ref{sec:sims} contains a simulation study, while Section \ref{sec:conclude} concludes. The main proofs of the upper bounds are given in Section \ref{sec:proofs}, while additional technical arguments as well as further simulations results are deferred to  the supplementary appendix. 

We provide a brief overview of some important notation: the Euclidean norm is denoted by $\| \cdot \|$, the dimension by $d$; and we write $\lambda$ for the $d$-dimensional Lebesgue measure; for $\kappa = (\kappa_1, \ldots , \kappa_d)^\top \in \mathbb{N}_0^d$ we denote $|\kappa| = \kappa_1 + \cdots + \kappa_d$; and, for a function $f(z)$ in $z=(z_1, \ldots, z_d)$, we denote its partial derivatives by 
$$\partial_\kappa f(z) = \frac{\partial^{|\kappa|}}{\partial z_1^{\kappa_1}\, \ldots \partial z_d^{\kappa_d}}f(z).$$ 

%Given $L \in \N_0$ set

\section{Estimation method and rate of convergence}\label{sec:estrate}

In this section we propose estimators for the functionals $\Psi$ in \eqref{eq:thefunctional} and $\Phi$ in \eqref{eq:estPhi}. We start in Section \ref{sec:motestbiascorr} with a special case of our method which illustrates the need for higher-order bias correction in multiple dimensions. Smoothing-parameter free, $\sqrt n$-consistent estimators of the functionals $\Psi$ and $\Phi$ are introduced in Sections \ref{sec:estmainresult} and \ref{sec:matchest}, respectively. These are complemented in Section \ref{sec:lowerbounds} with a discussion of lower bounds.  

\subsection{Motivation}\label{sec:motestbiascorr}

To stress the need for a higher-order bias correction, let us start by considering and briefly analyzing a simple special case for the estimator of the functional $\Psi$ in \eqref{eq:thefunctional}. Here we approximate the integral in \eqref{eq:thefunctional} by a Riemann sum based on the first-order Voronoi tessellation of $S^*$. 
Consider the first order Voronoi cells
$$ C_j \, := \, \big\{z \in S^* \mid \|z-Z_j\| < \|z-Z_k\| \, , \, \forall \ k \neq j\big\}\,. $$
An estimator of $\Psi$ is then given by
\begin{equation}\label{eq:simpleest}
	\hat{\Psi} \, := \, \sum_{j=1}^n g(U_j,Z_j) \cdot \lambda(C_j)\,, 
\end{equation} 
where $\lambda$ denotes the $d$-dimensional Lebesgue measure. Denoting by $\sigma_Z$ the $\sigma$-field generated by $Z_1,\ldots,Z_n$ we obtain that 
$$ \E\big[ \hat{\Psi} \mid \sigma_Z\big] = \sum_{j=1}^n G(Z_j) \cdot \lambda(C_j),$$
a discretized version of \eqref{eq:thefunctional}. Here $G$ is the regression function defined in \eqref{eq:regressionfct}. 

 In the decomposition of the mean squared error
	\begin{align}\label{eq:condvarbias}
		\E \big[\big|\hat{\Psi} - \Psi\big|^2\big] &  \, = \, \E \, \big[\mbox{var}\big(\hat{\Psi} | \sigma_Z\big)\big] \, + \, \E \big[\big|\E[\hat{\Psi} | \sigma_Z] - \Psi\big|^2\big] \,,
	\end{align}
	under the Assumptions \ref{ass:deigndens} and \ref{ass:boundeffect} below and if $\mbox{var}\big(g(U_1,Z_1) | Z_1\big)$ is uniformly bounded one can bound the expected conditional variance by  
	$ \E\big[ \mbox{var}\big(\hat{\Psi} | \sigma_Z\big) \big] \lesssim n^{-1}$ from above.
	%
	%
	%To show why this simple estimator does not achieve the parametric rate in dimensions $d >3$ and even has issues for $d=2$, consider the  and expected  squared conditional bias decomposition 
	%
	%say by the constant $\CV>0$, we have that 
	%
	%\begin{align*}% \label{eq:var}
	%	 \E\big[ \mbox{var}\big(\hat{\Psi} | \sigma_Z\big) \big] & \, = \, \E\Big[  \sum_{j=1}^n \lambda^2(C_j) \cdot \mbox{var}\big(g(U_j,Z_j) | Z_j\big)\Big] \, \leq \, n\, \CV \, \E \big[\lambda^2(C_1)\big] , \end{align*}	
%
%since the $\lambda^2(C_j) $ are identically distributed and 
%
%Now  $ \displaystyle \E \big[\sum_{j=1}^n \lambda^2(C_j)\big] = n \, \E \big[\lambda^2(C_1)\big]$, 
%Lemma \ref{lem:vorcellorder} in Section \ref{sec:expcondvar} implies that $\E \big[\lambda^2(C_1)\big] \lesssim n^{-2} $.   
%
%\hajo{muss noch formuliert werden}
%so
%
However, for the squared conditional bias, if $G$ is Lipschitz continuous with constant $\CH>0$ we have that
\begin{align*}% \label{eq:bias1}
	&	\E \big[\big|\E[\hat{\Psi} | \sigma_Z] - \Psi\big|^2\big]   \, = \, \E \Big[\Big|\sum_{j=1}^n \int_{C_j} \big(G(Z_j)-G(z)\big) \rd z\Big|^2\Big]\,\\
	\, \leq \, &\, \CH^2\, \sum_{j,k=1}^n \iint_{(S^*)^2}\,  \E\big[1_{C_j}(z)\, 1_{C_k}(z')\, \|z - Z_{(1)}(z)\| \,\|z' - Z_{(1)}(z')\|\,  \big] \, \rd z \, \rd z' \, \lesssim \, n^{-2/d},	
\end{align*}
where $Z_{(1)}(z)$ denotes the first-nearest neighbor of $z$ in $Z_1, \ldots, Z_n$, and the final bound follows from Lemma \ref{lem:biasone} in the Section \ref{sec:squaredcondbias}. Thus in dimensions $d \geq 3$ this upper bound on the order of the bias is slower than the parametric rate, and even for $d=2$ it is not negligible compared to the variance.
To improve the rate of the bias, instead of the constant approximation to $G$ on each Voronoi cell, we use polynomial approximation. In order to make this well-defined and computable we require higher-order Voronoi cells. 

%$$ \E\big[ \mbox{var}\big(\hat{\Psi} | \sigma_Z\big) \big] \lesssim n^{-1}.$$

\subsection{Asymptotically unbiased, $\sqrt n$-consistent estimation of $\Psi$}\label{sec:estmainresult}

Our general method for estimating $\Psi$ in \eqref{eq:thefunctional} uses polynomial approximation of degree $L$ and the $K$-th order Voronoi tessellation of $S^*$ for sufficiently large integers $K,L \in \N$ depending on the dimension $d$.  
For each $J \subseteq \{1, \ldots, n\}$ with $\# J=K$ we set
\begin{equation}\label{eq:voronoihighorder}
C(J) \, := \, \big\{z \in S^* \, \mid \, \|z-Z_{j}\|  < \|z - Z_{k}\| \quad \forall \ j \in J, \ k \in \{1,\ldots,n\} \backslash J \big\}\,.
\end{equation} 
Then for  $z \in C(J)$ we let
$\widehat{G}(z)$ be the least squares polynomial of degree  $L$  in  $d$ variables when regressing  $g(U_j,Z_j)$ on $Z_j$ for indices $j \in J$. The exact definition of $\widehat{G}(z)$ is given in (\ref{eq:estimagexp}). Let ${\cal J}_K$ be the collection of all subsets of $\{1,\ldots,n\}$ with exactly $K$ elements. We set 
%
%\hajo{Notation $\tilde G$ oder $\hat G$}
%
\begin{equation}\label{eq:estfinal}
\hat{\Psi} \, =\, \hat{\Psi}(L,K) \, := \, \sum_{J\in {\cal J}_K} \int_{C(J)} \, \widehat{G}(z) \, \rd z\,. 
\end{equation} 
The parameters $K$ and $L$ need not be chosen depending on the sample size $n$ and hence do not take the role of smoothing parameters. They simply have to be sufficiently large depending on the dimension $d$, see Remark \ref{rem:choiceLK}. Also note that $\widehat{G}(z)$ is a weighted average of $g(U_j,Z_j)$ for those $j$ corresponding to the $K$-nearest neighbors of $z$ in $Z_1, \ldots, Z_n$. Specifically, for $L=0$ we obtain
\begin{equation}\label{eq:estlto0}
\hat{G}(z) \, = \, \frac1K\, \sum_{j\in J} g(U_{j},Z_{j}) \quad \text{if } z \in C(J), \quad \text{ and}\quad  \hat{\Psi} \, = \, \sum_{J\in {\cal J}_K} \, \Big(  \frac1K\,\sum_{j\in J} g(U_{j},Z_{j})\Big)\,\lambda(C(J))\,. 
\end{equation}

For general $L$, more formally we introduce the notation
%recall the notation ${\cal K}$ from \eqref{eq:indexset} and set  
%
%\begin{equation}\label{eq:indexset}
%\end{equation}
%
%
\begin{align*}
%{\cal K} & \, := \, {\cal K}(d,L) \, := \, \big\{\kappa \in \mathbb{N}_0^d \mid \kappa_1 + \cdots + \kappa_d \leq L\big\}\,, \qquad K^*= K^*(d,L) = \# {\cal K},\\
{\cal K}  & \, := \, {\cal K}(d,L) \, := \, \big\{\kappa \in \mathbb{N}_0^d \mid |\kappa| \leq L\big\}\,, \\%\qquad K^*= K^*(d,L) = \# {\cal K},\\
\xi_\kappa(z,w)  & \, := \, \prod_{k=1}^d (w_k - z_k)^{\kappa_k}\,, \ \kappa \in \mathcal K, \ w \in \mathbb{R}^d\,,\qquad \qquad \xi(z,w) = \big( \xi_\kappa(z,w) \big)_{\kappa \in \cK}.
\end{align*}
Monomials of degree $\leq L$ are indexed by elements in ${\cal K}$, so that $K^*$ is the number of free parameters of a polynomial in $d$ variables of total degree $\leq L$ (see (\ref{eq:constants}) for its precise definition). We consider $\xi_\kappa(z,w)$ as a monomial function in $w$, centered at $z$. Then $\xi(z,w)$ is a vector of functions in $w$ of dimension $K^*$. 

For $z \in S^*$ let $J(z)$ denote that $J \in {\cal J}_K$ which satisfies $z \in C(J)$, and note that Lebesgue-almost all $z \in S^*$ belong to some $C(J)$, since the boundaries of the $C(J)$ are sets of measure $0$. 

The least squares polynomial $\widehat{G}(z)$  of degree  $L$  in  $d$ variables at $z$  when regressing $Z_j$ on $g(U_j,Z_j)$ for indices $j \in J$ is obtained as the first entry $\gamma_{(0,\ldots,0)}$ of the vector
$$ \argmin_{\gamma \in \mathbb{R}^{\cal K}} \sum_{j\in J} \big(g(U_{j},Z_{j}) - \gamma^\top  \cdot \xi(z,Z_{j})\big)^2\,  = {\cal M}(z)^{-1} \hat{\cal G}(z), $$

%$$ \argmin_{\gamma \in \mathbb{R}^{\cal K}} \sum_{j\in J} \Big|g(U_{j},Z_{j}) - \sum_{\kappa\in{\cal K}} \gamma_\kappa \cdot \xi_\kappa(z,Z_{j})\Big|^2\, : = {\cal M}(z)^{-1} \hat{\cal G}(z), $$
%
where
\begin{align*}
{\cal M}(z) & \, := \,\sum_{j\in J(z)} \xi(z,Z_{j})\, \xi(z,Z_{j})^\top \, = \,  \Big(\sum_{j\in J(z)} \xi_{\kappa+\kappa'}(z,Z_{j})\Big)_{\kappa,\kappa' \in {\cal K}}\,,\\
\widehat{\cal G}(z) & \, := \, \sum_{j\in J(z)} g(U_{j},Z_{j}) \cdot \xi(z,Z_{j})\,,
\end{align*}
and the matrix ${\cal M}(z)$ is invertible and, hence, strictly positive definite almost surely for sufficiently large $K$ (this will be shown in Lemma \ref{L:u0bound}), making the estimator $\widehat{G}(z)$ uniquely defined, see Section \ref{sec:expcondvar}. 
Thus, we can write 
\begin{equation}\label{eq:estimagexp}
\widehat{G}(z) \, := \, {\bf e}_0^\top {\cal M}(z)^{-1} \widehat{\cal G}(z)\, = \, \sum_{j\in J(z)} g(U_{j},Z_{j}) \cdot {\bf e}_0^\top \, {\cal M}(z)^{-1} \, \xi(z,Z_{j}) 
\end{equation}
with the unit row vector ${\bf e}_0$ whose component for $\kappa = {\bf 0}$ is $1$.

\medskip

Our main result requires the following two assumptions. 

\begin{ass}[Support and design density]\label{ass:deigndens}
The support $S$ of $Z$ is a compact and convex subset of $S \subseteq \mathbb{R}^d$. The diameter of $S$ is called $\rho''$. The distribution of $Z$ has a Lebesgue-density $f_Z$ that is  bounded away from zero by some $\rho>0$ on $S$ and bounded from above by $\overline{\rho}$. %is supported on some compact and convex set $S \subseteq \mathbb{R}^d$;
\end{ass}

\begin{ass}[Exclusion of boundary effects]\label{ass:boundeffect}
We assume that there exists some $\rho'>0$ such that the Euclidean balls $B_{d}(z, \rho')$ with the center $z$, for all $z \in S^*$, and the radius $\rho'$ are included in $S$ as a subset. 
\end{ass}

Let $B_{d}(S^*,\rho') = \bigcup_{z \in S^*} B_{d}(z, \rho')$ denote the open $\rho$-neighborhood of $S^*$. We assume that the restriction of the regression function $G$ to $B_{d}(S^*,\rho')$ belongs to the Hölder class $\cG(l, \beta, \CH)$ with parameters $l \in \N_0$, $0< \beta \leq 1$, $\CH>0$ of functions on $B_{d}(S^*,\rho')$ defined as follows:
\begin{align}
\cG(l, \beta, \CH) & = \Big\{ h: B_{d}(S^*,\rho') \to \R \mid h \text{ has partial derivatives of order } \leq l, \nonumber \\
& \qquad \qquad \text{for each } \kappa \in \N_0^d, \text{ with } |\kappa| = l \text{ we have that }\label{eq:hoelderclass}\\ 
& \qquad \qquad \big|\partial_\kappa h(z) - \partial_\kappa h(z') \big| \leq \CH\, \| z - z'\|^{\beta},\quad z,z' \in B_{d}(S^*,\rho')\Big\}.\nonumber
\end{align}
Furthermore we set
\begin{align} \nonumber 
K^* & = K^*(d,L) = \# {\cal K} = \sum_{l=0}^L\, \binom{d+l-1}{l}  = {d+L \choose L}, \\ \label{eq:constants}
D & = D(d,L ) = \sum_{l=1}^L\, l\, \binom{d+l-1}{l}  = d\cdot {d+L \choose L-1}. 
\end{align}
%
%Note that $K^*(d,L)$ is the number of monomials in $d$ variables of degree at most $L$. 
%
\begin{theorem}\label{th:theoremparrate}
%
%\hajo{$d \geq 2$ nur benötigt für $L \geq 1$?} \alex{Ich denke, dass wir den Fall $d=1$ nicht extra zu erw\"ahnen brauchen, da man dort sowieso eher mit den Ordnungstatistiken arbeiten würde.}
Consider estimation of the functional $\Psi$ in \eqref{eq:thefunctional} in dimensions $d \geq 2$ under the Assumptions \ref{ass:deigndens} and \ref{ass:boundeffect}. Suppose that the regression function $G$ in \eqref{eq:regressionfct} belongs to the Hölder class $\cG(l, \beta, \CH)$ in \eqref{eq:hoelderclass} with parameters $l \in \N_0$, $0< \beta \leq 1$, $\CH>0$. Moreover assume that the conditional variance $\mbox{var}\big(g(U_1,Z_1) | Z_1\big)$ is uniformly bounded by a constant $\CV>0$.  
In the estimator $\hat \Psi(L,K)$ in \eqref{eq:estfinal}, choose $L = l$. If $L=l=0$ take any fixed $K \geq 1$, while for $L \geq 1$ choose some $K \geq 2 + (2D + 1)\, K^*,$ with $K^*$ and $D$ defined in \eqref{eq:constants}. 
%\item Regressions function $G$ is $l$ - times partially differentiable, , 
%\begin{align*}
%K^*(d,L) & \, := \, \text{ number of monomials of order } \leq L,\\
%D(d,L) & \, := \, \sum_{\ell=1}^L \ell \cdot \#\big\{\kappa\in \mathbb{N}_0^d \, : \, \kappa_1 + \cdots + \kappa_d = \ell\big\}\,.
%\end{align*}
Then for the mean conditional variance and the mean squared conditional bias given the $\sigma$-field $\sigma_Z$ generated by $Z_1, \ldots, Z_n$ we have that
\begin{align}
	& \sup_{G \in \cG(l, \beta, \CH)} \E\big[ \mathrm{var}\big(\hat{\Psi} | \sigma_Z\big) \big] \leq \CV \cdot C \cdot  n^{-1},\nonumber\\ 
	& \sup_{G \in \cG(l, \beta, \CH)} \E \big[\big|\E[\hat{\Psi} | \sigma_Z] - \Psi\big|^2 \big] \leq \CH^2 \cdot C\cdot  n^{-\frac{2(l + \beta)}{d}},\label{eq:boundsth}
\end{align} 
where the constant $C>0$ only depends on $L,K,\rho, \rho',\rho'',\bar \rho, d$ and $\lambda (S^*)$. 
Thus, if 	$l + \beta \geq d/2$, 
$$ \sup_{G \in \cG(l, \beta, \CH)} \E \big[\big|\hat{\Psi}  - \Psi\big|^2 \big] \lesssim n^{-1},$$
and the contribution of the squared bias is negligible if $l + \beta > d/2$. 
\end{theorem}	
The main steps of the proof are provided in Section \ref{sec:proofthe1}, with some technical details being deferred to Section \ref{sec:technicaldetails}  in the supplementary appendix. 
\begin{remark}[Choice of $L$ and $K$]\label{rem:choiceLK}  For nested smoothness classes, e.g.~if lower-order derivatives are  uniformly bounded in addition, one can work with minimal smoothness: making the bias negligible requires $l + \beta > d/2$, which is satisfied for $l = \lfloor d/2\rfloor$ and $\beta >0$ for even $d$ as well as $\beta > 1/2$ and odd $d$, as well as for all $l \geq \lfloor d/2\rfloor +1$. Thus, if smoothness is such that the bias can be made negligible this is always achieved by choosing the polynomial degree $L=\lfloor d/2 \rfloor$, resulting in $L=1$ in the important cases $d=2$ and $d=3$.

	Next, for given dimension $d$ and order of the polynomial approximation $L$, the theorem specifies a minimal choice $K \geq 2 + (2D + 1)\, K^*$, $K^*$ in \eqref{eq:constants}, under which the upper  risk bounds are theoretically guaranteed. 
Specifically, the minimal value of $K$ for $L=1$ is $K = 2 + (2\, d+1)\, (1+d)$. For $d=2$ and $L$ our result requires $K= 17$, and for $d=3$ and $L=1$  we need $K=30$ in our theoretical result.  
%Computation of the estimator is numerically challenging and requires large sample sizes for such high values of $K$. 
In the simulations, we illustrate that smaller values of $K$ often seem to suffice.%, and give some practical recommendations for the choice of $K$ for $L=1$ depending on $d$. 
%
%\begin{itemize}
%   \item $L=1$: \quad . 
%    \item $L=2$: \quad $K = 2 + \big(2\, (d+2)\,d+1\big)\, \big(1+d + (d+1)\,d/2\big)$,
%\end{itemize}

\end{remark}

%\begin{remark}[Constants and boundary effect]\label{rem:choiceLK}
%
%Constants in \eqref{eq:boundsth}, boundary effect  \alex{noch}

%\end{remark}

%\begin{remark}[Weighted nearest neighbor estimators]\label{rem:choiceLK}
%    \alex{noch}
%\end{remark}

\begin{remark}[Higher moments of the volume of $C(J)$]\label{rem:highmom}
For the derivation of upper bounds as in Theorem \ref{th:theoremparrate}, the asymptotic properties of the $d$-dimensional Lebesgue measure $\lambda(C(J))$ of the Voronoi cells as $n\to\infty$ are essential. Since 
$$ \lambda(S^* ) = \sum_{J \in {\cal J}_K} \lambda(C(J))$$
and the $C(J)$ are identically distributed we have that $\E[\lambda(C(J))] = \cO(n^{-K}) $. Hence $\lambda(C(J)) = \cO_\Pr(n^{-K}) $ which implies $\lambda(C(J))^\ell = \cO_\Pr(n^{- \ell\, K}) $. 
As a surprising phenomenon the higher moments of this term show an unusual behaviour since  $\E[\lambda(C(J))^\ell] \simeq n^{-\ell-K+1}$. The upper bound is provided in Lemma \ref{lem:vorcellorder} for $\ell = 2$, and for the lower bound see Section \ref{sec:technicaldetails} in the supplementary appendix. 
\end{remark}

\subsection{Asymptotically unbiased matching estimators for $\Phi$}\label{sec:matchest}

Now let us turn to estimating the functional $\Phi$ in \eqref{eq:estPhi} in the setting where $f$ is unknown. 
Suppose that we have an additional independent sample $X_1, \ldots, X_m$, each $X_j$ with density $f$, independent of $(Y_1, Z_1), \ldots, (Y_n, Z_n)$. 
One possibility to estimate \eqref{eq:estPhi} is to use a nonparametric estimator $\hat f$ of $f$ from $X_1, \ldots, X_m$, and then to employ \eqref{eq:estfinal} with the estimated function $\hat g(y,z) = y\, \hat f(z)$. 
Alternatively, we can avoid a nonparametric estimator of $f$ and proceed as follows. 

Consider the Voronoi cells as in \eqref{eq:voronoihighorder}, and for $J \subseteq \{1, \ldots, n\}$ with $\# J=K$, that is $J \in \cJ_K$, 
and   $z \in C(J)$ we let
$\widehat{G}(z)$ be the least squares polynomial fit of order $L$ as in \eqref{eq:estimagexp}, with the observed $Y_j$ replacing $g(U_{j},Z_{j})$. Then we set
\begin{equation}\label{eq:matchest}
\widehat{\Phi} = \widehat{\Phi}(L,K) = \frac1m\, \sum_{k=1}^m\, \sum_{J \in \cJ_K} \, 1\big(X_k \in C(J)\big)\, \widehat{G}(X_k).
\end{equation}
Consider the class of functions
$$ \cG(l, \beta, \CH, C_G) = \Big\{ h: B_{\rho'}(S^*) \to \R \mid h \in \cG(l, \beta, \CH),\ |h(z)| \leq C_G,\  \forall z \in B_{\rho'}(S^*)\Big\}.$$
\begin{theorem}\label{th:theoremparratematch} %\alex{Manches aus Theorem 1 musste  angepasst werden, bitte nochmals überprüfen}
Consider estimation of the functional $\Phi$ in \eqref{eq:estPhi} in dimensions $d \geq 2$ under the Assumptions \ref{ass:deigndens} and \ref{ass:boundeffect}. %As in Theorem \ref{th:theoremparrate}, suppose that the regression function $G$ in \eqref{eq:regressionfct} belongs to the Hölder class $\cG(l, \beta, \CH)$ in \eqref{eq:hoelderclass} with parameters $l \in \N_0$, $0< \beta \leq 1$, $\CH>0$. 
Moreover assume that the conditional variance $\mbox{var}\big(Y_1 | Z_1\big)$ is uniformly bounded by a constant $\CV>0$, and that the density $f$ is bounded from above by $\Cf$.   
In the estimator $\hat \Phi(L,K)$ in \eqref{eq:matchest}, choose $L = l$. If $L=l=0$ take any fixed $K \geq 1$, while for $L \geq 1$ choose some $K \geq 2 + (2D + 1)\, K^*,$ with $K^*$ and $D$ defined in \eqref{eq:constants}. 
Then 
\begin{equation}\label{eq:riskboundmatch}
	\sup_{G \in \cG(l, \beta, \CH, C_G)} \E \big[\big|\hat{\Phi}  - \Phi\big|^2 \big] \leq C\, \big((\CV + C_G^2)\, m^{-1} + \CV\,   n^{-1}  +  \CH^2 \, n^{-\frac{2(l + \beta)}{d}}\big),
\end{equation} 
where the constant $C>0$ only depends on $L,K,\rho, \rho',\rho'',\bar \rho, d, \lambda (S^*)$ and $\Cf$. 
Thus, if 	$l + \beta \geq d/2$, 
$$ \sup_{G \in \cG(l, \beta, \CH, C_G)} \E \big[\big|\hat{\Phi}  - \Phi\big|^2 \big] \lesssim n^{-1} + m^{-1},$$
and the contribution of the squared bias is negligible if $l + \beta > d/2$. 
\end{theorem}	
The proof is provided in Section \ref{sec:proofmatchingest}. 
\begin{remark}[Matching estimators]
For $L=0$ the estimator in \eqref{eq:matchest} corresponds to the matching estimator from the treatment effect literature \citep{abadie2006large, imbens2004nonparametric}: For $k=1, \ldots, m$ let $\{ Z_j \mid j \in J_k\}$, $J_k \subseteq \{1, \ldots, n\}$ with $\# J_k=K$ be  the set of $K$-nearest neighbors of $X_k$ in $Z_1,\ldots, Z_n$, and set $\hat Y_k = K^{-1} \, \sum_{j \in J_k} Y_j$. Then \eqref{eq:matchest} can be written as
\begin{equation*}
	\widehat{\Phi} = \frac1m\, \sum_{k=1}^m\, \sum_{J \in \cJ_K} \, 1\big(X_k \in C(J)\big)\, \frac1K \, \sum_{j \in J} Y_j =  \widehat{\Phi} = \frac1m\, \sum_{k=1}^m\, \hat Y_k.
\end{equation*}
See Section \ref{sec:averagetreat} for further discussion of estimating average treatment effects. 

\end{remark}

\subsection{Discussion of lower bounds}\label{sec:lowerbounds}

In the oracle setting when $f_Z$ is perfectly known the usual parametric rate is always attained by the estimator
$$ \hat{\Psi}_{\mbox{oracle}} \, := \, \frac1n \sum_{j=1}^n g(U_j,Z_j) / f_Z(Z_j)\,, $$
under the Assumptions \ref{ass:deigndens} and \ref{ass:boundeffect} and the condition that $\mbox{var}(g(U_1,Z_1)|Z_1)$ is bounded, where no smoothness constraints are required. 

On the other hand we can show, inspired by techniques from quadratic functional estimation (e.g. \cite{Tsybakov2009nonparametric}), that, under unknown $f_Z$, the parametric rate cannot be kept when no smoothness conditions are imposed. This occurs even in the univariate setting, as demonstrated in the following theorem. 
\begin{theorem} \label{T:low} 
Consider the classical univariate nonparametric regression model $U_j = h(Z_j) + \varepsilon_j$, $j=1,\ldots,n$. Impose that both the design density $f_Z$ is contained in the class ${\cal F}={\cal F}(\rho,\overline{\rho})$ of all measurable functions on $[0,1]$ which are bounded from above by $\overline{\rho}>1$ and from below by $\rho>0$; the regression function $h$ lies in the class ${\cal H}={\cal H}(\tilde{\rho})$ of all measurable functions bounded by $\tilde{\rho}>0$; and the regression errors are $\varepsilon_j$ are standard Gaussian and independent of $Z_j$. Then,
$$ \liminf_{n\to\infty}\, n^{1/2} \cdot \sup_{f_Z\in {\cal F},h\in {\cal H}} \, \mathbb{E}_{f_Z,h} \Big|\hat{H}_n - \int_{1/4}^{3/4} h(x) dx\Big|^2\, > \, 0\,, $$
for any sequence of estimators $\hat{H}_n$ based on the data $(U_j,Z_j)$, $j=1,\ldots,n$.  
\end{theorem}
The proof is given in the supplement (Section \ref{Section:low}). 
Note that the regression model in Theorem \ref{T:low} is included in the general framework by putting $g(u,z)=u$ and the conditional density $f_{U|Z}(u|z) = f_\varepsilon(u - h(z))$ when $h$ denotes the regression function and $f_\varepsilon$ stands for the standard normal density. 

\begin{remark}[Continuous design density and regression functions] \label{rem:continuousdesreg}
We mention that the claim of Theorem \ref{T:low} still holds when the classes ${\cal F}$ and ${\cal H}$ are changed to the corresponding sub-classes ${\cal F}^0$ and ${\cal H}^0$ which contain only the continuous functions in ${\cal F}$ and ${\cal H}$, respectively, as ${\cal F}^0$ and ${\cal H}^0$ are dense subsets with respect to the $L_1$-metric. 
\end{remark}

\begin{remark}[Fixed design]\label{rem:fixeddesign}
Linear functional estimation in multivariate regression with fixed design is not included in the general framework of this work. In particular, Assumption \ref{ass:deigndens} is violated. Still we mention that, in this setting, no estimator can attain faster convergence rates than $n^{-2(l+\beta)/d}$ as they occur as an upper bound in Theorem \ref{th:theoremparrate}. This can be seen as follows. Assume that the design points $Z_1,\ldots,Z_n$ form an equidistant grid of some $d$-dimensional cube $S$ so that $\|Z_j-Z_k\|\geq c\cdot n^{-1/d}$ for all $j\neq k$ and some fixed constant $c>0$. Consider the competing regression functions $h_0 \equiv 0$ and $h_n(z) := \sum_{j=1}^n b^{l+\beta} \, K(\|z-Z_j\|/b)$ for some non-negative kernel function $K$ which is infinitely often differentiable on $\mathbb{R}$, supported on $[0,1]$ and satisfies $K(0)=0$ as well as $K(1/2)>0$, while $b = (c/2)\cdot n^{-1/d}$. Note that $h_0$ and $h_n$ lead to the same distribution of the data, whereas the (non-squared) distance between $\int_{S^*} h_0$ and $\int_{S^*} h_n$ is bounded from below by the rate $n \cdot b^{l+\beta}\cdot b^d \asymp n^{-(l+\beta)/d}$ for any open non-void subset $S^*$ of $S$. Note that these arguments cannot be applied to the random design case as, then, the competing regression functions must not depend on the design variables $Z_1,\ldots,Z_n$. The rate $n^{-(l+\beta)/d}$ also occurs in \citet{Kohler2014} for the $L_1$-risk in noiseless regression. \end{remark}

\section{Applications}\label{sec:applcompute}

%\subsection{Applications}\label{sec:applcompute}

In this section we present various applications of our methodology.  
In Section \ref{sec:seriesBerkson} we give a detailed analysis of 
%a given orthogonal system with respect to the Lebesgue measure, where for random design the Fourier coefficients are of the form \eqref{eq:thefunctional}. %, see \eqref{eq:FCoeff} below.
the nonparametric Berkson errors in variables problem \citep{zbMATH05071790}, where our estimator is an ingredient of a Fourier series deconvolution estimator. In Section \ref{sec:randomcoef} we indicate how estimating functionals as in \eqref{eq:thefunctional} arise in non- and semiparametric estimators in random coefficients regression models. 

Sections \ref{sec:averagetreat} on estimating average treatment effects and \ref{sec:transferlearn} on transfer learning involve the functional \eqref{eq:estPhi} with unknown density $f_0$, for which as in Section \ref{sec:matchest} an additional sample is available. For the average treatment effect on the treated and for the average treatment effect over a subset of the covariate space we achieve a bias correction resulting in $\sqrt n$-consistent estimators without involving nonparametric estimators of the regression functions and under much milder smoothness assumptions than in \citet{abadie2011bias}.  

\subsection{Berkson errors in variables models}\label{sec:seriesBerkson}

Consider the nonparametric Berkson errors-in-variables model, in which one observes $(U,Z)$ according to 
\begin{equation}\label{eq:berksonerrorsinvar}
U = h(Z+\delta) + \epsilon,
\end{equation}
where $\epsilon, \delta$ and $Z$ are independent and $\epsilon$ is centered and square-integrable. When estimating $h$ by deconvolution methods, \citet{zbMATH05071790} use a local linear estimator of the calibrated regression function $g(z) = \E[U | Z=z]$  and numerically compute its Fourier coefficients. \citet[Section 3.4.2]{Meister2009Deconvolution} shows how to avoid nonparametric estimation of $g$ in one dimension $d=1$ by using spacings. Here we extend these results to higher dimensions $d \geq 2$ using our novel approach. 

We choose the Lebesgue measure $\lambda$ and the Fourier basis 
$$\phi_{\bf j}({\bf x}) = \exp\Big(i\sum_{k=1}^d j_k x_k\Big), \qquad {\bf x} = (x_1,\ldots,x_d)$$ and the multi-index ${\bf j} = (j_1,\ldots,j_d)$ where all components $j_k$, $k=1,\ldots,d$, are integers satisfying $|j_k|\leq J_n$ for some smoothing parameter $J_n$. Assume that $\delta$ has a known $d$-dimensional Lebesgue density and let $\tilde f_{\delta}$ denote the density of $-\delta$. We derive our result under the following set of assumptions. 

\begin{ass}\label{ass:deconv}
Impose that $h$ and $f_\delta$ are supported on some fix compact subset of $(-\pi,\pi)^d$; that the support of the convolution $h*\tilde{f}_\delta$ is included in some fixed compact set $S^*$, which is contained in the interior of the compact and convex support $S$ of $f_Z$ as a subset; that $f_Z$ is bounded away from zero on $S$ (such that the Assumptions \ref{ass:deigndens} and \ref{ass:boundeffect} are satisfied); that the known error density $f_\delta$ is ordinary smooth, i.e. its Fourier coefficients $f_\delta^{\ft}({\bf j}) := \langle f_\delta , \phi_{{\bf j}} \rangle_\lambda$ satisfy 
$$ c_\delta \cdot (1+\|{\bf j}\|)^{-\gamma} \, \leq \, \big|f_\delta^{\ft}({\bf j})\big| \, \leq \, C_\delta\cdot (1+\|{\bf j}\|)^{-\gamma}\,, $$
for all ${\bf j}$ and some constants $0<c_\delta< C_\delta$ and $\gamma > d/2$, which guarantees that $f_\delta$ is square-integrable on the whole of $\mathbb{R}^d$.  
\end{ass}

Such conditions are also imposed in \citet{zbMATH05071790} and \citet[Section 3.4.2]{Meister2009Deconvolution} where, in the univariate setting without local polynomial approximation, the set $S^*$ may correspond to $S$. Now apply our estimator $\hat{\Psi}=\hat{\Psi}_{{\bf j}}$ in (\ref{eq:estfinal}) for 
$$ g(u,z)=g_{{\bf j}}(u,z)= u\cdot \phi_{{\bf j}}(z)\,, $$
so that 
$ G(z) = G_{{\bf j}}(z) = \phi_{{\bf j}}(z)\cdot \big[h*\tilde{f}_\delta\big](z)\,, $
and $S^*$ as in Assumption \ref{ass:deconv}. Finally we consider the estimator
$$ \hat{h} \, = \, (2\pi)^{-d} \, \sum_{{\bf j}\in {\bf J}^{(n)}} \hat{\Psi}_{{\bf j}} \cdot \phi_{-{\bf j}} / f_\delta^{\ft}(-{\bf j})\, $$
of the regression function $h$, where ${\bf J}^{(n)}:=\{-J_n,\ldots,J_n\}^d.$
\begin{theorem}\label{th:ratedeconv}
Consider the Berkson errors in variables model \eqref{eq:berksonerrorsinvar} under Assumption \ref{ass:deconv} for $d \geq 2$.  Furthermore suppose that the regression function $h$ fulfills the Sobolev condition 
\begin{equation}\label{eq:sobolev}
	\sum_{{\bf j}} |h^{\ft}({\bf j})|^2 \, (1+\|{\bf j}\|^{2\alpha}) \, \leq \, C_\alpha^*\,, 
\end{equation}
for constants $\alpha, C_\alpha^* > 0$, where $\alpha$ is assumed to be sufficiently large such that
$$ 4(\alpha-1)(\alpha+\gamma) \geq d(2\alpha+2\gamma+d)\,. $$ 
Then under the smoothing regime $J_n \asymp n^{1/(2\alpha+2\gamma+d)}$, the selection $L=\lceil\alpha\rceil-1$ and the choice of $K$ according to Theorem \ref{th:theoremparrate} in the construction of estimator $\hat{h}$ we obtain uniformly over the class \eqref{eq:sobolev} 
\begin{equation}\label{eq:boundberkson}
	\mathbb{E} \Big[\int_{[-\pi,\pi]^d} |\hat{h}(x) - h(x)|^2 \dd x\Big] \, \lesssim \, n^{-2\alpha/(2\alpha+2\gamma+d)}. 
\end{equation} 
%
%as an upper bound on the MISE. 
\end{theorem}
The proof is provided in the supplement (Section \ref{sec:proofapps}).

\subsection{Random coefficient regression}\label{sec:randomcoef}

Random coefficient regression models form another field of application of our methods. We consider the linear random coefficient regression model
\begin{equation}\label{eq:linearrandomcoef}
Y = \alpha + \beta^\top X ,
\end{equation}
where $Y$ is the observed scalar response, $X$ is a $d$-dimensional covariate vector, and $\beta \in \R^d$ and $\alpha \in \R$ are random regression coefficients, with $(\alpha, \beta)$ being independent of $X$. Note that $\alpha$ contains the deterministic intercept $\beta_0$ and the centered random errors $\epsilon$, that is $\alpha = \beta_0 + \epsilon$.  
The joint density $f_{\alpha, \beta}$ is of interest and to be estimated. The model \eqref{eq:linearrandomcoef} contains heterogeneity in a linear equation, as is often required in econometric applications.

%(take $\epsilon= \beta_{p+1}$, $X_{p+1}=1$). Assume $\beta$ random, and $X$ and $\beta$ independent. 

%Aim: Estimation of the density $f_\beta$ of $\beta$. 

\citet{gaillac2022adaptive} propose an estimator which allows for compactly supported covariates $X$. They use the identified relation 
$$\E[\exp(i t Y) | X=x] = f^{\ft}_{\alpha, \beta}(t,tx),$$
where $f^{\ft}_{\alpha, \beta}$ is the $(d+1)$ - dimensional Fourier transform of $f_{\alpha, \beta}$. Their estimator involves a series estimator of a partial Fourier transform of  $f_{\alpha, \beta}$ in the first coordinate using the prolate spheroidal wave functions as basis. Thus, estimators of a form related to \eqref{eq:FCoeff} are required, to which our method can be applied.  Note that the flexibility in choosing a subset $[-x_0,x_0]^d$ of the support of $X$ in the method by \citet{gaillac2022adaptive} will make Assumption \ref{ass:boundeffect} feasible in this setting. 

\citet{hoderlein2017triangular} consider a triangular model with random coefficients. In case of additional exogenous covariates their semiparametric estimation method uses the simple estimator of the form \eqref{eq:simpleest} without higher-order bias correction to estimate the contrast function.  While they only give asymptotic analysis in case of $d=1$ for an estimator based on spacings, obtaining the parametric rate for the minimum contrast estimator with exogenous covariates requires an estimator of the contrast function with parametric rate as well. This could be achieved by the methods proposed in the present paper. Note that the support restriction in \citet{hoderlein2017triangular} will make Assumptions \ref{ass:deigndens}  and \ref{ass:boundeffect} applicable.   

%The estimates of the coefficients of this estimator are involved, but require division by the density $f_X$, and thus are of the form \eqref{eq:thefunctional}.      

In their analysis of model \eqref{eq:linearrandomcoef}, \citet{Hoderlein2010} normalize the covariates to the sphere (upper hemisphere) by setting
$ T = (1,X) /\|(1,X)\|$ and $U = Y / \|(1,X)\|$. Setting $(\alpha,\beta^\top ) =: \bar \beta^\top$ this leads to the equation
$ U = T^\top \bar \beta,$
in which $T$ and $\bar \beta^\top$ remain independent.  Their reconstruction formula for $f_{\bar \beta}$ is in terms of a Lebesgue integral over the sphere of a conditional expectation given $T=t$.
%
%A regularized approximation to $f_{\bar \beta}$ is given by
%
%\begin{equation}\label{eq:reginvrc}
%\bar f_{\bar \beta}(b;h) =  \int_{\mathcal S^d}\, \E\big[K_{h,p}(T^\top b - U) | T=t\big]\, \dd \mu(t)
%\end{equation}
%
%where $\mu$  is the Lebesgue measure on the $d$ -dimensional sphere $\mathcal S^d$, $h>0$ is a bandwidth parameter, and  $K_{h,d}$ is a kernel tailored to the estimation problem. See \citet{Hoderlein2010} for the formula. \eqref{eq:reginvrc} is of similar form to \eqref{eq:thefunctional}, but replacing the integral over  Euclidean space with an integral over the sphere. 
Similar expressions involving integrals of regression functions over the sphere occur in \citet{dunker2019}. For the binary choice model with random coefficients $Y = 1(\alpha + \beta^\top X \geq 0) = 1(T^\top \bar \beta \geq 0)$ \citet{Gautier2013} construct a series estimator, in which the Fourier coefficients are of the form \eqref{eq:FCoeff} but with $X$ replaced by the spherical random variable $T$. An extension of our estimation approach to cover integrals of regression functions on the sphere would therefore be of interest.

\subsection{Average treatment effects}\label{sec:averagetreat}

Let us discuss how the asymptotically unbiased matching estimator in Section \ref{sec:matchest} can be applied to estimate average treatment effects in the potential outcome framework. Our methods complement the seminal work of \citet{abadie2006large, abadie2011bias, imbens2004nonparametric} in several aspects. First we show that asymptotic unbiasedness can be achieved without involving nonparametric estimators and explicitly correcting for the bias. Further we quantify a precise finite amount of smoothness of the regression functions to achieve $\sqrt n$-consistent estimation, without assuming infinite smoothness as in \citet{abadie2011bias}. Moreover, our results on lower bounds show that some amount of smoothness of design densities or regression functions is required for $\sqrt n$-consistent estimation to be possible. Recently, \citet{Lin2023Econometrica}  further analyzed the bias-corrected estimator of \citet{abadie2011bias} in case of diverging $K$ and under milder smoothness assumptions (Assumptions 4.4 and 4.5 in \citet{Lin2023Econometrica}) which seem to be closer to but still stronger than those imposed in Theorem \ref{th:theoremparratematch}. Further, \citet{Lin2023Econometrica} relate to the recent literature on debiased/ doubly robust machine learning methods for treatment effects \citep{wang2020debiased, Chernozhukov2018, chernozhukov2024applied}. 

Consider the potential outcome framework: We have i.i.d.~observations $(Y_i,Z_i,D_i)$, $i=1, \ldots, n$, distributed as the generic $(Y,Z,D)$. Here $Y$ is the real-valued target quantity, $Z$ a covariate vector, and $D$ the binary treatment indicator. Suppose that the random variables $Y(0)$ and $Y(1)$ denote the potential outcomes, and that the observed outcome is generated as $Y = Y(D)$. 

First consider estimation of the average treatment effect on the treated (ATT)
$$\tau^t = \E\big[Y(1) - Y(0) \, | \, D=1\big].$$
Assuming that $Y(0)$ and $D$ are independent given $Z$ (conditional exogeneity,  mean independence suffices), and that $\Pr(p(Z) < 1) = 1$, where 
$p(Z) = \Pr(D=1 | Z)$ is the propensity score, the ATT is identified as \citep[Section 5.C]{chernozhukov2024applied}
$$\tau^t = \E\big[Y | D=1\big]  - \frac1{\pi}\, \E\Big[ {\bf 1}(D=0)\, Y \, \frac{p(Z)}{1-p(Z)}\Big] : = \tau^t_1 - \tau^t_0,$$
where $\pi = \Pr(D=1)$. Let $n_1 = \#\{i \mid D_i = 1\}$, $n_0 = n-n_1$. We estimate $\tau^t_1$ by the simple average $\hat \tau^t_1 = n_1^{-1}\, \sum_{i=1}^n Y_i {\bf 1}(D_i=1)$. 
%The first term can be estimated by an average over the observations for which $D_i=1$. 
Next assume that the conditional distribution of $Z$ given $D=i$ has Lebesgue density $f_i$, $i=0,1$. Then we can write
$$ \tau^t_0 = \E\Big[ Y \, \frac{f_1(Z)}{f_0(Z)} \big| D=0\Big].$$
We estimate this parameter by $\hat \tau_0^t$ with the estimator \eqref{eq:matchest}, using those $n_0$-observations $(Y_i,Z_i)$ for which $D_i=0$, and as sample from $f_1$ (which has the role of $f$ in \eqref{eq:estPhi}) we take those $Z_i$ for which $D_i = 1$. Finally set 
$$\hat \tau^t = \hat \tau_1^t - \hat \tau_0^t.$$ 
\begin{ass}\label{ass:avetot}
The density $f_0$ satisfies Assumption \ref{ass:deigndens} for $f_Z$. Further, the  support $\cS^*$ of $f_1$ is contained in the support $\cS$ as in Assumption \ref{ass:boundeffect}. Moreover, $f_1$ is bounded, and the conditional variance $\mbox{var}(Y | Z=z, D=0)$ is uniformly bounded.  
\end{ass}
 The support restriction in Assumption \ref{ass:avetot}, which guarantees that Assumption \ref{ass:boundeffect} applies, is plausible if only individuals with certain covariates may receive the treatment.
Under Assumption \ref{ass:avetot}, if the regression function $G_0(z) = \E[Y | Z=z, D=0]$ is contained in $ \cG(l, \beta, \CH, C_G)$, when choosing $K$ and $L$ as in Theorem \ref{th:theoremparratematch} in  $\hat \tau_0^t$, we obtain, conditionally on the treatment variables, that
$$  \E \big[\big(\hat \tau^t  -  \tau^t\big)^2 \big| \sigma_D \big] \, \lesssim \, n_1^{-1} + \,   n_0^{-1}  +   \, n_0^{-\frac{2(l + \beta)}{d}},
$$
where $\sigma_D = \sigma\{D_1, \ldots, D_n\}$. If 	$l + \beta \geq d/2$ we obtain a bound of order $n_1^{-1} + \,   n_0^{-1} $. If $n_0$ is large compared to $n_1$, $n_0^{-2(l + \beta)/d}$  might be small compared to $n_1$ even for lower smoothness of $G_0$. In an asymptotic analysis, this would require the distribution of $D$ to depend on $n$ with $\Pr(D=1)$ tending to zero, see Assumption 3' in \citet{abadie2006large}.  

\smallskip

Now let us turn briefly to the average treatment effect (ATE). Here we require conditional exogeneity for both $Y(1)$ and $Y(0)$, that is $Y(d)$ and $D$ are independent given $Z$, $d=0,1$, and that $\Pr(0 < p(Z) < 1) = 1$. Let $S^*$ be a subset of the support of the covariate $Z$. Then the average treatment effect over $S^*$ is
$$ \tau_{S^*} = \E\big[Y(1)\, {\bf 1}(Z \in S^*)\big] -  \E\big[Y(0)\, {\bf 1}(Z \in S^*)\big]$$
The average potential outcome of $Y(1)$ over $S^*$ is identified as
\begin{align}\label{eq:identpropens}
\E\big[Y(1)\,& {\bf 1}(Z \in S^*)\big]  = \E\Big[Y\, {\bf 1}(D=1) \, \frac{{\bf 1}(Z \in S^*)}{p(Z)}\Big] \nonumber\\
& = \E\big[Y\, {\bf 1}(D=1) \, {\bf 1}(Z \in S^*)\big] + (1-\pi)\, \E\Big[Y \, \frac{f_0(X)}{f_1(X)} \,  {\bf 1}(Z \in S^*)| \, D=1\Big],
\end{align}
with a similar expression for $\E\big[Y(0)\, {\bf 1}(Z \in S^*)\big]$. As for ATT we can apply the estimator \eqref{eq:matchest} to the second expression.  If Assumption \ref{ass:deigndens} holds for both $f_0$ and $f_1$ in place of $f_Z$, and if Assumption \ref{ass:boundeffect} applies to $S^*$ relative to their support, then we may obtain the $\sqrt n$ - rate of convergence. 

Of course, in order to avoid Assumption \ref{ass:avetot} for the ATT, we may proceed similarly by considering
$$\tau^t_{S^*} = \E\big[{\bf 1}(Z \in S^*)\, (Y(1) - Y(0)) \, | \, D=1\big],$$
in which case we require that $S^*$ satisfies Assumption \ref{ass:boundeffect} relative to the support of $f_0$. 

%and under Assumption \ref{ass:avetot} for both $f_0$ and $f_1$ . 

%Denote $\pi = \Pr(D=1)$, assume that $X$ has conditional densities $f_d$ given $D=d$. 
%
%Then the unconditional density of $x$ is $f_X = (1-\pi)\, f_0 + \pi\, f_1$, and we can express the propensity score as $p(x) = \pi f_1 / f_X$. Inserting in \eqref{eq:identpropens} gives
%
%\begin{equation}\label{eq:identpropensone}
%\E\big[Y(1)] =   
%\end{equation}
%
%The conditional expectation in the second term on the right side can be estimated using our methods by setting $g(y,x) = y\, f_0(x)$ and using the sub-sample with $D=1$, in which $X$ has density $f_1$.

%Require knowledge of or an estimate of $f_0$. 

%\bigskip

\subsection{Reweighting in transfer learning under covariate shift}\label{sec:transferlearn}

In the transfer learning problem for classification, we have the target problem from a distribution $(X,Y)$ as well as the   source problem $(X^S,Y^S)$. Here $Y$ and $Y^S$ are binary variables, which are to  be predicted from the covariate $Z$ respectively $Z^S $. 
Interest focuses on the classification problem for $(X,Y)$, but labeled data are mainly or only available from the source problem $(Z^S_i, Y^S_i)$, $i=1, \ldots, n_S$, together with unlabeled data $Z_i$, $i=1, \ldots, n$ from the target sample, that is with distribution $\rP_Z$.  
Under covariate shift it is assumed that the conditional distributions $\rP_{Y | X}$ and $\rP_{Y^S | X^S}$ are equal, and only the marginal distributions of covariates, $\rP_X$ and $\rP_{X^S}$ differ. See \citet{kouw2018introduction, portier2023scalable, sugiyama2007direct}  

If $\ell(\hat y,y)$ is a loss function, the goal would be to train a classifier $h$ which minimizes the expected loss under the target distribution $\E[\ell(h(X),Y)]$, which is however not directly accessible due to insufficient  labeled data from the target classification problem. 
Suppose that the distributions $\rP_X$ and $\rP_{X^S}$  have Lebesgue densities  $f_X$ and $f_{X^S}$, respectively. Then using the equality of the conditional distributions $\rP_{Y | X}$ and $\rP_{Y^S | X^S}$ we have that the average loss for the target problem can be computed as a importance-weighted average of the source problem,
\begin{equation}\label{eq:translearn}
\E[\ell(h(X),Y)] = \E\Big[\frac{f_X(X^S)}{f_{X^S}(X^S)}\, \ell\big(h(X^S),Y^S\big)\Big],
\end{equation}
where the second expected value is over the source distribution, for which a labeled sample is available. \eqref{eq:translearn} is of the form \eqref{eq:estPhi}, and we can apply the matching estimator \eqref{eq:matchest}.

\section{Simulations}\label{sec:sims}

In a brief simulation study we illustrate the practical feasibility of our methods.  The R-files are available on \href{https://github.com/hajo-holzmann/matching_bias_correct}{Github}.  
First let us describe how to implement the estimators. Starting with $\hat \Phi$ in \eqref{eq:matchest}, first we compute and store the $K$-nearest neighbors in $Z_1, \ldots, Z_n$ for each of the variables in the additional sample $X_1, \ldots, X_m$. To this end we use the function \verb+kNN+ from the R-library \verb+dbscan+. Then for each $j$, using the $K$-nearest neighbors of $X_j$, we compute $\hat G(X_j)$ by \eqref{eq:estimagexp} with $Y_j$ in place of $g(U_j,Z_j)$, and finally average over the values $\hat G(X_j)$. Turning to $\hat \Psi$ in \eqref{eq:estfinal}, we simulate $x_1^*, \ldots, x_m^*$ from the uniform distribution on $S^*$ for large $m$, and then effectively compute the estimator $\hat \Phi$ in \eqref{eq:matchest} with the $x_j^*$ taking the role of observed $X_j$'s, and normalize by $\lambda(S^*)$.  

Now, in our simulations we consider regressors in $d=3$, and focus on the orders $L=0$ and $L=1$ for polynomial approximation. An additional scenario in $d=2$ is given in the supplementary material, Section \ref{sec:addsim}.  Computing the estimator once for $n=1000$ and $m=10000$ takes about 1 second on a computer with i7-10700 CPU, 2.90 GHz and 32GB RAM. 

We  choose the coordinates $Z_{i,1}, Z_{i,2}$ and $Z_{i,3}$ of regressors as independent and Beta-distributed with parameters $\alpha = \beta = 3$, and let $S^* = [0.2,0.8]^3$. 
The regression function is chosen as 
$$ f_1(z_1,z_2,z_3) = \exp(2\, \cos(7\, z_1)\, \sin(7\,z_2))\cdot (4-8\,(z_3-0.5)^2),$$
and as DGP we take $Y_i = f_1(Z_{i,1},Z_{i,2},Z_{i,3}) + 0.4\, \varepsilon_i$, with $\varepsilon_i$ independent standard normally distributed. We use $N=1000$ repetitions in each scenario, and sample sizes $n \in \{100,1000\}$ and $m=10000$ for the generated sample $x_1^*, \ldots, x_m^*$. The results are displayed in Table \ref{tab:scen3}. The true value in this setting is $\Psi_2 = 1.4176$. While, for the smaller sample size $n=100$, using $L=0$ is preferable, for $n=1000$ using $L=1$ leads to a notable reduction in the bias and MSE.  Let us also remark that overall, for given $L$, the asymptotic bias increases with the number $K$ of nearest neighbors used, as expected.

\begin{table}[]
\begin{tabular}{lll|llllll}
	$n=100$ & $L=0$ & $K$ & 1 &     3  &    6 &     8   &  10 &    12  \\ \hline
	& & $\sqrt n \cdot $ BIAS& 0.3610 & 0.5321 & 0.7031 & 0.7948 & 0.8683 & 0.9172  \\
	&& $\sqrt n \cdot $ STDV&  0.7247 & 0.8702 & 1.0417 & 1.1153 & 1.1824 & 1.2430 \\
	&& $\sqrt n \cdot $ RMSE&  {\bf 0.8096} & 1.0200 & 1.2568 & 1.3696 & 1.4670 & 1.5448  \\\hline
	& $L=1$ & $K$    &   5  &    6  &    8    & 10  &   12 & 14\\\hline
	&& $\sqrt n\cdot $ BIAS  &-0.3991 & -0.5366 & -0.6988 & -0.8336 & -0.9510 & -1.0609\\
	&& $\sqrt n\cdot $ STDV  &1.1092 & 0.6587 & 0.6849 & 0.7320 & 0.7802 & 0.8229 \\
	&& $\sqrt n\cdot $ RMSE  &1.1788 &{\bf 0.8496} & 0.9784 & 1.1094 & 1.2301 & 1.3427  \\\hline \hline
	$n=1000$ & $L=0$ & $K$ &  1 &     3  &    6 &     8   &  10 &    12  \\ \hline 
	& & $\sqrt n \cdot $ BIAS& 0.2826 & 0.4231 & 0.6272 & 0.7677 & 0.8481 & 0.9537    \\
	&& $\sqrt n \cdot $ STDV&  0.4856 & 0.5312 & 0.5899 & 0.6161 & 0.6746 & 0.6698\\
	&& $\sqrt n \cdot $ RMSE&  {\bf 0.5618} & 0.6791 & 0.8610 & 0.9843 & 1.0837 & 1.1655  \\\hline
	& $L=1$ & $K$ &   5  &    6  &    8    & 10  &   12 & 14\\\hline
	&& $\sqrt n\cdot $ BIAS  &-0.1325 & -0.1570 & -0.2010 & -0.2436 & -0.2888 & -0.3308\\
	&& $\sqrt n\cdot $ STDV  &0.4913 & 0.4180 & 0.4455 & 0.4373 & 0.4197 & 0.4499 \\
	&& $\sqrt n\cdot $ RMSE  &0.5088 & {\bf 0.4465} & 0.4887 & 0.5006 & 0.5095 & 0.5584
\end{tabular}
\caption{Simulation scenario with $f_1$ in $d=3$}\label{tab:scen3}
\end{table}

 For $n=1000$ we also numerically investigate asymptotic normality of the estimators. For $L=0$ and $K=1$ the Shapiro-Wilk test gives a p-value of $0.02$, while for $L=1$ and $K=6$ it is $0.20$. In the supplementary appendix, Section \ref{asno} we include density and QQ-plots for $\sqrt n (\hat \Psi - \Psi)$. Thus while our theoretical results do not include it, at least in this simulation setting asymptotic normality is plausible for $L=1$.  

 To investigate potential boundary effects of our method when choosing $S^*=S$, in the supplementary appendix (Section \ref{sec:addsim}), we extend this simulation setting to cover $S^* = [0,1]^3$, the full support of $Z$. The methods still perform reasonably well, both for $L=0$ and $L=1$.

\section{Concluding remarks}\label{sec:conclude}

In contrast to other proposals in literature our estimators for expected values weighted by the inverse of a multivariate density do not rely on nonparametric function estimators and therefore do not require the choice of data-dependent smoothing parameters. In our analysis we focus on risk bounds and the amount of smoothness of the regression functions required for the parametric rate. A further novel contribution, previously not available in the literature, are the  lower bounds which we provide and which show that, for unknown design density, some smoothness is required for the parametric rate to be attainable. Our analysis is thus complementary to much of the literature which focuses on asymptotic normality, testing and efficiency issues.

Assumption \ref{ass:deigndens} is sometimes called the strong density assumption in the statistics literature \citep{Audibert2007}.  For $d=1$ \citet{holzmann2020rate} consider the effect on the minimax rates of a design density which tends to zero on the boundary of the support in the nonparametric problem of estimating the density in a linear random coefficients model. Their method could be adapted to study the simpler problem of estimating the parametric functionals  $\Psi$ and $\Phi$ for $d=1$ under such weaker design assumptions.

Concerning Assumption \ref{ass:boundeffect} it would of course be desirable to allow for $S^*=S$, as is feasible for $d=1$. Indeed, without any polynomial approximation, that is $L=0$, Lemma \ref{L:u0bound} is not required. An inspection of the proof of Theorem \ref{th:theoremparrate} indicates that the result remains true for $S^* = S$ if $S$ is such that, for some $c>0$, it holds that $\lambda(S \cap B_{d}(z,r)) \geq c r^d$ for all $z \in S$ and all sufficiently small $r>0$. However, an extension of the analysis to  $S^* = S$ in case of $L \geq 1$ comes along with additional technical difficulties and is left for future research. 

Asymptotic normality of the estimator would also be of interest. We provide a brief discussion for $L=K=0$ in Appendix D in the supplementary material, but have to leave the more general setting for future research.

A further extension of some interest would be to estimate  functionals $\Psi$ in which the integral is over more abstract compact manifolds, such as the unit sphere equipped with spherical Lebesgue measure, with the polynomial approximation being taken locally in the tangent space.

\section{Proofs}\label{sec:proofs}

\subsection{Proof of Theorem \ref{th:theoremparrate}}\label{sec:proofthe1}

We use the notation introduced at the beginning of Section \ref{sec:estmainresult}. 
%Organization of the proofs. \hajo{noch}
Recall the estimator \eqref{eq:estfinal}, 
$$ \hat{\Psi} \, = \, \sum_{J\in {\cal J}_K} \int_{C(J)} \, \widehat{G}(z) \, \dd z\,, $$
where the Voronoi cells $C(J)$ are given in \eqref{eq:voronoihighorder}, and $\widehat{G}(z)$ in \eqref{eq:estimagexp}. 
%
%Note that for $L=0$ we have $K^*=1$, $\xi(z,w) =1$, ${\cal M}(z) = K$, and we obtain \eqref{eq:estlto0}.
%
%\begin{equation}\label{eq:estlto0}
%\hat{G}(z) \, = \, \frac1K\, \sum_{j\in J(z)} g(U_{j},Z_{j}), \qquad \hat{\Psi} \, = \, \sum_{J\in {\cal J}_K} \, \Big(  \frac1K\,\sum_{j\in J} g(U_{j},Z_{j})\Big)\,\lambda(C(J))\,. 
%\end{equation}
%
%
%Let  ${\cal I}$ denote the ${\cal K}\times {\cal K}$ identity matrix.
%
Again consider the expected conditional variance  and expected  squared conditional bias decomposition of the mean squared error \eqref{eq:condvarbias}, 
\begin{align*}
\E \big[\big|\hat{\Psi} - \Psi\big|^2\big] &  \, = \, \E \, \big[\mbox{var}\big(\hat{\Psi} | \sigma_Z\big)\big] \, + \, \E \big[\big|\E(\hat{\Psi} | \sigma_Z) - \Psi\big|^2\big] \,,
\end{align*}
where $\sigma_Z$ is the $\sigma$-field generated by $Z_1, \ldots, Z_n$.
%where . . As a procedure to estimate $\Psi$ we propose

We provide the main steps in the proof of the bound \eqref{eq:boundsth} in Theorem \ref{th:theoremparrate} for the expected conditional variance term in Section \ref{sec:expcondvar}, and for the expected squared conditional bias in Section \ref{sec:squaredcondbias}. Section \ref{sec:technicaldetails} in the supplement contains further technical details.   

\subsubsection{Expected conditional variance in \eqref{eq:boundsth}}\label{sec:expcondvar}

We show the first inequality in \eqref{eq:boundsth}, and proceed in the three steps. 

In Step 1 we make use of conditional independence of the $g(U_1,Z_1), \ldots, g(U_n,Z_n)$ given $\sigma_Z$, of the independence of two random Voronoi cells $C(J)$ and $C(J')$ for $J \cap J' = \emptyset$, as well as of the assumption of a bounded conditional variance $\mbox{var}(g(U,Z) | Z)$. 

In Step 2 we derive a tight bound on $\mathbb{E}\,[ \lambda(C(J)\big) \,  \lambda\big(C(J')) ]$ for $J \cap J' \not = \emptyset$, and conclude in the case $L=0$.

Finally, in Step 3 we use a bound on the conditional moments of $({\bf e}_0^\top {\cal M}_{J}(z)^{-1} {\bf e}_0)$ from Lemma \ref{L:u0bound}, proven in the supplement (Section \ref{sec:technicaldetails}), to extend the bound to $L \geq 1$.
Here, the invertability of ${\cal M}(z)$ almost surely for sufficiently large $K$ follows from Lemma \ref{L:01} in Section \ref{sec:technicaldetails} in the supplementary appendix. 

\medskip

\textit{Step 1: Using conditional independence}

\smallskip

Note that if $J(z)$ and $J(z')$ are disjoint, the estimators $\widehat{G}(z)$ and $\widehat{G}(z')$ are independent conditionally on $\sigma_Z$. Therefore
%
%For the expected conditional variance we deduce that
\begin{align} \nonumber
& \mathbb{E} \, \big[\mbox{var}\big(\hat{\Psi} | \sigma_Z\big)\big] \\  \nonumber&  =  \mathbb{E} \Big[\sum_{J,J'\in {\cal J}_K} {\bf 1}\{J\cap J' \neq \emptyset\} \int_{C(J)} \int_{C(J')} \mbox{cov}\big( \widehat{G}(z) , \widehat{G}(z') \mid \sigma_Z\big)\, \dd z \, \dd z'\Big] \\  
&  \leq \sum_{J,J'\in {\cal J}_K} {\bf 1}\{J\cap J' \neq \emptyset\} \, \mathbb{E} \Big[\int_{C(J)} \big\{\mbox{var}\big(\widehat{G}(z) | \sigma_Z\big)\big\}^{1/2} \dd z  \cdot  \int_{C(J')} \big\{\mbox{var}\big(\widehat{G}(z') | \sigma_Z\big)\big\}^{1/2} \dd z'\Big] \label{eq:var.0}
%&  =  \sum_{J,J'\in {\cal J}_K} {\bf 1}\{J\cap J' \neq \emptyset\} \, \mathbb{E} \int_{C(J)} \int_{C(J')}  \mathbb{E}\big(\hat{G}^2(z) | \sigma_Z\big)^{1/2} \cdot \mathbb{E}\big(\hat{G}^2(z') | \sigma_Z\big)^{1/2} \, dz \, dz'\,,
\end{align}
by applying the Cauchy-Schwarz inequality conditionally on $\sigma_Z$ in the second step.
By conditional independence of the $g(U_{j},Z_{j})$, $j=1, \ldots, n$ given $\sigma_Z$, observing \eqref{eq:estimagexp} for all $z \in S^*$ we obtain
\begin{align}
\mbox{var}\big(\widehat{G}(z) | \sigma_Z\big) & \, = \, \sum_{j\in J(z)} \mbox{var}\big(g(U_{j},Z_{j}) | \sigma_Z\big) \cdot {\bf e}_0^\top \, {\cal M}(z)^{-1} \, \xi(z,Z_{j}) \, \xi(z,Z_{j})^\top \,{\cal M}(z)^{-1} \, {\bf e}_0 \nonumber\\ 
& \, \leq \,  \, \CV \cdot {\bf e}_0^\top {\cal M}(z)^{-1} {\bf e}_0\,. \label{eq:boundcondvar}
\end{align}

%\begin{align*}
%\mathbb{E}&\big(\hat{G}^2(z) | \sigma_Z\big) \, = \, \mathbb{E} \big(\big|{\bf e}_0^\top {\cal M}(z)^{-1} \hat{\mathcal{G}}(z)\big|^2 \mid \sigma_Z\big) \\
%& \, = \, \sum_{j,j'\in J(z)} \mathbb{E}\big(g(U_j,Z_j) g(U_{j'},Z_{j'}) | \sigma_Z\big) \, \big({\bf e}_0^\top {\cal M}(z)^{-1} \{\xi_{\kappa}(z,Z_j)\}_{\kappa\in {\cal K}}\big) \cdot \big({\bf e}_0^\top {\cal M}(z)^{-1} \{\xi_{\kappa}(z,Z_{j'})\}_{\kappa\in {\cal K}}\big) \\
%& \, \leq \, \Big(\sum_{j\in J(z)} G_2^{1/2}(Z_j) \, \big|{\bf e}_0^\top {\cal M}(z)^{-1} \{\xi_{\kappa}(z,Z_{j})\}_{\kappa\in {\cal K}}\big|\Big)^2 \\
%& \, \leq \, K \cdot \|G_2\|_\infty \cdot {\bf e}_0^\top {\cal M}(z)^{-1} {\bf e}_0\,,
%\end{align*}

%for all $z \in S^*$, where
%$$ G_2(z) \, := \, \int g^2(u,z) f_{U|Z}(u|z) du\,. $$

Therefore,  \eqref{eq:var.0} can be upper bounded by
\begin{align} \label{eq:var.1}
& \mathbb{E} \, \big[\mbox{var}\big(\hat{\Psi} | \sigma_Z\big)\big]  
\leq  \, \CV\, \sum_{J,J'\in {\cal J}_K} {\bf 1}\{J\cap J' \neq \emptyset\} \, \iint_{(S^*)^2} \mathbb{E}\,\Big[\, {\bf 1}_{C(J)}(z) \, {\bf 1}_{C(J')}(z') \, \nonumber\\
& \hspace{5cm} \cdot \big({\bf e}_0^\top {\cal M}(z)^{-1} {\bf e}_0\big)^{1/2}  \big({\bf e}_0^\top {\cal M}(z')^{-1} {\bf e}_0\big)^{1/2} \Big] \, \dd z \, \dd z'.
\end{align}
To deal with the inner expected value, for $J\in {\cal J}_K$ and $z \in S^*$ let us introduce
\begin{equation}\label{eq:mjz}
{\cal M}_J(z) \, := \, \sum_{j\in J} \xi(z,Z_j)\, \xi(z,Z_j)^\top = \Big(\sum_{j\in J} \xi_{\kappa+\kappa'}(z,Z_j)\Big)_{\kappa,\kappa'\in {\cal K}}\,, 
\end{equation}
so that ${\cal M}_{J(z)}(z) = {\cal M}(z)$. 
Then for all $z,z' \in S^*$ and $J,J'\in {\cal J}_K$, since ${\bf 1}_{C(J)}(z)=1$ if and only if $J = J(z)$ we have that
\begin{align} \nonumber
\mathbb{E} & \, \big[ {\bf 1}_{C(J)}(z) \, {\bf 1}_{C(J')}(z') \, \big({\bf e}_0^\top {\cal M}(z)^{-1} {\bf e}_0\big)^{1/2}  \big({\bf e}_0^\top {\cal M}(z')^{-1} {\bf e}_0\big)^{1/2}\big] \\ \nonumber
\, = \, & \mathbb{E}  \, \big[ {\bf 1}_{C(J)}(z) \, {\bf 1}_{C(J')}(z') \, \big({\bf e}_0^\top {\cal M}_J(z)^{-1} {\bf e}_0\big)^{1/2}  \big({\bf e}_0^\top {\cal M}_{J'}(z')^{-1} {\bf e}_0\big)^{1/2}\big] \\ \label{eq:var.101}
\, = \, & \mathbb{E}\,\Big[ \big({\bf e}_0^\top {\cal M}_J(z)^{-1} {\bf e}_0\big)^{1/2} \,  \big({\bf e}_0^\top {\cal M}_{J'}(z')^{-1} {\bf e}_0\big)^{1/2} \, \mathbb{P}\big(z \in C(J), \, z'\in C(J') \mid Z_k \, , \, k\in J\cup J'\big)\, \Big]\,.
\end{align}
Further,
\begin{align} \nonumber
\mathbb{P}&\big(z \in C(J), \, z'\in C(J') \mid Z_k \, , \, k\in J\cup J'\big) \\ \nonumber & \, = \, \mathbb{P}\big(\|z-Z_j\| < \|z-Z_k\|\, , \, \|z'-Z_{j'}\| < \|z'-Z_{k'}\| \, , \, \forall j\in J,j'\in J',\\ \nonumber
& \hspace{6cm} k\not\in J, k'\not\in J' \mid  Z_k \, , \, k\in J\cup J'\big) \\ \nonumber
& \, \leq \, \mathbb{P}\big(\|z-Z_j\| < \|z-Z_k\|\, , \, \|z'-Z_{j'}\| < \|z'-Z_{k'}\| \, , \, \forall j\in J,j'\in J',\\\nonumber
& \hspace{6cm} k,k'\not\in J\cup J' \mid  Z_k \, , \, k\in J\cup J'\big) \\ \label{eq:var.3}
& \, = \, \Big(1 - \mathbb{P}_Z\big(B_d\big(z,\max_{j\in J}\|z-Z_j\|\big) \cup B_d\big(z',\max_{j\in J'}\|z'-Z_j\|\big)\big)\Big)^{n - \# J\cup J'}\,,
\end{align}
where we write $\mathbb{P}_Z$ for the image measure of $Z_1$, and $B_d(x,r)$ denotes the $d$-dimensional Euclidean ball around $x$ with the radius $r$. For fixed $z,z'\in S^*$ and $J,J'\in {\cal J}_K$, put
\begin{equation}\label{eq:alphadef1}
\alpha \, := \, \alpha(z,z',J,J') \, := \,\max\big\{\max_{j\in J}\|z-Z_j\| \, , \, \max_{j\in J'}\|z'-Z_j\|\big\}\,. 
\end{equation} 
Then using Assumptions \ref{ass:deigndens} and \ref{ass:boundeffect}, \eqref{eq:var.3} is bounded from above by
$$ \exp\big(-(n-\#J\cup J') \cdot \rho \cdot \pi^{d/2} \cdot \min\{\rho',\alpha\}^d / \Gamma(d/2+1)\big)\,, $$
and using the Cauchy-Schwarz inequality \eqref{eq:var.101} is smaller or equal to
\begin{align} \nonumber &
\Big( \mathbb{E}\,\Big[ \big({\bf e}_0^\top {\cal M}_J(z)^{-1} {\bf e}_0\big) \, \exp\big(-(n-\#J\cup J') \cdot \frac{\rho  \pi^{d/2}}{\Gamma(d/2+1)} \cdot \min\{\rho',\alpha\}^d \big)\Big]\Big)^{1/2} \\ \label{eq:var.4} & \cdot \Big( \mathbb{E}\, \Big[\big({\bf e}_0^\top {\cal M}_{J'}(z')^{-1} {\bf e}_0\big) \, \exp\big(-(n-\#J\cup J') \cdot \frac{\rho  \pi^{d/2}}{\Gamma(d/2+1)} \cdot \min\{\rho',\alpha\}^d \big)\,\Big] \Big)^{1/2}\,.
\end{align}

To further bound \eqref{eq:var.4} we make use of the fact that the non-zero terms in \eqref{eq:var.1} have $J \cap J' \not= \emptyset$. If $J$ and $J'$ are not disjoint there exists some $j\in J\cap J'$ so that
$$ \alpha \geq (\|z-Z_j\|+\|z'-Z_j\|)/2 \geq \|z-z'\|/2\,. $$
On the other hand,
$$ \alpha \, \geq \, \max\big\{\max_{j\in J}\|z-Z_j\| \, , \, \max_{j\in J'}\|z-Z_j\| - \|z-z'\|\big\} \, \geq \,  \max_{j\in J\cup J'} \|z-Z_j\| \, - \, \|z-z'\|\,, $$
so that
\begin{align} \alpha & \, \geq \,  \, \max\big\{\|z-z'\|/2 , \max_{j\in J\cup J'} \|z-Z_j\| \, - \, \|z-z'\|\big\} \nonumber \\
& \, \geq \, \max\big\{\|z-z'\|/2 , \max_{j\in J\cup J'} \|z-Z_j\| / 3\} \nonumber \\
& \, =: \, \alpha^*\, = \alpha^*(z,z',J,J'). \label{eq:alphadef}
\end{align}
Let us summarize the bounds obtained so far in the following lemma. 
\begin{lemma}\label{lem:firststepvarest}
Under the assumptions of Theorem \ref{th:theoremparrate}, setting ${\cal M}_J(z)$ as in \eqref{eq:mjz}, $\alpha^* = \alpha^*(z,z',J,J')$ as in \eqref{eq:alphadef} and $C_\alpha = \rho \cdot \pi^{d/2}/ \Gamma(d/2+1)$ we have the bound
\begin{align} \nonumber
	& \mathbb{E} \, \big[\mbox{var}\big(\hat{\Psi} | \sigma_Z\big)\big] 
	\leq  \, \CV\, \sum_{J,J'\in {\cal J}_K} {\bf 1}\{J\cap J' \neq \emptyset\} \, \iint_{(S^*)^2} \, \Big( \mathbb{E}\,\Big[ \big({\bf e}_0^\top {\cal M}_J(z)^{-1} {\bf e}_0\big)\, \nonumber \\
	& \, \hspace{1cm} \exp\big(-(n-\#J\cup J') \cdot C_\alpha  \cdot \min\{\rho',\alpha^*\}^d  \big)\Big]\Big)^{1/2} \cdot \Big( \mathbb{E}\,\Big[ \big({\bf e}_0^\top {\cal M}_{J'}(z')^{-1} {\bf e}_0\big)\nonumber\\
	&  \, \hspace{3cm} \exp\big(-(n-\#J\cup J') \cdot C_\alpha \cdot \min\{\rho',\alpha^*\}^d  \big)\Big]\Big)^{1/2} \dd z \, \dd z'.\label{eq:boundvarfirst}
	%&  =  \sum_{J,J'\in {\cal J}_K} {\bf 1}\{J\cap J' \neq \emptyset\} \, \mathbb{E} \int_{C(J)} \int_{C(J')}  \mathbb{E}\big(\hat{G}^2(z) | \sigma_Z\big)^{1/2} \cdot \mathbb{E}\big(\hat{G}^2(z') | \sigma_Z\big)^{1/2} \, dz \, dz'\,,
\end{align}
Moreover, for $J,J'\in {\cal J}_K$ with $J\cap J' \neq \emptyset$, 
\begin{align} \label{eq:boundexpecvorcell}
	\mathbb{E}\,\Big[ \lambda\big(C(J)\big) \,  \lambda\big(C(J')\big) \Big] \leq \iint_{(S^*)^2} \,  \mathbb{E}\,\Big[   \exp\big(-(n-\#J\cup J') \cdot C_\alpha \cdot \min\{\rho',\alpha^*\}^d  \big)\Big]\, \dd z \, \dd z'.
	%&  =  \sum_{J,J'\in {\cal J}_K} {\bf 1}\{J\cap J' \neq \emptyset\} \, \mathbb{E} \int_{C(J)} \int_{C(J')}  \mathbb{E}\big(\hat{G}^2(z) | \sigma_Z\big)^{1/2} \cdot \mathbb{E}\big(\hat{G}^2(z') | \sigma_Z\big)^{1/2} \, dz \, dz'\,,
\end{align}
\end{lemma}

\noindent Note that, in order to get (\ref{eq:boundexpecvorcell}), we consider the integrals on the right side of (\ref{eq:var.1}) without the term $\big({\bf e}_0^\top {\cal M}(z)^{-1} {\bf e}_0\big)^{1/2}  \big({\bf e}_0^\top {\cal M}(z')^{-1} {\bf e}_0\big)^{1/2}$ and apply the subsequent bounding techniques analogously as for (\ref{eq:boundvarfirst}).

\medskip

\textit{Step 2: Concluding for $L=0$.}

\smallskip

For $L=0$ we can drop ${\bf e}_0^\top {\cal M}_{J}(z)^{-1} {\bf e}_0 = 1/K$ from the expected values, so that we can insert \eqref{eq:boundexpecvorcell} into \eqref{eq:var.1}.
%\eqref{eq:boundvarfirst}.

Writing
$$n^* := (n-2K) \, C_\alpha \, \asymp \, n$$
and using the inequality $\# J\cup J' \leq 2K$, we have that 
\begin{align} & \mathbb{E}\,\Big[   \exp\big(-(n-\#J\cup J') \cdot C_\alpha \cdot \min\{\rho',\alpha^*\}^d  \big)\Big] 
\,\nonumber\\
\leq \, &  \exp\big(-n^*\cdot (\rho')^d\big) \, + \, \mathbb{E}\,\big[ \exp\big(-n^* \cdot (\alpha^*)^d\big)\big]\,.\label{eq:boundhelpag1}
\end{align}
Since the first term decreases exponentially in $n$, we focus on the second.
It holds that
\begin{align*}
& \mathbb{E}\,\big[ \exp\big(-n^* \cdot (\alpha^*)^d\big)\big] \\
\, = \, & \int_0^1 \mathbb{P}\big[\alpha^* < \big(- (\log t) / n^*\big)^{1/d}\big]\, \dd t \,\\
\, = \,&  \int_0^\infty \mathbb{P}\big[\alpha^* < (s/n^*)^{1/d} \big]\, \exp(-s)\, \dd s \\
\, = \, & \int_{s> n^* \|z-z'\|^d / 2^d} \mathbb{P}\big[\|z-Z_j\| < 3 (s/n^*)^{1/d} \, , \, \forall j\in J\cup J'\big] \, \exp(-s)\, \dd s \\
\, = \,&  \int_{s> n^* \|z-z'\|^d / 2^d} \mathbb{P}_Z\big(B_d(z,3 (s/n^*)^{1/d})\big)^{\# J\cup J'} \, \exp(-s)\, \dd s \\
\, \leq \,& \big\{\overline{\rho} \pi^{d/2} 3^d / \big( n^* \Gamma(d/2+1)\big)\big\}^{\# J\cup J'} \, \int_{s> n^* \|z-z'\|^d / 2^d} \, s^{\# J\cup J'} \, \exp(-s)\, \dd s\,.
\end{align*}
As $\max\{1,s^{2K}\} \exp(-s) \, \leq \, (4K)^{2K} \exp(-s/2)$ for all $s>0$, it follows that
\begin{align*}
\int_{s> n^* \|z-z'\|^d / 2^d} \, s^{\# J\cup J'} \, \exp(-s)\, \dd s & \, \leq \, (4K)^{2K} \, \int_{s> n^* \|z-z'\|^d / 2^d} \, \exp(-s/2)\, \dd s \\ & \, = \, 2 (4K)^{2K} \cdot \exp\big(- n^* \|z-z'\|^d/ 2^{d+1}\big)\,.
\end{align*}
Moreover, 
\begin{align*} \iint_{(S^*)^2} \exp\big(-n^* \|z-z'\|^d/ 2^{d+1}\big) \, \dd z \, \dd z' & \, \leq \, \iint_{S^*\times \mathbb{R}^d}  \exp\big(-n^* \|w\|^d/ 2^{d+1}\big) \, \dd w \, \dd z' \\ & \, = \, \lambda_d(S^*) \cdot \int \exp\big(-\|w\|^d / 2^{d+1}\big) \, \dd w\, / \, n^*.
\end{align*}
Observing \eqref{eq:boundexpecvorcell} in Lemma \ref{lem:firststepvarest}, we obtain the first part of the following lemma.
\begin{lemma}\label{lem:vorcellorder}
Under the assumptions of Theorem \ref{th:theoremparrate}, for $J,J'\in {\cal J}_K$ with $J\cap J' \neq \emptyset$, 
\begin{align} \label{eq:boundexpecvorcell1}
	\mathbb{E}\,\Big[ \lambda\big(C(J)\big) \,  \lambda\big(C(J')\big) \Big] & \leq \iint_{(S^*)^2} \,  \mathbb{E}\,\Big[   \exp\big(-(n-\#J\cup J') \cdot C_\alpha \cdot \min\{\rho',\alpha^*\}^d  \big)\Big]\, \dd z \, \dd z'\nonumber\\
	&\leq \mbox{const.} \cdot n^{-\# J\cup J' - 1},
	%&  =  \sum_{J,J'\in {\cal J}_K} {\bf 1}\{J\cap J' \neq \emptyset\} \, \mathbb{E} \int_{C(J)} \int_{C(J')}  \mathbb{E}\big(\hat{G}^2(z) | \sigma_Z\big)^{1/2} \cdot \mathbb{E}\big(\hat{G}^2(z') | \sigma_Z\big)^{1/2} \, dz \, dz'\,,
\end{align}
where the constant only depends on $K$, $d$, $\rho$, $\overline{\rho}$, $\rho'$ and $\rho''$.

In particular, for $L=0$ we obtain 
$$ \mathbb{E} \, \big[\mbox{var}\big(\hat{\Psi} | \sigma_Z\big)\big]  \lesssim n^{-1}.$$
\end{lemma}
For the second part, using Lemma \ref{lem:firststepvarest} and the previous bound \eqref{eq:boundexpecvorcell1} we obtain
\begin{align*} \nonumber
\mathbb{E} \, \big[\mbox{var}\big(\hat{\Psi} | \sigma_Z\big)\big] 
\leq  \mbox{const.} \, \CV\,  n^{-1}\, \cdot \sum_{J,J'\in {\cal J}_K} {\bf 1}\{J\cap J' \neq \emptyset\} \, n^{-\# J\cup J'}  
%&  =  \sum_{J,J'\in {\cal J}_K} {\bf 1}\{J\cap J' \neq \emptyset\} \, \mathbb{E} \int_{C(J)} \int_{C(J')}  \mathbb{E}\big(\hat{G}^2(z) | \sigma_Z\big)^{1/2} \cdot \mathbb{E}\big(\hat{G}^2(z') | \sigma_Z\big)^{1/2} \, dz \, dz'\,,
\end{align*}
Results from combinatorics provide that ${\cal J}_K$ contains exactly ${n\choose K}$ elements; and that, for each $J\in {\cal J}_K$, there exist exactly ${K \choose \ell} {n-K \choose K-\ell}$ sets $J' \in {\cal J}_K$ with $\# (J\cap J') = \ell$. Therefore,
\begin{align}  
&\, \sum_{J,J'\in {\cal J}_K} {\bf 1}\{J\cap J' \neq \emptyset\} \, n^{-\# J\cup J'} =   \, \sum_{\ell=1}^K \sum_{J,J'\in {\cal J}_K}  {\bf 1}\{\# J\cap J'=\ell\} \, n^{\ell-2K} \nonumber\\
\, \leq & \, \frac{n^K}{K!} \sum_{\ell=1}^K {K \choose \ell} \frac1{(K-l)!} \cdot n^{K-\ell+\ell-2K} \, \leq \, (2^K-1)/K!\,,\label{eq:sumbinom}
\end{align}
which proves the second part of the lemma. 

%Therefore (\ref{eq:ff}) admits the upper bound
%$$ \mbox{const.} \cdot n^{-\# J\cup J'/2} \cdot\exp\big(- n^* \|z-z'\|^d/ 2^{d+2}\big)\,, $$
%where the constant only depends on $K$, $d$, $L$, $\rho$, $\overline{\rho}$, $\rho'$ and $\rho''$. In the sequel, $\mbox{const.}$ may change its value from line to line. The second factor in (\ref{eq:var.4}) can be treated analogously by exchanging the roles of $J$ and $J'$ as well as those of $z$ and $z'$. Finally, (\ref{eq:var.101}) is bounded from above by

%$$ \mbox{const.} \cdot  n^{-\# J\cup J'} \cdot  \exp\big(- n^* \|z-z'\|^d/ 2^{d+1}\big)\,, $$
%and (\ref{eq:var.1}) is smaller or equal to

%$$ \mbox{const.}\cdot \|G_2\|_\infty \, \sum_{\ell=1}^K \sum_{J,J'\in {\cal J}_K}  {\bf 1}\{\# J\cap J'=\ell\} \, n^{\ell-2K} \cdot \iint_{(S^*)^2} \exp\big(-n^* \|z-z'\|^d/ 2^{d+1}\big) \, dz \, dz'
%$$
%Moreover consider that
%\begin{align*} \iint_{(S^*)^2} \exp\big(-n^* \|z-z'\|^d/ 2^{d+1}\big) \, dz \, dz' & \, \leq \, \iint_{S^*\times \mathbb{R}^d}  \exp\big(-n^* \|w\|^d/ 2^{d+1}\big) \, dw \, dz' \\ & \, = \, \lambda_d(S^*) \cdot \int \exp\big(-\|w\|^d / 2^{d+1}\big) \, dw\, / \, n^*,
%\end{align*}

%\newpage

\medskip

\textit{Step 3: Reducing the case $L \geq 1$ to $L=0$}

\smallskip

We show that for $L \geq 1$, 
	\begin{align} & \mathbb{E}\,\Big[ \big({\bf e}_0^\top {\cal M}_{J}(z)^{-1} {\bf e}_0\big) \, \exp\big(-(n-\#J\cup J') \cdot C_\alpha \cdot \min\{\rho',\alpha^*\}^d  \big)\Big] \nonumber \\
		\, \leq \, & C_3\, \mathbb{E}\,\Big[  \exp\big(-(n-\#J\cup J') \cdot C_\alpha \cdot \min\{\rho',\alpha^*\}^d  \big)\Big], \label{eq:boundgencase}
	\end{align}
	for some constant $C_3 >0$. Then, by symmetry of \eqref{eq:boundvarfirst} in $z$, $J$ and $z'$ and $J'$ we obtain that the argument that we give here also applies to $L \geq 1$. 

We need to prove \eqref{eq:boundgencase}, then the calculation from the second step for $L=0$ applies to $L \geq 1$ as well. 
This requires a quite sophisticated conditioning argument.

%Using the equality (\ref{eq:ff1}), Lemma \ref{L:u0bound} and the inequality $\# J\cup J' \leq 2K$, the term (\ref{eq:ff}) obeys the upper bound
%\begin{align*} C_3^{1/2} & \cdot  \Big\{ \mathbb{E}\, \exp\big\{-(n-2K) \cdot \rho \cdot \pi^{d/2} \cdot \min\{\rho',\alpha^*\}^d / \Gamma(d/2+1)\big\}\Big\}^{1/2} \\
% & \, \leq \, C_3^{1/2} \cdot  \big\{\exp\big(-n^*\cdot (\rho')^d\big) \, + \, \mathbb{E}\, \exp\big(-n^* \cdot (\alpha^*)^d\big)\big\}^{1/2}\,.
% \end{align*}
%for $\eta=1$.

By $\widehat{j}$ we denote the smallest $j\in J$ such that $\|z-Z_j\| \geq \|z-Z_k\|$ for all $k\in J$; moreover we define 
$$\alpha' := \max_{j \in J'\backslash J} \|z-Z_j\|.$$ 
The $\sigma$-fields generated by $\widehat{j}$, $Z_{\widehat{j}}$, $\alpha'$, on the one hand, and by $\widehat{j}$, $Z_{\widehat{j}}$, on the other hand, are called $\mathfrak{A}'$ and $\mathfrak{A}$, respectively. 
Note that with this notation we can write $\alpha^*$ in \eqref{eq:alphadef} as 
$$\alpha^* = \max\big\{\max\{\|z-Z_{\widehat{j}}\|,\alpha'\}/3 \, , \, \|z-z'\|/2\big\}\,, $$
and that the random variable $\alpha^*$ is measurable with respect to $\mathfrak{A}'$. 
Then  the first factor in (\ref{eq:var.4}) has the upper bound
\begin{align} \label{eq:ff}
& \mathbb{E}\,\Big[ \big({\bf e}_0^\top {\cal M}_{J}(z)^{-1} {\bf e}_0\big) \, \exp\big(-(n-\#J\cup J') \cdot C_\alpha \cdot \min\{\rho',\alpha^*\}^d  \big)\Big]\nonumber \\
\, = \, & \mathbb{E}\,\Big[ \mathbb{E}\,\big[\big({\bf e}_0^\top {\cal M}_{J}(z)^{-1} {\bf e}_0\big) \mid \mathfrak{A}'\big] \, \exp\big(-(n-\#J\cup J') \cdot C_\alpha \cdot \min\{\rho',\alpha^*\}^d  \big)\Big]\nonumber \\
\, = \, & \mathbb{E}\,\Big[ \mathbb{E}\,\big[\big({\bf e}_0^\top {\cal M}_{J}(z)^{-1} {\bf e}_0\big) \mid \mathfrak{A}\big] \, \exp\big(-(n-\#J\cup J') \cdot C_\alpha \cdot \min\{\rho',\alpha^*\}^d  \big)\Big],
%
%
%\Big\{ \mathbb{E}\, \mathbb{E}\big({\bf e}_0^\top {\cal M}_J(z)^{-1} {\bf e}_0 \mid \mathfrak{A}'\big) \, \exp\big\{-(n-\#J\cup J') \cdot \rho \cdot \pi^{d/2} \cdot \min\{\rho',\alpha^*\}^d / \Gamma(d/2+1)\big\}\Big\}^{1/2}\,.
\end{align}
where in the last step we used that the random vector $\big({\bf e}_0^\top {\cal M}_J(z) {\bf e}_0,\widehat{j},Z_{\widehat{j}}\big)$ and $\alpha'$ are independent. 

\begin{lemma} \label{L:u0bound}
Under the assumptions of Theorem \ref{th:theoremparrate}, given $L\geq 1$ and $\eta\geq 1$ for $K \geq 1 + (\lfloor 2\eta D\rfloor+1)K^*$, the matrix ${\cal M}_J(z)$ is invertible and we have that 
$$ \mathbb{E}\,\big[\big({\bf e}_0^\top {\cal M}_{J}(z)^{-1} {\bf e}_0\big)^\eta \mid \mathfrak{A}\big] \, \leq \, C_3\,, $$
almost surely, where the deterministic constant $C_3 \in (0,\infty)$ only depends on $\eta$, $d$, $L$, $\rho$, $\overline{\rho}$, $\rho'$, $\rho''$. 
\end{lemma}

The proof of the Lemma is deferred to the supplement (Section \ref{sec:technicaldetails}). This concludes the proof of the first inequality in \eqref{eq:boundsth}. \hfill $\qed$

%where $\lambda_d$ stands for the $d$-dimensional Lebesgue measure. Finally we arrive at the following upper bound on the expected conditional variance, which reflects the usual parametric convergence rates.
%\begin{proposition} \label{P:1}
%Grant that $d\geq 2$; $L\geq 1$; and choose $K > 1+(\lfloor 2D \rfloor+1)K^*$ where $D$ and $K^*$ are as in Lemma \ref{L:01}. Note that $L$ and $K$ do not increase in $n$. Then,
%$$ \mathbb{E} \, \mbox{var}\big(\hat{\Psi} \mid \sigma_Z\big) \, \leq \, \mbox{const.} \cdot \|G_2\|_\infty \, / \, n\,, $$
%where the constant factor only depends on $d$, $L$, $K$, $\lambda_d(S^*)$, $\rho$, $\overline{\rho}$, $\rho'$ and $\rho''$.
%\end{proposition}

\subsubsection{Expected squared conditional bias in \eqref{eq:boundsth}}\label{sec:squaredcondbias}

To prove the second inequality in \eqref{eq:boundsth} we proceed in two steps. First, for the case $L=0$ we use a bound on expected integrated distances of points in $S^*$ to their $K$-nearest neighbors provided in Lemma \ref{lem:biasone} below. Second, we use a Taylor approximation and again the bound for the conditional moments of $({\bf e}_0^\top {\cal M}_{J}(z)^{-1} {\bf e}_0)$ from Lemma \ref{L:u0bound} to cover the case $L \geq 1$.

\medskip

\textit{Step 1: The case $L=0$.}

\smallskip

Suppose that $L=0$ and that $ G \in \cG(0, \beta, \CH)$. Then 

\begin{align*}% \label{eq:bias1}
&	\E \big[\big|\E[\hat{\Psi} | \sigma_Z] - \Psi\big|^2\big]   \, = \, \E \Big[\Big|
\sum_{J\in {\cal J}_K} \,  \int_{C(J)} \frac1K\,\sum_{j\in J} \big(G(Z_{j}) - G(z)\big)\,\dd z\Big|^2\Big]\,\\
\, \leq \, &\, \CH^2\, \sum_{J, J'\in {\cal J}_K} \int_{S^*}\, \int_{S^*} \E\big[1_{C(J)}(z)\, 1_{C(J')}(z')\, \|z - Z_{(K)}(z)\|^\beta \,\|z' - Z_{(K)}(z')\|^\beta \,  \big] \, \rd z \, \rd z' \, ,
\end{align*}
where $Z_{(K)}(z)$ is the $K$ - nearest neighbor of $z$ in $Z_1, \ldots, Z_n$. 
\begin{lemma}\label{lem:biasone}
Under the Assumptions \ref{ass:deigndens} and \ref{ass:boundeffect} we have for $\beta >0$ (not necessarily $\leq 1$) and $K \in \N$ that 
$$ \sum_{J, J'\in {\cal J}_K} \int_{S^*}\, \int_{S^*} \E\big[1_{C(J)}(z)\, 1_{C(J')}(z')\, \|z - Z_{(K)}(z)\|^\beta \,\|z' - Z_{(K)}(z')\|^\beta \,  \big] \, \rd z \, \rd z' \, \lesssim n^{-\frac{2\, \beta}{d}}.$$
\end{lemma}
\begin{proof}[Proof of Lemma \ref{lem:biasone}]
For $\alpha = \alpha(z,z',J,J')$ as in \eqref{eq:alphadef1}, with the argument leading to \eqref{eq:var.4} we have that
\begin{align*}
&   \E\big[1_{C(J)}(z)\, 1_{C(J')}(z')\, \|z - Z_{(K)}(z)\|^\beta \,\|z' - Z_{(K)}(z')\|^\beta \,  \big]\\
\leq \, & \mathbb{E}\,\Big[ \alpha^{2\, \beta}  \exp\big(-(n-\#J\cup J') \cdot C_\alpha \cdot \min\{\rho',\alpha\}^d  \big)\Big]\\
\leq \, & (\rho'')^{2 \beta}\, \exp\big(- n^*\, (\rho')^d \big) + \mathbb{E}\,\big[ \alpha^{2\, \beta}  \exp\big(-n^* \cdot \alpha^d  \big)\big],
\end{align*}
where as before $C_\alpha = \rho \cdot \pi^{d/2}/ \Gamma(d/2+1)$ and $n^* = (n-2K)\, C_\alpha$.
Now
\allowdisplaybreaks
\begin{align*}
& \mathbb{E}\,\big[ \alpha^{2\, \beta}  \exp\big(-n^* \cdot \alpha^d  \big)\big] \\
= & \int_{S^{J \cup J'}}\, \max_{j\in J\cup J'}\|z-z_j\|^{2 \beta} \, \cdot \exp\big( -n^*\,\max_{j\in J\cup J'}\|z-z_j\|^d\big)\, \prod_{j \in J\cup J'} f_Z(z_j) \dd z_j\, \\
\leq & \, (\bar \rho)^{2K}\, \int_{(\R^d)^{J \cup J'}}\,\max_{j \in {J \cup J'}} \|u_j\|^{2 \beta}\,  \cdot \exp\big( -n^*\,\max_{j \in {J \cup J'}} \|u_j\|^d\big)\, \prod_{j \in J\cup J'} \dd u_j\,\\
\leq & \, (\bar \rho d C_d)^{2K}\, \int_{(0,\infty)^{J \cup J'}}\,\max_{j \in {J \cup J'}} r_j^{2 \beta}\,  \cdot \exp\big( -n^*\,\max_{j \in {J \cup J'}} r_j^d\big)\, \prod_{j \in J\cup J'} r_j^{d-1}\, \dd r_j\,\\
=\, & (\bar \rho d C_d)^{2K}\, (n^*)^{-2\beta/d}\,(n^*)^{-\#(J \cup J')\, (d-1)/d}\,\,(n^*)^{- \#(J \cup J')\,/d}  \\
& \qquad \qquad \cdot \int_{(0,\infty)^{J \cup J'}}\,\max_{j \in {J \cup J'}} r_j^{2 \beta}\,  \cdot \exp\big( - \max_{j \in {J \cup J'}} r_j^d\big)\, \prod_{j \in J\cup J'} r_j^{d-1}\, \dd r_j\,\\
%\leq & (\bar \rho d C_d)^{2K}\, (n^*)^{-2(L+1)/d}\,(n^*)^{-2K}\,  
%\int_{(0, \infty)^K}\, \int_{(0, \infty)^K}\, \big\{\max_{j=1, \ldots, K}\max\big(r_j, s_j\big\}^{2(L+1)}\, \\
%& \quad \cdot \exp\Big( -\big\{(r_1 + \ldots + r_K + s_1 + \ldots + s_K )/(2K)\big\}^d\Big)\, \prod_{j=1}^K r_j^{d-1} \, s_j^{d-1}\,  \dd r_1\, \ldots  \dd r_K\, \dd s_1\, \ldots  \dd s_K\\
\leq & \, \text{const.}\cdot  n^{-2\beta/d}\,n^{- \#(J \cup J')}\,,  
\end{align*}
where $C_d := \pi^{d/2}/\Gamma(d/2+1)$, so that observing  as in \eqref{eq:sumbinom},  
$$ \sum_{J,J'\in {\cal J}_K} \, n^{-\# J\cup J'} \leq 2^K/K!\,
$$
concludes the proof of the lemma.  
\end{proof}

\bigskip

\textit{Step 2: Taylor expansion for $L=l \geq 1$}

Now suppose that $ G \in \cG(l, \beta, \CH)$, $l \geq 1$ and that $L=l$.

The Taylor expansion of $G$ in $z$, evaluated at $Z_j$ up to order $L$ can be written as
$$G(Z_j) = \xi(z,Z_j)^\top \, {\cal G}(z) + {\cal R}(z,Z_j),$$
where
\begin{align*}
{\cal G}(z) &  :=  \big(\partial_{\kappa} G(z) \, / \, \kappa!\big)_{\kappa\in {\cal K}}\,, \\
%{\cal G}(z) &  :=  \Big(\frac{\partial^{|\kappa|} G}{\partial z^\kappa}(z) \, / \, \kappa!\Big)_{\kappa\in {\cal K}}\,, \\
{\cal R}(z,Z_j) &  :=  L\,  \sum_{|\kappa'|=L} \frac{\xi_{\kappa'}(z,Z_{j})}{\kappa'!}  \cdot \int_0^1 (1-s)^{L-1} \Big(\partial_{\kappa'}G\big(z + s(Z_{j}-z)\big) - \partial_{\kappa'}G(z)\Big) \, \dd s.
%{\cal R}(z) &  :=  (L+1) \sum_{|\kappa'|=L+1} \frac1{\kappa'!} \sum_{j\in J} \xi_{\kappa'}(z,Z_{j}) \cdot \int_0^1 (1-s)^L \frac{\partial^{L+1}G}{\partial z^{\kappa'}}\big(z + s(Z_{j}-z)\big) \, ds \cdot \big\{\xi_\kappa(z,Z_{j})\big\}_{\kappa\in{\cal K}}.
\end{align*}
Therefore, observing ${\bf e}_0^\top {\cal G}(z) = G(z)$ and \eqref{eq:estimagexp} we obtain that
\begin{align*}
& \mathbb{E}\big[ \hat{G}(z) | \sigma_Z\big]  = \, \sum_{j\in J(z)}   {\bf e}_0^\top \, {\cal M}(z)^{-1} \, \xi(z,Z_{j}) \, G(Z_j)\\
= \, & G(z) +    \sum_{j\in J(z)} \, {\bf e}_0^\top \, {\cal M}(z)^{-1} \, \xi(z,Z_{j}) \, {\cal R}(z,Z_j)
\end{align*}
Therefore, it holds that
\begin{align*}
\mathbb{E}[\hat{\Psi} | \sigma_Z] - \Psi & \, = \, \sum_{J\in {\cal J}_K} \int_{C(J)} \sum_{j\in J} \, {\bf e}_0^\top \, {\cal M}(z)^{-1} \, \xi(z,Z_{j}) \, {\cal R}(z,Z_j)  \, \dd z\,,
\end{align*}
almost surely, where we have used that the $C(J)$, $J\in {\cal J}_K$, are pairwise disjoint and form an almost-partition of $S^*$; concretely, $S^* \backslash \bigcup_{J\in {\cal J}_K} C(J)$ is included in 
$$ \bigcup_{j,k=1}^n \big\{z \in S^* \mid \|z - Z_k\| = \|z - Z_j\|\big\}\,, $$
as a subset, which is a $d$-dimensional Borel set with the Lebesgue measure zero as $Z_1$ has a Lebesgue density $f_Z$. 
Using the Cauchy-Schwarz inequality and the multinomial theorem twice we deduce that
\begin{align*}
\Big| & \sum_{j\in J(z)} \, {\bf e}_0^\top \, {\cal M}(z)^{-1} \, \xi(z,Z_{j}) \, {\cal R}(z,Z_j) \Big| \\
\leq\, &  \CH\, \Big\{\sum_{j\in J(z)} \|z - Z_j\|^{2 \beta}\, \Big(\sum_{|\kappa'|=L} \frac1{\kappa'!} \xi_{\kappa'}(z,Z_j)^2\Big)\sum_{|\kappa'|=L} \frac1{\kappa'!}\Big\}^{1/2} \\ & \hspace{5cm} \cdot \Big(\sum_{j\in J(z)} \big|{\bf e}_0^\top {\cal M}(z)^{-1} \xi(z,Z_{j})\big|^2\Big)^{1/2} \\
\leq  \, & \CH\, \, \frac{d^{L/2}}{L!} \, \Big( \sum_{j\in J(z)} \|Z_{j} - z\|^{2(L+\beta)}\Big)^{1/2} \big({\bf e}_0^\top {\cal M}(z)^{-1} {\bf e}_0\big)^{1/2}\,\\
\leq  \, & \CH\, \, \frac{d^{L/2}K^{1/2}}{L}! \,  \|Z_{(K)}(z) - z\|^{L+\beta}\, \big({\bf e}_0^\top {\cal M}(z)^{-1} {\bf e}_0\big)^{1/2}.
\end{align*}

%With respect to the bias term we consider that
%$$ \mathbb{E}(\hat{\Psi} | \sigma_Z) \, = \, \sum_{J \in {\cal J}_K} \int_{C(J)} \, {\bf e}_0^\top {\cal M}(z)^{-1} \mathbb{E}(\hat{\cal G}(z) | \sigma_Z) \, dz\,. $$
%Taylor approximation of the function $G$ around $z \in C(J)$ yields that
%\begin{align*}
%\mathbb{E}(\hat{\cal G}(z) | \sigma_Z) & \, = \, \Big\{\sum_{j\in J} G(Z_{j}) \cdot \xi_{\kappa}(z,Z_{j})\Big\}_{\kappa \in {\cal K}} \, = \, {\cal M}(z) {\cal G}(z) \, + \, {\cal R}(z)\,,
%\end{align*}
%where

%Hence,
%\begin{align*}
% {\bf e}_0^\top {\cal M}(z)^{-1} \mathbb{E}(\hat{\cal G}(z) | \sigma_Z) & \, = \, G(z) \, + \, {\bf e}_0^\top {\cal M}(z)^{-1} {\cal R}(z)\,. \end{align*}
%for all $z\in C(J)$. 
Thus the expected squared conditional bias is bounded from above as follows.
\begin{align} \nonumber 
\mathbb{E}  \big[\big|\mathbb{E}(\hat{\Psi} | \sigma_Z) - \Psi\big|^2 \big]\, 
& \, \leq \,  \CH^2 \, \frac{d^L K}{L!^2} \, \sum_{J,J'\in {\cal J}_K} \iint_{(S^*)^2} \mathbb{E} \Big[\, {\bf 1}_{C(J)}(z) {\bf 1}_{C(J')}(z') \cdot \nonumber\\ 
\big({\bf e}_0^\top {\cal M}_J(z)^{-1} {\bf e}_0\big)^{1/2}\, & \big({\bf e}_0^\top {\cal M}_{J'}(z')^{-1} {\bf e}_0\big)^{1/2} \,  \|Z_{(K)}(z) - z\|^{L+\beta}\,  \|Z_{(K)}(z') - z'\|^{L+\beta} \, \Big]\, \dd z\, \dd z'\nonumber \\
&\, \leq \,  \CH^2 \, \frac{d^L K}{L!^2} \, \sum_{J,J'\in {\cal J}_K} \iint_{(S^*)^2} \mathbb{E} \Big[\,\big({\bf e}_0^\top {\cal M}_J(z)^{-1} {\bf e}_0\big)^{1/2}\,  \,\nonumber \\ 
& \hspace{1cm} \big({\bf e}_0^\top {\cal M}_{J'}(z')^{-1} {\bf e}_0\big)^{1/2} \alpha^{2(L+\beta)}  \exp\big(-n^* \cdot \min\{\rho',\alpha\}^d  \big)\,\Big]\, \dd z\, \dd z'\label{eq:intbiasest}
\end{align}
with $\alpha = \alpha(z,z',J,J')$ as in \eqref{eq:alphadef1}, and again $n^* = (n-2K)\, C_\alpha$ and using the argument leading to \eqref{eq:var.4}. Using the Cauchy Schwarz inequality,
\begin{align*}
& \mathbb{E} \Big[\,\big({\bf e}_0^\top {\cal M}_J(z)^{-1} {\bf e}_0\big)^{1/2}\, \big({\bf e}_0^\top {\cal M}_{J'}(z')^{-1} {\bf e}_0\big)^{1/2} \,  \alpha^{2\, (L+\beta)}  \exp\big(-n^* \cdot \min\{\rho',\alpha\}^d  \big)\,\Big]\\
\leq\, & \Big(\mathbb{E} \Big[\,\big({\bf e}_0^\top {\cal M}_J(z)^{-1} {\bf e}_0\big)\,   \alpha^{2\, (L+\beta)}  \exp\big(-n^* \cdot \min\{\rho',\alpha\}^d  \big)\,\Big]\Big)^{1/2}\\
& \cdot \Big(\mathbb{E} \Big[\, \big({\bf e}_0^\top {\cal M}_{J'}(z')^{-1} {\bf e}_0\big)\,   \alpha^{2\, (L+\beta)}  \exp\big(-n^* \cdot \min\{\rho',\alpha\}^d  \big)\,\Big]\Big)^{1/2}
\end{align*}

Now the conditioning argument from step 3 in Section \ref{sec:expcondvar} can again be applied to eliminate the terms ${\bf e}_0^\top {\cal M}_{J}(z)^{-1} {\bf e}_0$ and reduce to the setting with $L=0$:

Let $\widehat{j}$ denote the smallest $j\in J$ such that $\|z-Z_j\| \geq \|z-Z_k\|$ for all $k\in J$ and define 
$$\alpha' := \max_{j \in J'\backslash J} \|z-Z_j\|,$$
where now possibly $J \cap J'=\emptyset$. 
The $\sigma$-fields generated by $\widehat{j}$, $Z_{\widehat{j}}$, $\alpha'$, on the one hand, and by $\widehat{j}$, $Z_{\widehat{j}}$, on the other hand, are again called $\mathfrak{A}'$ and $\mathfrak{A}$, respectively. 
Note that with this notation we can write $\alpha$ as 
$$\alpha = \max\{\|z-Z_{\widehat{j}}\|,\alpha'\}. $$
and that the random variable $\alpha$ is measurable with respect to $\mathfrak{A}'$. Then as in \eqref{eq:ff}, 
\begin{align*}
& \mathbb{E}\,\Big[ \big({\bf e}_0^\top {\cal M}_{J}(z)^{-1} {\bf e}_0\big) \,\,   \alpha^{2\, (L+\beta)}  \exp\big(-n^* \cdot \min\{\rho',\alpha\}^d  \big)\,\Big]\nonumber \\
\, = \, & \mathbb{E}\,\Big[ \mathbb{E}\,\big[\big({\bf e}_0^\top {\cal M}_{J}(z)^{-1} {\bf e}_0\big) \mid \mathfrak{A}'\big] \, \,   \alpha^{2\, (L+\beta)}  \exp\big(-n^* \cdot \min\{\rho',\alpha\}^d  \big)\,\Big]\nonumber \\
\, = \, & \mathbb{E}\,\Big[ \mathbb{E}\,\big[\big({\bf e}_0^\top {\cal M}_{J}(z)^{-1} {\bf e}_0\big) \mid \mathfrak{A}\big] \, \,   \alpha^{2\, (L+\beta)}  \exp\big(-n^* \cdot \min\{\rho',\alpha\}^d  \big)\,\Big]\\
\, \leq \, & C_3\, \mathbb{E}\,\Big[    \alpha^{2\, (L+\beta)}  \exp\big(-n^* \cdot \min\{\rho',\alpha\}^d  \big)\,\Big]
%
%
%\Big\{ \mathbb{E}\, \mathbb{E}\big({\bf e}_0^\top {\cal M}_J(z)^{-1} {\bf e}_0 \mid \mathfrak{A}'\big) \, \exp\big\{-(n-\#J\cup J') \cdot \rho \cdot \pi^{d/2} \cdot \min\{\rho',\alpha^*\}^d / \Gamma(d/2+1)\big\}\Big\}^{1/2}\,.
\end{align*}
where in the second step we used that the random vector $\big({\bf e}_0^\top {\cal M}_J(z) {\bf e}_0,\widehat{j},Z_{\widehat{j}}\big)$ and $\alpha'$ are independent, and in the last step we applied Lemma \ref{L:u0bound} with $\eta = 1$.  The argument is completed as in step 1. \hfill  $\qed$

\subsection{Proof of Theorem \ref{th:theoremparratematch}}\label{sec:proofmatchingest}

Let $\sigma_{X,Z}$ denote the $\sigma$-field generated by $X_1, \ldots, X_m,$\linebreak $Z_1, \ldots, Z_n$, and let $\sigma_{Z}$ denote the $\sigma$-field generated by $Z_1, \ldots, Z_n$ only. 
Then we have the decomposition
\begin{align}\label{eq:condvarbias2step}
\E \big[\big|\widehat{\Phi} - \Phi\big|^2\big] &  \, = \, \E \big[\mbox{var}\big(\widehat{\Phi} | \sigma_{X,Z}\big)\big] \, + \, \E \big[\mbox{var}\big(\E[\widehat{\Phi} | \sigma_{X,Z}] | \sigma_Z\big)\big] \,+\, \E \big[\big|\E[\widehat{\Phi} | \sigma_Z] - \Phi\big|^2\big] \,.
\end{align}
We shall show the following bounds:

\begin{align}
\sup_{G \in \cG(l, \beta, \CH, C_G)} \E \big[\mbox{var}\big(\widehat{\Phi} | \sigma_{X,Z}\big)\big] & \leq \CV \cdot C \cdot \big(  n^{-1} + m^{-1}\big),\label{eq:boundsth21}\\
\sup_{G \in \cG(l, \beta, \CH, C_G)} \E \big[\mbox{var}\big(\E[\widehat{\Phi} | \sigma_{X,Z}] | \sigma_Z\big)\big] & \leq C_G^2 \cdot C \cdot  m^{-1},\label{eq:boundsth22}\\
\sup_{G \in \cG(l, \beta, \CH, C_G)} \E \big[\big|\E[\hat{\Phi} | \sigma_Z] - \Phi\big|^2 \big] & \leq \CH^2 \cdot C\cdot  n^{-\frac{2(l + \beta)}{d}},\label{eq:boundsth23}
\end{align} 
where the constant $C>0$ only depends on $L,K,\rho, \rho',\rho'',\bar \rho, d, \lambda (S^*)$ and $\Cf$. These bounds imply \eqref{eq:riskboundmatch}, and hence the statement of the theorem. 

\medskip

\begin{proof}[Proof of \eqref{eq:boundsth21}]: By conditional independence if $J \cap J' = \emptyset$,
\begin{align*}
& \mbox{var}\big(\widehat{\Phi} | \sigma_{X,Z}\big) =  \frac1{m^2}\, \sum_{k,k'=1}^m\, \sum_{J,J' \in \cJ_K} \,{\bf 1} \big(J \cap J' \not= \emptyset\big) \\
& \hspace{5cm}\,{\bf 1} \big(X_k \in C(J), X_{k'} \in C(J')\big)\, \mbox{cov}\Big(\widehat{G}(X_k), \widehat{G}(X_{k'}) | \sigma_{X,Z}\Big).
\end{align*}
Then 
\begin{align}
\E\Big[\mbox{var}\big(\widehat{\Phi} | \sigma_{X,Z}\big) | \sigma_Z\Big]  \, = \, & \frac{1}{m}\, \int_{S^*} \mbox{var}\big( \widehat{G}(z)  \mid \sigma_Z\big)\,f(z)\,  \dd z \nonumber \\  
+ \frac{m (m-1)}{m^2}\,  \sum_{J,J'\in {\cal J}_K} & {\bf 1}\big(J\cap J' \neq \emptyset\big) \int_{C(J)} \int_{C(J')} \mbox{cov}\big( \widehat{G}(z) , \widehat{G}(z') \mid \sigma_Z\big)\,f(z)\, f(z') \, \dd z \, \dd z'.\label{eq:condtwovar}
\end{align}
Since
\begin{align} \nonumber
& \sum_{J,J'\in {\cal J}_K} {\bf 1}\{J\cap J' \neq \emptyset\} \int_{C(J)} \int_{C(J')} \mbox{cov}\big( \widehat{G}(z) , \widehat{G}(z') \mid \sigma_Z\big)\,f(z)\, f(z') \,   \dd z \, \dd z' \\  
\leq \, & \Cf^2\, \sum_{J,J'\in {\cal J}_K} {\bf 1}\{J\cap J' \neq \emptyset\} \, \int_{C(J)} \big\{\mbox{var}\big(\widehat{G}(z) | \sigma_Z\big)\big\}^{1/2} \dd z  \cdot  \int_{C(J')} \big\{\mbox{var}\big(\widehat{G}(z') | \sigma_Z\big)\big\}^{1/2} \dd z',
\end{align}
the expected value of the second term in \eqref{eq:condtwovar} can be upper bounded as the conditional variance term in \eqref{eq:boundsth}, see \eqref{eq:var.0}. For the first term in \eqref{eq:condtwovar}, we use the bound \eqref{eq:boundcondvar} on $\mbox{var}(\widehat{G}(z) | \sigma_Z)$ as well as Lemma \ref{L:u0bound} to obtain that $\E[\mbox{var}(\widehat{G}(z) | \sigma_Z)]$ is bounded from above uniformly in $z$. This proves \eqref{eq:boundsth21}. 

\end{proof}

\begin{proof}[Proof of \eqref{eq:boundsth22} and \eqref{eq:boundsth23}]

Write
\begin{equation}\label{eq:condexpec}
\bar G(z;Z) = \E\big[\hat G(z) | \sigma_{Z}\big] = \, \sum_{j\in J(z)}   {\bf e}_0^\top \, {\cal M}(z)^{-1} \, \xi(z,Z_{j}) \, G(Z_j) .
\end{equation} 

Then
$$\E\big[\hat G(X_k) | \sigma_{X,Z}\big] = \bar G(X_k;Z),$$
and therefore
$$\E[\widehat{\Phi} | \sigma_{X,Z}] =  \frac1m\, \sum_{k=1}^m\, \sum_{J \in \cJ_K} \, 1\big(X_k \in C(J)\big)\, \bar G(X_k;Z).$$
Further, 
$$\big|\E[\widehat{\Phi} | \sigma_{Z}] - \Phi \big| = \Big|    \int_{S^*} \, \big(\E\big[\hat G(z) | \sigma_{Z}\big] - G(z)\big)\, f(z)\, \rd z \Big|  \leq \Cf\,   \int_{S^*} \, \big|\E\big[\hat G(z) | \sigma_{Z}\big] - G(z)\big|\, \rd z ,$$
which can be bounded as the conditional bias term in \eqref{eq:boundsth} so that \eqref{eq:boundsth23} follows. 
Moreover, 
\begin{align*}
\mbox{var}\big(\E[\widehat{\Phi} | \sigma_{X,Z}] | \sigma_Z\big) & \, = \, \frac1m\, \mbox{var}\Big(\sum_{J \in \cJ_K} \, 1\big(X_1 \in C(J)\big)\, \bar G(X_1;Z) | \sigma_Z\Big)\\
& \, \leq \, \frac1m\, \int_{S^*}\,  \big( \E\big[\hat G(z) | \sigma_{Z}\big] \big)^2\, f(z) \, \rd z.
\end{align*}
Using the Cauchy-Schwarz inequality in \eqref{eq:condexpec} we obtain 
$$\big( \E\big[\hat G(z) | \sigma_{Z}\big] \big)^2 \leq K\, C_G^2 \,\cdot \, \sum_{j\in J(z)} \big|{\bf e}_0^\top {\cal M}(z)^{-1} \xi(z,Z_{j})\big|^2 =  K\, C_G^2\, \cdot \,{\bf e}_0^\top {\cal M}(z)^{-1}\, {\bf e}_0.$$
Using Lemma \ref{L:u0bound} proves \eqref{eq:boundsth22} and hence concludes the proof of the Theorem.  
\end{proof}

\begin{funding}
% The first author was supported by ...
%
A.~Meister is supported by the Research Unit 5381 (DFG), ME 2114/5-1.
\end{funding}

\begin{acks}[Acknowledgments]
The authors are grateful to the editors and the referees for their detailed and constructive comments.
\end{acks}

%\begin{supplement}
%In the supplement \cite{hmsupp} we provide the proof of Theorem \ref{T:low}, Lemma \ref{L:u0bound}, the statement of Remark \ref{rem:highmom} and Theorem \ref{th:ratedeconv} from the main paper, as well as further technical material and additional simulation scenarios.   \end{supplement}

	\appendix

	\section{Proof of Theorem \ref{T:low}} \label{Section:low}
	
	%\hajo{Reihenfolge u und z tauschen}\alex{erledigt}
	
	\begin{proof}[Proof of Theorem \ref{T:low}:] Set $I_k' := [k/M,(k+1/2)/M]$, $I_k'' := ((k+1/2)/M,(k+1)/M)$, $k=0,\ldots, M-1$ for some integer $M>1$. Moreover, let
		$$ \varphi_k \, := \, (1+\theta_k\alpha) 1_{I_k'} + (1-\theta_k\alpha) 1_{I_k''}\,, $$
		for some $\alpha > 0$ and $\theta = (\theta_0,\ldots,\theta_{M-1}) \in \{0\} \cup \{-1,1\}^M$; and
		$$ f_\theta \, := \, \sum_{k=0}^{M-1} \varphi_k\,. $$
		Now write $f^{(n)}$ for the $n$-fold product density of a density $f$; and let $\hat{\theta} = (\hat{\theta}_0,\ldots,\hat{\theta}_{M-1})$ be a random vector with i.i.d. components which satisfy $P[\hat{\theta}_k=1]=P[\hat{\theta}_k=-1] = 1/2$ for all $k=0,\ldots,M-1$. The densities $f_{\hat{\theta}}$ and $f_0$ are used as candidates for the design density $f_Z$ where $\alpha$ is chosen sufficiently small such that $f_{\hat{\theta}}, f_0 \in {\cal F}$ are guaranteed. Put $x=(x_1,\ldots,x_n)$ and 
		\begin{align*}
			N_k^+(x) & \, := \, \sum_{j=1}^n 1_{\{1\}}(\hat{\theta}_k) \cdot 1_{I_k'}(x_j) +  1_{\{-1\}}(\hat{\theta}_k) \cdot 1_{I_k''}(x_j) \,, \\
			N_k^-(x) & \, := \, \sum_{j=1}^n 1_{\{1\}}(\hat{\theta}_k) \cdot 1_{I_k''}(x_j) +  1_{\{-1\}}(\hat{\theta}_k) \cdot 1_{I_k'}(x_j) \,. 
		\end{align*}
		Consider that
		\begin{align} \nonumber
			f_{\hat{\theta}}^{(n)}(x) & \, = \, \prod_{k=0}^{M-1} (1+\alpha)^{N_k^+(x)}\cdot (1-\alpha)^{N_k^-(x)}\,.
		\end{align}
		
		The likelihood function $\eta_\theta^{(n)}$ equals
		$$ \eta_\theta^{(n)}({\bf u},{\bf z}) \, = \, \prod_{j=1}^n f_\theta(z_j) f_\varepsilon\big(u_j - h_\theta(z_j)\big)\,, $$
		with ${\bf u} = (u_1,\ldots,u_n)$ and ${\bf z} = (z_1,\ldots,z_n)$, where we choose the regression function $h_\theta =  \beta \cdot F(f_\theta)$ for the regression function where the function $F$ satisfies $F(1)=1$ but, apart from that, remains to be selected; and $f_\varepsilon$ denotes the standard normal density. The parameter $\beta>0$ is chosen sufficiently small such that $h_{\hat{\theta}}, h_0 \in {\cal H}$. Writing 
		$$ I_k^{+} \, := \, \begin{cases} I_k'\,, & \mbox{ if }\hat{\theta}_k=1\,, \\
			I_k''\,, & \mbox{ if }\hat{\theta}_k=-1\,, \end{cases} $$
		and
		$$ I_k^{-} \, := \, \begin{cases} I_k'\,, & \mbox{ if }\hat{\theta}_k=-1\,, \\
			I_k''\,, & \mbox{ if }\hat{\theta}_k=1\,, \end{cases} $$
		we deduce that
		\allowdisplaybreaks
		\begin{align*}
			\big|\mathbb{E} & \eta_{\hat{\theta}}^{(n)}({\bf u},{\bf z})\big|^2 \\ & \, = \, \big[f_\varepsilon^{(n)}({\bf u})\big]^2 \cdot \Big| \prod_{k=0}^{M-1} \mathbb{E} (1+\alpha)^{N_k^+({\bf z})} (1-\alpha)^{N_k^-({\bf z})} \, \\
			& \hspace{3.7cm} \cdot \exp\Big\{-\frac12 \beta^2 \big[F^2(1+\alpha) N_k^+({\bf z}) + F^2(1-\alpha) N_k^-({\bf z})\big]\Big\} \\ 
			& \hspace{3.7cm} \cdot \exp\Big\{ \beta F(1+\alpha) \sum_{j=1}^n 1_{I_k^+}(z_j) u_j + \beta F(1-\alpha) \sum_{j=1}^n 1_{I_k^-}(z_j) u_j\Big\}\Big|^2 \\ & \, = \, \big[f_\varepsilon^{(n)}({\bf u})\big]^2 \cdot \prod_{k=0}^{M-1} \Big\{\frac14 (1+\alpha)^{2N_k'({\bf z})} (1-\alpha)^{2N_k''({\bf z})} \,\\
			& \hspace{3.7cm} \cdot \exp\Big[-\beta^2 \big[F^2(1+\alpha) N_k'({\bf z}) + F^2(1-\alpha) N_k''({\bf z})\big]\Big] \\ 
			& \hspace{3.7cm} \cdot \exp\Big[ 2\beta F(1+\alpha) \sum_{j=1}^n 1_{I_k'}(z_j) u_j + 2\beta F(1-\alpha) \sum_{j=1}^n 1_{I_k''}(z_j) u_j\Big] \\
			& \hspace{1cm} + \, \frac14 (1+\alpha)^{2N_k''({\bf z})} (1-\alpha)^{2N_k'({\bf z})} \, \exp\Big[-\beta^2 \big[F^2(1+\alpha) N_k''({\bf z}) + F^2(1-\alpha) N_k'({\bf z})\big]\Big] \\ & \hspace{3.7cm} \cdot \exp\Big[ 2\beta F(1+\alpha) \sum_{j=1}^n 1_{I_k''}(z_j) u_j + 2\beta F(1-\alpha) \sum_{j=1}^n 1_{I_k'}(z_j) u_j\Big] \\
			& \hspace{1cm} + \, \frac12 (1-\alpha^2)^{N_k({\bf z})} \exp\Big[-\frac{\beta^2}2 \big(F^2(1+\alpha) + F^2(1-\alpha)\big) \cdot N_k({\bf z})\Big] \\ & \hspace{3.7cm} \cdot \exp\Big[\beta \big[F(1+\alpha)+F(1-\alpha)\big]\, \sum_{j=1}^n 1_{I_k}(z_j) u_j\Big]\Big\}\,,
		\end{align*}
		where $N_k'(x) := \sum_{j=1}^n 1_{I_k'}(x_j)$, $N_k''(x) := \sum_{j=1}^n 1_{I_k''}(x_j)$ and $N_k(x) := N_k'(x) + N_k''(x)$. On the other hand,
		$$ \eta_0^{(n)}({\bf u},{\bf z}) \, = \, f_\varepsilon^{(n)}({\bf u}) \cdot \exp(-\beta^2 n / 2) \cdot \exp\Big(\beta \sum_{j=1}^n u_j\Big)\,, $$
		on ${\bf z}\in [0,1]^n$ so that
		\begin{align*}
			\int &\big|\mathbb{E}  \eta_{\hat{\theta}}^{(n)}({\bf u},{\bf z})\big|^2 /  \eta_0^{(n)}({\bf u},{\bf z}) d{\bf u} \\ 
			& \, = \, \exp(\beta^2n/2) \cdot \prod_{k=0}^{M-1} \Big\{\frac14 (1+\alpha)^{2N_k'({\bf z})} (1-\alpha)^{2N_k''({\bf z})} \,\\
			& \hspace{3.7cm} \cdot \exp\Big[-\beta^2 \big[F^2(1+\alpha) N_k'({\bf z}) + F^2(1-\alpha) N_k''({\bf z})\big]\Big] \\ 
			& \hspace{3.7cm} \cdot \exp\Big[\frac{\beta^2}2 \big[2F(1+\alpha)-1\big]^2 N_k'({\bf z}) + \frac{\beta^2}2 \big[2F(1-\alpha)-1\big]^2 N_k''({\bf z})\Big] \\ 
			& \hspace{1cm} + \, \frac14 (1+\alpha)^{2N_k''({\bf z})} (1-\alpha)^{2N_k'({\bf z})} \, \exp\Big[-\beta^2 \big[F^2(1+\alpha) N_k''({\bf z}) + F^2(1-\alpha) N_k'({\bf z})\big]\Big] \\ & \hspace{3.7cm} \cdot  \exp\Big[\frac{\beta^2}2 \big[2F(1+\alpha)-1\big]^2 N_k''({\bf z}) + \frac{\beta^2}2 \big[2F(1-\alpha)-1\big]^2 N_k'({\bf z})\Big] \\
			& \hspace{1cm} + \, \frac12 (1-\alpha^2)^{N_k({\bf z})}  \exp\Big[-\frac{\beta^2}2 \big(F^2(1+\alpha) + F^2(1-\alpha)\big) \cdot N_k({\bf z})\Big] \\ 
			& \hspace{3.7cm} \cdot \exp\Big[\frac{\beta^2}2 \big(F(1+\alpha)+F(1-\alpha)-1\big)^2\cdot N_k({\bf z})\Big]\Big\}\,.
		\end{align*}
		Now we integrate this term with respect to ${\bf z}$. Thus consider ${\bf Z}$ as an $n$-dimensional random vector which consists of i.i.d. components that are uniformly distributed on $[0,1]$. Given $(N_0({\bf Z}),\ldots,N_{M-1}({\bf Z}))$, the random variable $N_k'({\bf Z})$ are conditionally independent and conditionally binomially distributed with the parameters $N_k({\bf Z})$ and $1/2$. Therefore,
		\begin{align} \nonumber
			\iint &\big|\mathbb{E}  \eta_{\hat{\theta}}^{(n)}({\bf u},{\bf z})\big|^2 /  \eta_0^{(n)}({\bf u},{\bf z}) d{\bf u}\, d{\bf z} \\  \nonumber & \, = \, \mathbb{E} \prod_{k=0}^{M-1} \big(\gamma_0^{N_k({\bf Z})}/2 + \gamma_1^{N_k({\bf Z})}/2\big) \\  \nonumber
			& \, = \, 2^{-M} \sum_{{\cal N} \subseteq \{0,\ldots,M-1\}} \mathbb{E}\, \gamma_0^{\sum_{k\in {\cal N}} N_k({\bf Z})} \gamma_1^{n - \sum_{k\in {\cal N}} N_k({\bf Z})} \\ \label{eq:LLB1.2}
			& \, = \,  \mathbb{E}\, \Big\{\gamma_0 \frac{B}M + \gamma_1 \Big(1 - \frac{B}M\Big)\Big\}^n\,,
		\end{align}
		where $B \sim {\bf B}(M,1/2)$ and 
		\begin{align*}
			\gamma_0 & \, := \, \frac12(1+\alpha)^2 \exp\big\{ \beta^2 \big[1-F(1+\alpha)\big]^2\big\} + \frac12(1-\alpha)^2 \exp\big\{ \beta^2 \big[1-F(1-\alpha)\big]^2\big\}\,, \\
			\gamma_1 & \, := \, (1-\alpha^2) \, \exp\big\{\beta^2 \big[1 - F(1+\alpha)\big]\cdot \big[1-F(1-\alpha)\big]\big\}\,.
		\end{align*}
		Then, (\ref{eq:LLB1.2}) equals
		\begin{align*}
			\Big(&\frac12 \gamma_0 + \frac12 \gamma_1\Big)^n \cdot \mathbb{E} \, \Big\{1 + \frac{2\gamma_0-2\gamma_1}{\gamma_0+\gamma_1} \cdot\Big(\frac{B}M - \frac12\Big)\Big\}^n \\
			& \, \leq \, \Big(\frac12 \gamma_0 + \frac12 \gamma_1\Big)^n \cdot \exp\Big\{\frac{(\gamma_0-\gamma_1)^2}{2(\gamma_0+\gamma_1)^2} \cdot\frac{n^2}M\Big\}\,.
		\end{align*}
		By Taylor expansion we obtain that
		\begin{align*} 
			\gamma_0/2 + \gamma_1/2 & \, = \, \frac14 \sum_{k=0}^\infty \frac1{k!} \, \beta^{2k} \, \Big((1+\alpha) \big[1-F(1+\alpha)\big]^{k} + (1-\alpha) \big[1-F(1-\alpha)\big]^{k}\Big)^2\,, \\
			\gamma_0-\gamma_1 & \, = \, \frac12 \sum_{k=0}^\infty \frac1{k!} \, \beta^{2k} \, \Big((1+\alpha) \big[1-F(1+\alpha)\big]^{k} - (1-\alpha) \big[1-F(1-\alpha)\big]^{k}\Big)^2\,.
		\end{align*}
		Thus the $\chi^2$-distance  between $\mathbb{E}\eta_{\hat{\theta}}^{(n)}$ and $\eta_{0}^{(n)}$ admits the following upper bound
		$$ \chi^2\big(\mathbb{E} \eta_{\hat{\theta}}^{(n)}, \eta_{0}^{(n)}\big) \, \leq \, \exp(c+c^2/16) \, - \, 1\,, $$
		for all $n$, $M$ whenever 
		\begin{align} \label{LowBou.Cond1}
			\sum_{k=1}^\infty \frac1{k!} \, \beta^{2k} \, \Big((1+\alpha) \big[1-F(1+\alpha)\big]^{k} + (1-\alpha) \big[1-F(1-\alpha)\big]^{k}\Big)^2 & \, \leq \, c/n \,, \\ \label{LowBou.Cond2}
			\sum_{k=0}^\infty \frac1{k!} \, \beta^{2k} \, \Big((1+\alpha) \big[1-F(1+\alpha)\big]^{k} - (1-\alpha) \big[1-F(1-\alpha)\big]^{k}\Big)^2 & \, \leq \, c\cdot n^{-1} M^{1/2} \,, 
		\end{align}
		for all $n$, $M$, where the constant is $c>0$ sufficiently small. 
		
		Impose that $M$ is an integer multiple of $4$. Then,
		\begin{align} \label{eq:target}
			\Big|\int_{1/4}^{3/4} h_{\hat{\theta}} - \int_{1/4}^{3/4} h_0 \Big|^2 & \, = \, \frac1{16} \beta^2 \, \big(F(1+\alpha) + F(1-\alpha) - 2\big)^2\,,
		\end{align}
		holds true almost surely. Now we specify that $F(x) = 1/x$, which provides the lower bound $\beta^2\alpha^4$ on (\ref{eq:target}) (up to a constant factor). Furthermore the left side of the inequality (\ref{LowBou.Cond1}) is bounded from above by $\alpha^4\beta^4$ (again up to some constant factor), while (\ref{LowBou.Cond2}) can always be satisfied by selecting $M$ sufficiently large. Note that there are no smoothness restrictions. Choose $\alpha$ as a positive constant sufficiently small; and $\beta = \beta_n \asymp n^{-1/4}$ with a sufficiently small constant factor. 
		
		Using Le Cam's inequality, the consideration
		\begin{align*}
			& \sup_{f_Z\in {\cal F},  h\in {\cal H}} \mathbb{E}_{f_Z,h} \Big| \hat{H}_n - \int_{1/4}^{3/4} h(x) dx\Big|^2 \\
			& \, \geq \, \frac12 \, \iint \mathbb{E} \Big| \hat{H}_n({\bf u},{\bf z}) - \int_{1/4}^{3/4} h_{\hat{\theta}}(x) dx\Big|^2 \eta_{\hat{\theta}}^{(n)}({\bf u},{\bf z})\, d{\bf u}\, d{\bf z} \\ & \quad  + \,  \frac12 \, \iint \Big| \hat{H}_n({\bf u},{\bf z}) - \int_{1/4}^{3/4} h_{0}(x) dx\Big|^2 \eta_{0}^{(n)}({\bf u},{\bf z})\, d{\bf u}\, d{\bf z} \\
			& \, \geq \, \frac1{64}\beta^2\big(F(1+\alpha)+F(1-\alpha)-2\big)^2 \cdot \iint \min\big\{\mathbb{E} \eta_{\hat{\theta}}^{(n)}({\bf u},{\bf z}) , \eta_{0}^{(n)}({\bf u},{\bf z})\big\} \, d{\bf u}\, d{\bf z} \\
			& \, \geq \, \frac1{64}\beta^2\big(F(1+\alpha)+F(1-\alpha)-2\big)^2 \cdot \Big\{1 - \Big[1 - \Big(1 - \frac12\chi^2\big(\mathbb{E} \eta_{\hat{\theta}}^{(n)}, \eta_{0}^{(n)}\big)\Big)^2\Big]^{1/2}\Big\}\,,
		\end{align*}
		concludes the proof of the theorem. 
	\end{proof}

	\section{Proof of Lemma \ref{L:u0bound}, and of statement of Remark \ref{rem:highmom}}\label{sec:technicaldetails}
	
	Recall that for $J \in {\cal J}_K$ we take $\widehat{j}$ as the smallest $j\in J$ for which $\|z-Z_j\| \geq \|z-Z_k\|$ for all $k\in J$, and denote by $\mathfrak{A}$ the $\sigma$ - field generated by $(\widehat{j}, Z_{\widehat{j}})$. 
	Further we set $\widehat{J} := J\backslash \{\widehat{j}\}$ and in the following denote
	$$ u_0 \, = \,{\bf e}_0^\top {\cal M}_J(z)^{-1} {\bf e}_0\, $$
	We need to show that for $\eta\geq 1$ and $K \geq 1 + (\lfloor 2\eta D\rfloor+1)K^*$ we have that 
	\begin{equation}\label{eq:bounduobound}
		\mathbb{E}\,[u_0^\eta \mid \mathfrak{A}] \, \leq \, C_3\,. 
	\end{equation}
	Note that  $u_0$ is the first component of the column vector ${\bf u}$ which satisfies ${\cal M}_J(z) {\bf u} \, = \, {\bf e}_0$. Multiplying this system of linear equations by ${\bf u}^\top$ from the left side we arrive at
	$$ u_0  \, = \, \sum_{j\in J} \Big(\sum_{\kappa \in {\cal K}} u_\kappa \, \xi_{\kappa}(z,Z_j)\Big)^2\,. $$
	As we are seeking for a positive upper bound on $u_0$ we may assume that $u_0>0$. Dividing by $u_0^2$ yields that
	\begin{align} \nonumber u_0 & \, = \, 1 \big/ \Big\{\sum_{j\in J} \Big(\sum_{\kappa \in {\cal K}} (u_\kappa/u_0) \, \xi_{\kappa}(z,Z_j)\Big)^2\Big\} \\ \nonumber
		& \, \leq \,  1 \big/ \Big\{\inf_{P \in {\cal P}_1} \sum_{j\in J} P^2(z-Z_j)\Big\} \\ \label{eq:u01}
		& \, \leq \, 1 \big/ \Big\{\inf_{P \in {\cal P}_1} \sum_{j\in \widehat{J}} P^2(z-Z_j)\Big\}\,,
	\end{align}
	where ${\cal P}_1$ denotes the set of all $d$-variate polynomials with the degree $\leq L$ which take on the value $1$ at $0$, and we used the notation $\widehat{J} := J\backslash \{\widehat{j}\}$ introduced above.

	%\begin{equation} \label{eq:ff1} \mathbb{E}\big({\bf e}_0^\top {\cal M}_J(z)^{-1} {\bf e}_0 \mid \mathfrak{A}'\big) \, = \, \mathbb{E}\big({\bf e}_0^\top {\cal M}_J(z)^{-1} {\bf e}_0 \mid \mathfrak{A}\big)\,, \end{equation}
	%almost surely.

	%Let $\widehat{J} := J\backslash \{\widehat{j}\}$ and
	%where  For further consideration the following result is required.
	
	\begin{lemma} \label{L:conddist}
		Given $\mathfrak{A}$, the random variables $Z_j$, $j\in \widehat{J}$, are conditionally independent and each $Z_j$ has the conditional Lebesgue density
		$$ f_{Z|\mathfrak{A}}(u) \, = \, f_Z(u) \cdot {\bf 1}_{[0,\|z-Z_{\widehat{j}}\|)}(\|z-u\|) \, / \, \int f_Z(u') \cdot {\bf 1}_{[0,\|z-Z_{\widehat{j}}\|)}(\|z-u'\|)\, \dd u'\,, \qquad u \in \mathbb{R}^d\,. $$
	\end{lemma}
	\begin{proof}[Proof of Lemma \ref{L:conddist}] Let $\Phi$ and $\Omega$ denote two probe functions and consider that
		\begin{align*}
			& \mathbb{E} \big[ \Phi\big(Z_j \, , \, j\in \widehat{J}\big) \cdot \Omega\big(\widehat{j},Z_{\widehat{j}}\big)\big] \, = \, \sum_{j'\in J} \mathbb{E} \, {\bf 1}\{\widehat{j}=j'\} \cdot \Phi\big(Z_k , k\in J\backslash \{j'\}\big) \cdot \Omega(j' , Z_{j'}) \\
			&  =  \sum_{j'\in J}  \int  \Omega(j',u) f_Z(u)  \idotsint \Phi\big(z_1,\ldots,z_{K-1}\big) \\ & \hspace{2cm} \prod_{k=1}^{K-1} f_Z(z_k) {\bf 1}_{[0,\|u-z\|)}(\|z_k-z\|)  dz_1\cdots dz_{K-1} du.
		\end{align*}
		Now, putting $\Phi\equiv 1$ and changing $\Omega$ to $\Omega \cdot \tilde{\Phi}$, we deduce that
		\begin{align*}
			& \mathbb{E} \big[ \tilde{\Phi}\big(\widehat{j},Z_{\widehat{j}}\big)\cdot \Omega\big(\widehat{j},Z_{\widehat{j}}\big)\big]   =  \sum_{j'\in J}   \int  \Omega(j',u) \tilde{\Phi}\big(j',u\big) f_Z(u)  \\
			& \hspace{5cm}\Big(\int f_Z(\zeta) {\bf 1}_{[0,\|u-z\|)}(\|\zeta-z\|) d\zeta\Big)^{K-1}\, du,
		\end{align*}
		so that
		\begin{align*}
			\mathbb{E} \big\{ \Phi\big(Z_j \, , \, j\in \widehat{J}\big) & \mid  \mathfrak{A}\big\} \, = \,  \idotsint \Phi\big(z_1,\ldots,z_{K-1}\big) \\ & \cdot \prod_{k=1}^{K-1} \frac{f_Z(z_k) {\bf 1}_{[0,\|z-Z_{\widehat{j}}\|)}(\|z_k-z\|) }{ \int f_Z(\zeta) {\bf 1}_{[0,\|z-Z_{\widehat{j}}\|)}(\|\zeta-z\|) d\zeta}  dz_1\cdots dz_{K-1},
		\end{align*}
		holds almost surely, from what follows the claim of the lemma. 
	\end{proof}

	Note that $\mathbb{P}[Z_{\widehat{j}}=z]=0$. As an immediate consequence of Lemma \ref{L:conddist}, the random variables $V_j \, := \, (z-Z_j) / \|z - Z_{\widehat{j}}\|$, $j\in \widehat{J}$, are conditionally i.i.d. and have the conditional Lebesgue density
	\begin{equation} \label{eq:fv} f_{V|\mathfrak{A}}(v) \, = \, {\bf 1}_{B_d(0,1)}(v) \cdot f_Z\big(z-v\|z-Z_{\widehat{j}}\|\big) \, \big/ \, \int_{B_d(0,1)} f_Z\big(z-v'\|z-Z_{\widehat{j}}\|\big)\, dv'\,, \qquad v\in \mathbb{R}^d\,, \end{equation}
	given $\mathfrak{A}$. Furthermore, continuing in \eqref{eq:u01} we have that
	\begin{equation} \label{eq:u02} \inf_{P \in {\cal P}_1} \sum_{j\in \widehat{J}} P^2(z-Z_j) \, = \, \inf_{P \in {\cal P}_1} \sum_{j\in \widehat{J}} P^2(V_j)\,, \end{equation}
	since $P \in {\cal P}_1$ implies that $x \mapsto P(x/\|z-Z_{\widehat{j}}\|) \in {\cal P}_1$ and vice versa. Now fix any $J^* \subseteq \widehat{J}$ with the cardinality $K^*=\# {\cal K} < K$. The joint conditional density $f^*_{V|\mathfrak{A}}$ of $(V_j , j\in J^*)$ given $\mathfrak{A}$ equals
	$$ f^*_{V|\mathfrak{A}}(v) \, = \, \prod_{j\in J^*} f_{V|\mathfrak{A}}(v_j) \, , \qquad v = (v_j , j \in J^*) \in (\mathbb{R}^d)^{K^*}\,. $$
	Preparatory to further consideration we are interested in the conditional distribution of $\|V-w\|$ given $\mathfrak{A}$ and $W:=(V-w)/\|V-w\|$ when $V$ has the conditional density $f^*_{V|\mathfrak{A}}$ given $\mathfrak{A}$; and $w=(w_j , j\in J^*) \in (\mathbb{R}^d)^{K^*}$ is deterministic. 
	
	\begin{lemma} \label{L:univar}
		(a)\; If $V$ has the conditional density $f^*_{V|\mathfrak{A}}$ given $\mathfrak{A}$, then the conditional Lebesgue density of $\|V-w\|$ given $\mathfrak{A}$ and $W$ equals
		$$ f^*_{\|V-w\||\mathfrak{A},W}(r) \, = \, {\bf 1}_{(0,\infty)}(r)\cdot f^*_{V|\mathfrak{A}}(w+rW) \, r^{dK^*-1} \, \big/ \, \int_0^\infty  f^*_{V|\mathfrak{A}}(w+sW) \, s^{dK^*-1} \, ds\, , \  r \in \mathbb{R}\,. $$
		\noindent (b)\; Writing
		\begin{align*}
			\beta_1 & \, := \, \min\{1/2,\rho'/(2\rho'')\}\,, \\
			\beta_2 & \, := \, \sqrt{2K^* + 2\beta_1^2}\,, \\
			\beta_3 & \, := \,  \beta_2^{dK^*-1} \cdot \overline{\rho}^{K^*} \cdot  d \cdot K^* \big/  \big\{ \rho^{K^*} \cdot\beta_1^{dK^*}\big\}\,,
		\end{align*}
		we assume that $\|w\|\leq \beta_1$. Then, the support of $f^*_{\|V-w\|\, | \mathfrak{A},W}$ is included in the interval $[0,\beta_2]$ and  $f^*_{\|V-w\| \,|\mathfrak{A},W}$ is bounded from above by $\beta_3$.
	\end{lemma}
	
	\begin{proof}[Proof of Lemma \ref{L:univar}] (a)\; For two probe functions $\Phi$ and $\Omega$, we consider that
		\begin{align} \nonumber
			\mathbb{E}&\big[\Phi(\|V-w\|) \Omega(W)  \mid \mathfrak{A}\big] \, = \, \int \Phi(\|v-w\|) \Omega\big((v-w)/\|v-w\|\big) \, f^*_{V|\mathfrak{A}}(v) \, dv \\ \nonumber
			& \, = \, \mathbb{E} \big[ \Phi(\|U-w\|) \Omega(\tilde{U}) \, f^*_{V|\mathfrak{A}}(U)\big] \cdot \pi^{dK^*/2} R^{dK^*} / \Gamma(dK^*/2+1) \\  \label{eq:Pr.01.01.0}
			& \, = \, \mathbb{E} \Big[ \Omega(\tilde{U}) \, \int_0^R \Phi(r) f^*_{V|\mathfrak{A}}(w + r\tilde{U}) r^{dK^*-1} dr\Big]  \cdot d \cdot K^* \cdot \pi^{dK^*/2} / \Gamma(dK^*/2+1)\,,
		\end{align}
		where $\tilde{U} := (U-w)/\|U-w\|$; and $R>0$ is sufficiently large such that the $dK^*$-dimensional ball around $w$ with the radius $R$ contains the closure of $B_d(0,1)^{(K^*)}$ -- and, hence, the support of $f^*_{V|\mathfrak{A}}$ --  as a subset; and the random vector $U$ is uniformly distributed on this ball. Therein we have used that $\|U-w\|$ and $\tilde{U}$ are independent. Now evaluating (\ref{eq:Pr.01.01.0}) for $\Phi \equiv 1$ and $\Omega$ being replaced by $\Omega\cdot\tilde{\Phi}$, we deduce that
		$$ \tilde{\Phi}(\tilde{U}) \, = \, \frac{\int_0^R \Phi(r) f^*_{V|\mathfrak{A}}(w + r\tilde{U}) r^{dK^*-1} dr }{ \int_0^R f^*_{V|\mathfrak{A}}(w + s\tilde{U}) s^{dK^*-1} ds}\,, \qquad \mbox{a.s.}\,, $$
		so that 
		$$ \mathbb{E}\big(\Phi(\|V-w\|) \mid W,\mathfrak{A}\big) \, = \, \tilde{\Phi}(W) \, = \, \frac{\int_0^R \Phi(r) f^*_{V|\mathfrak{A}}(w + rW) r^{dK^*-1} dr }{ \int_0^R f^*_{V|\mathfrak{A}}(w + sW) s^{dK^*-1} ds}, $$
		holds true almost surely. As the left side does not depend on $R$ the equality remains valid for $R\to\infty$. That completes the proof of part (a). \\
		
		\noindent (b)\; Part (a) yields that $\|w_j + rW_j\|\leq 1$ for all $j\in J^*$ with $W = (W_j , j\in J^*)$ whenever $f^*_{\|V-w\| \, |\mathfrak{A},W}(r) > 0$ since the support of $f^*_{V|\mathfrak{A}}$ is included in $B_d(0,1)^{(K^*)}$. Thus, $r\|W_j\|\leq 1+\|w_j\|$ holds for all $j\in J^*$ so that
		\begin{align*}
			r^2 & \, = \, \sum_{j\in J^*} r^2 \|W_j\|^2 \, \leq \, \sum_{j\in J^*} (2 + 2\|w_j\|^2) \, = \, 2K^* + 2\beta_1^2\,.
		\end{align*} 
		Therefore the support of $f^*_{\|V-w\| \,|\mathfrak{A},W}$ is included in the interval $[0,\beta_2]$. Moreover, we have
		\begin{align*} & f^*_{\|V-w\| ,|\mathfrak{A},W}(r) \\
			& \, = \,  {\bf 1}_{[0,\beta_2]}(r)\,\Big\{\prod_{j\in J^*} {\bf 1}_{B_d(0,1)}(w_j+rW_j) f_Z\big(z-(w_j+rW_j)\|z-Z_{\widehat{j}}\|\big)\Big\} \, r^{dK^*-1} \\ & \, \big/ \, \int_0^\infty \Big\{\prod_{j\in J^*} {\bf 1}_{B_d(0,1)}(w_j+sW_j) f_Z\big(z-(w_j+sW_j)\|z-Z_{\widehat{j}}\|\big)\Big\}  \, s^{dK^*-1} \, ds \\
			& \, \leq \,  \beta_2^{dK^*-1} \cdot \overline{\rho}^{K^*}  \, \big/ \, \Big( \rho^{K^*}\, \int_0^{\beta_1} \, s^{dK^*-1} \, ds\Big) \\
			& \, = \, \beta_3\,,
		\end{align*}
		where we have used that, for $s\in [0,\beta_1)$, it holds that $\|w_j+sW_j\|<1$; that 
		$\|w_j+sW_j\|\cdot \|z-Z_{\widehat{j}}\| < \rho'$; and, hence, $z-(w_j+sW_j) \|z-Z_{\widehat{j}}\| \in S$ as $Z_{\widehat{j}} \in S$, for all $j\in J^*$ and $z\in S^*$.
	\end{proof}

	These properties of $f^*_{\|V-w\|  \,|\mathfrak{A},W}$ are essential to establish the following result.
	\begin{lemma} \label{L:01}
		As above consider $J^* \subseteq \hat J$. 
		%and $K^*$ from Lemma \ref{L:univar}. 
		Assume that $d\geq 2$; that $L\geq 1$; and that $K^* = \# {\cal K}(d,L) < K$. Define $D$ as in \eqref{eq:constants}, that is
		$$ D \, := \, \sum_{\ell=1}^L \ell \cdot \#\big\{\kappa\in \mathbb{N}_0^d \mid \kappa_1 + \cdots + \kappa_d = \ell\big\}\,. $$ 
		By $\Xi_{d,L}({\bf x})$ we denote the square matrix which consists of the row vectors $\big(\xi_\kappa(0,{\bf x}_j)\big)_{\kappa \in {\cal K}(d,L)}$ for $j=1,\ldots,K^*$ where ${\bf x} := \big({\bf x}_1,\ldots,{\bf x}_{K^*}\big) \in (\mathbb{R}^d)^{K^*}$. Let us write $\vartheta_{d,L}({\bf x})$ for the determinant of $\Xi_{d,L}({\bf x})$. Let $V_j$, $j\in J^*$ be a conditionally i.i.d. random sample, drawn from $f^*_{V|\mathfrak{A}}$, given $\mathfrak{A}$. Write $V$ for the $dK^*$-dimensional row vector $V := (V_j \, , \, j\in J^*)$. Then, the random matrix $\Xi_{d,L}(V)$ is invertible almost surely, and
		$$ \mathbb{P}\big[|\vartheta_{d,L}(V)| \leq \varepsilon\mid \mathfrak{A}\big] \, \leq \, C_1 \cdot \varepsilon^{1/D} \, , \qquad \forall \varepsilon \in (0,\varepsilon_1)\,, $$
		almost surely for deterministic positive constants $C_1$ and $\varepsilon_1$ which only depend on $d$, $L$, $\rho$, $\overline{\rho}$, $\rho'$, $\rho''$.
	\end{lemma}
	
	\begin{proof}[ Proof of Lemma \ref{L:01}]  Clearly the function $\vartheta_{d,L}$ forms a polynomial in $dK^*$ variables (namely the components of ${\bf x}$) with the degree of at most $D$. Note that $\vartheta_{d,L}$ is non-constant since it contains exactly $K^*! \geq 2$ distinct monomials where each of them has the coefficient of either $1$ or $-1$. Hence the zero set of $\vartheta_{d,L}$ has the $dK^*$-dimensional Lebesgue measure zero. It follows from there that $\vartheta_{d,L}(V)$ does not vanish and, thus, $\Xi_{d,L}(V)$ is invertible almost surely as the random vector $V$ has the $dK^*$-dimensional Lebesgue density $f^*_{V|\mathfrak{A}}$. 
		
		We quantify the small ball probabilities of $|\vartheta_{d,L}(V)|$ around zero in terms of upper bounds. By $M>0$ we denote the maximum of $|\vartheta_{d,L}|$ on the closure of the $dK^*$-dimensional ball with the center $0$ and the radius $\beta_1/2$ where $\beta_1$, $\beta_2$ and $\beta_3$ are as in Lemma \ref{L:univar}. We specify $w$ from Lemma \ref{L:univar} as some element of that closed ball in which this maximum is taken on. Thus, $\|w\|\leq \beta_1/2$ and $|\vartheta_{d,L}(w)|=M$. Note that any non-negative, non-constant and continuous function takes on at least one maximum on any compact domain. Consider that
		\begin{align} \label{eq:Pr.01.00}
			\mathbb{P}\big[|\vartheta_{d,L}(V)| \leq \varepsilon\mid \mathfrak{A}\big] & \, = \, \mathbb{E} \big\{ \mathbb{P}\big[|\vartheta_{d,L}(w + \|V-w\|W)| \leq \varepsilon \mid W,\mathfrak{A}\big] \mid \mathfrak{A}\big\}\,,
		\end{align}
		where $W = (V-w)/\|V-w\|$. We apply the fundamental theorem of algebra to the univariate polynomial $r \mapsto \vartheta_{d,L}(w + rW)$ with the degree $D^* \leq D$; and we exploit that the absolute value of this polynomial takes on the value $M$ at $0$ so that
		$$ |\vartheta_{d,L}(w + rW)| \, \geq \, M \cdot \prod_{j=1}^{D^*} \big| 1 - r/|\zeta_j|\big|\,, $$
		for all $r\geq 0$ where $\zeta_j$, $j=1,\ldots,D^*$, are the complex roots of the polynomial. Hence,
		\begin{align} \nonumber
			\mathbb{P}&\big[|\vartheta_{d,L}(w + \|V-w\|W)| \leq \varepsilon \mid W,\mathfrak{A}\big] \, \leq \, \mathbb{P}\Big[\prod_{j=1}^{D^*} \big| 1 - \|V-w\|/|\zeta_j|\big| \leq \varepsilon /M \mid W,\mathfrak{A}\Big] \\ \label{eq:Pr.01.01}
			& \, \leq \, \sum_{j=1}^{D^*}  \mathbb{P}\Big[|\zeta_j|\cdot\big(1 + (\varepsilon/M)^{1/D^*}\big) \geq \|V-w\| \geq |\zeta_j|\cdot\big(1 - (\varepsilon/M)^{1/D^*}\big)\mid W,\mathfrak{A}\Big]\,, 
		\end{align}
		for all $\varepsilon \in (0,M/2)$. If $|\zeta_j| > \beta_2 / \big(1-(1/2)^{1/D^*}\big)$ then 
		$$ \mathbb{P}\Big[|\zeta_j|\cdot\big(1 + (\varepsilon/M)^{1/D^*}\big) \geq \|V-w\| \geq |\zeta_j|\cdot\big(1 - (\varepsilon/M)^{1/D^*}\big)\mid W,\mathfrak{A}\Big] \, = \, 0\,. $$
		Otherwise, 
		\begin{align*}
			&  \mathbb{P}\Big[|\zeta_j|\cdot\big(1 + (\varepsilon/M)^{1/D^*}\big) \geq \|V-w\| \geq |\zeta_j|\cdot\big(1 - (\varepsilon/M)^{1/D^*}\big)\mid W,\mathfrak{A}\Big] \\
			\, \leq \, & 2  \beta_2 \, \beta_3 \, (\varepsilon/M)^{1/D}  / \big(1-(1/2)^{1/D}\big).     
		\end{align*}
		
		Therefore, the term (\ref{eq:Pr.01.01}) is bounded from above by
		$$  2 D \, \beta_2 \, \beta_3 \, (\varepsilon/M)^{1/D} \, / \, \big(1-(1/2)^{1/D}\big)\,, $$
		for all $\varepsilon \in (0,M/2)$. Inserting this upper bound into (\ref{eq:Pr.01.00}) finally completes the proof. 
	\end{proof}

	The inequality from Lemma \ref{L:01} can be extended to the term $u_0$.
	\begin{lemma} \label{L:02}
		Grant the assumptions and the notation from Lemma \ref{L:01}. Then,
		$$ \mathbb{P}\Big[\inf_{P\in {\cal P}_1} \sum_{j\in J^*} P^2(V_j) \, \leq \, \varepsilon \mid \mathfrak{A}\Big] \, \leq \, C_2 \cdot \varepsilon^{1/(2D)}\,, \qquad \forall \varepsilon \in (0,\varepsilon_2)\,, $$
		for deterministic positive constants $C_2$ and $\varepsilon_2$ which only depend on $d$, $L$, $\rho$, $\overline{\rho}$, $\rho'$, $\rho ''$.
	\end{lemma}
	
	\begin{proof}[ Proof of Lemma \ref{L:02}] Consider that
		\begin{align*}
			\inf_{P\in {\cal P}_1} \sum_{j\in J^*} P^2(V_j) & \, = \, \inf_{{\bf \alpha} \in \mathbb{R}^{\cal K} , \alpha_0=1} \, \| \Xi_{d,L}(V) {\bf \alpha}\|^2  \, = \, \inf_{{\bf \alpha} \in \mathbb{R}^{\cal K} , \alpha_0=1} \, {\bf \alpha}^\top \Xi_{d,L}(V)^\top \Xi_{d,L}(V) {\bf \alpha}\,,
		\end{align*}
		so that the above term is larger or equal to the smallest eigenvalue of the symmetric and positive definite matrix $\Xi_{d,L}(V)^\top \Xi_{d,L}(V)$. On the other hand, the largest eigenvalue of $\Xi_{d,L}(V)^\top \Xi_{d,L}(V)$ is bounded from above by the product of the spectral norms of $\Xi_{d,L}(V)^\top$ and $\Xi_{d,L}(V)$ and, thus, by $\|\Xi_{d,L}(V)\|_F^2$ where $\|\cdot\|_F$ denotes the Frobenius norm. We have
		\begin{align*}
			\|\Xi_{d,L}(V)\|_F^2 & \, = \, \sum_{j\in J^*} \sum_{\kappa \in {\cal K}} \xi_\kappa^2(0,V_j)  \, \leq \, \sum_{j\in J^*} \sum_{\ell=0}^L \|V_j\|^{2\ell} \, \sum_{|\kappa|=\ell} {\ell \choose \kappa_1,\ldots,\kappa_d} \, \prod_{k=1}^d \big( V_{j,k}^2 / \|V_j\|^2\big)^{\kappa_k} \\
			& \, = \, \sum_{j\in J^*} \sum_{\ell=0}^L \|V_j\|^{2\ell} \, \leq \, (L+1) \cdot K^*\,,
		\end{align*}
		as $\|V_j\|\leq 1$ a.s., see (\ref{eq:fv}). As the determinant of a symmetric matrix equals the product of all of its eigenvalues it follows that
		\begin{align*}
			\vartheta_{d,L}^2(V) \, = \, \det\big(\Xi_{d,L}(V)^\top \Xi_{d,L}(V)\big) \, \leq \, \big((L+1) K^*\big)^{K^*-1} \cdot \inf_{P\in {\cal P}_1} \sum_{j\in J^*} P^2(V_j)\,. 
		\end{align*}
		Then, Lemma \ref{L:01} yields that
		\begin{align*}
			\mathbb{P}\Big[\inf_{P\in {\cal P}_1} \sum_{j\in J^*} P^2(V_j) \, \leq \, \varepsilon \mid \mathfrak{A}\Big] 
			& \, \leq \, \mathbb{P}\big[\big|\vartheta_{d,L}(V)\big| \leq \varepsilon^{1/2} \big((L+1)K^*\big)^{K^*/2-1/2} \mid \mathfrak{A}\big] \\
			& \, \leq \, C_1 \, \big((L+1)K^*\big)^{K^*/(2D)-1/(2D)} \cdot \varepsilon^{1/(2D)}\,, 
		\end{align*}
		for all $\varepsilon \in (0,\varepsilon_2)$ where the constant $\varepsilon_2>0$ only depends on $d$, $L$, $\rho$, $\overline{\rho}$, $\rho'$ and $\rho''$. 
	\end{proof}
	
	Now we ready to bound the conditional moments of $u_0$ as stated in Lemma \ref{L:u0bound}.%, from (\ref{eq:u01}).
	
	\begin{proof}[Proof of Lemma \ref{L:u0bound}] Combining (\ref{eq:u01}) and (\ref{eq:u02}), we deduce that
		\begin{align} \label{eq:Pruob1}
			\mathbb{E}&\big[\big({\bf e}_0^\top {\cal M}_J(z)^{-1} {\bf e}_0\big)^\eta \mid \mathfrak{A}\big] \, \leq \, \int_0^\infty \mathbb{P}\Big[\inf_{P\in {\cal P}_1} \sum_{j\in \hat{J}} P^2(V_j) < t^{-1/\eta} \mid \mathfrak{A}\Big] \, dt\,.
		\end{align} 
		There exist pairwise disjoint subsets $\hat{J}_k$, $k=1,\ldots,m$, of $\hat{J}$ with $\# \hat{J}_k = K^*$ for all $k=1,\ldots,m$, where $m > 2\eta D$. Clearly,
		$$ \inf_{P\in {\cal P}_1} \sum_{j\in \hat{J}} P^2(V_j) \, \geq \, \sum_{k=1}^m \inf_{P\in {\cal P}_1} \sum_{j\in \hat{J}_k} P^2(V_j) \geq \, \max_{k=1,\ldots,m} \inf_{P\in {\cal P}_1} \sum_{j\in \hat{J}_k} P^2(V_j)\,, $$
		where the random variables $\inf_{P\in {\cal P}_1} \sum_{j\in \hat{J}_k} P^2(V_j)$, $k=1,\ldots,m$, are i.i.d. (conditionally on $\mathfrak{A}$). By Lemma \ref{L:02}, the right side of (\ref{eq:Pruob1}) is smaller or equal to
		\begin{align*}
			\int_0^\infty \mathbb{P}&\Big[\max_{k=1,\ldots,m} \inf_{P\in {\cal P}_1} \sum_{j\in \hat{J}_k} P^2(V_j) < t^{-1/\eta} \mid \mathfrak{A}\Big] \, dt \\ & \, \leq \, \varepsilon_2^{-\eta} \, + \, \int_{t>\varepsilon_2^{-\eta}} \Big\{\mathbb{P}\Big[\inf_{P\in {\cal P}_1} \sum_{j\in \hat{J}_1} P^2(V_j) < t^{-1/\eta} \mid \mathfrak{A}\Big]\Big\}^m \, dt \\
			& \, \leq \, \varepsilon_2^{-\eta} \, + \,  C_2^m \, \int_{t>\varepsilon_2^{-\eta}} t^{-m/(2\eta D)} \, dt \\
			& \, \leq \, \varepsilon_2^{-\eta} \, + \,  C_2^m \, \frac{2\eta D}{m-2\eta D} \, \varepsilon_2^{-\eta + m/(2D)}\,,
		\end{align*}
		so that the lemma has been shown. 
	\end{proof}
	
	\begin{proof}[Proof of Remark \ref{rem:highmom}:] For any $J\in {\cal J}_K$ and integer $\ell \geq 1$, we deduce by Fubini's theorem that
		\begin{align*}
			& \mathbb{E} \lambda^\ell(C(J))\\ & \, = \, \mathbb{E} \Big(\int_{S^*} {\bf 1}_{C(J)}(x)\, dx\Big)^\ell \\
			& \, = \, \int_{S^*} \cdots \int_{S^*} \, \mathbb{P}\big[x_1,\ldots,x_\ell \in C(J)\big] \, dx_1 \cdots dx_\ell \\
			& \, = \, \int_{S^*} \cdots \int_{S^*} \, \mathbb{P}\big[\|x_m-Z_j\| < \|x_m-Z_k\| \,, \, \forall m=1,\ldots,\ell,\, j\in J, \, k\not\in J\big] \, dx_1 \cdots dx_\ell \\
			& \, = \, \int_{S^*} \cdots \int_{S^*} \, \mathbb{E} \, \mathbb{P}_Z^K\Big(\bigcap_{m=1}^\ell B_d\big(x_m,\min_{k\not\in J} \|x_m-Z_k\|\big)\Big) \, dx_1 \cdots dx_\ell\,,
		\end{align*}
		where $B(c,r)$ denotes the Euclidean ball around the center $c$ with the radius $r$; and $\mathbb{P}_Z$ stands for the image measure of $Z_1$. For some deterministic sequence $(\alpha_n)_n \downarrow 0$ we consider that
		\begin{align*}
			\mathbb{P}&\big[ \min_{m=1,\ldots,\ell} \, \min_{k\not\in J} \, \|x_m-Z_k\| \, > \, \alpha_n\big] \, = \, \mathbb{P} \big[\|x_m-Z_k\| > \alpha_n \, , \, \forall m=1,\ldots,\ell,\, k\not\in J\big] \\
			& \, = \, \mathbb{P}_Z^{n-K}\Big(\bigcap_{m=1}^\ell \mathbb{R}^d\backslash B(x_m,\alpha_n)\Big) \, = \, \Big\{1 - \mathbb{P}_Z\Big(\bigcup_{m=1}^\ell B(x_m,\alpha_n)\Big)\Big\}^{n-K} \\
			& \, \geq \, \Big\{1 - \sum_{m=1}^\ell \mathbb{P}_Z\big(B(x_m,\alpha_n)\big)\Big\}^{n-K} \, \geq \, \Big\{1 - \ell \cdot \overline{\rho} \cdot \frac{\pi^{d/2}}{\Gamma(d/2 +1)} \cdot \alpha_n^d\Big\}^{n-K}\,,
		\end{align*}
		for $n$ sufficiently large. We write ${\cal E}_n(x)$, $x=(x_1,\ldots,x_\ell)$, for the event that 
		$$\min_{m=1,\ldots,\ell} \, \min_{k\not\in J} \, \|x_m-Z_k\| \, > \, \alpha_n.$$ We fix that $\alpha_n \asymp n^{-1/d}$ so that 
		\begin{equation} \label{eq:Rev.1} \liminf_{n\to\infty} \, \inf_{x \in (S^*)^\ell} \, \mathbb{P}({\cal E}_n(x)) \, > \, 0\,. \end{equation}
		Furthermore we introduce the set 
		$$ {\cal S}_n \, := \, \big\{x=(x_1,\ldots,x_\ell) \in (S^*)^\ell \, : \, \|x_1 - x_m\| < \alpha_n/2 \, , \, \forall m=1,\ldots,\ell\big\}\,. $$
		On the event ${\cal E}_n(x)$ the ball $B(x_1,\alpha_n/2)$ is included in $\bigcap_{m=1}^{\ell} B_d\big(x_m,\min_{k\not\in J} \|x_m-Z_k\|\big)$ and in $S$ (for $n$ sufficiently large) as a subset for any $x\in {\cal S}_n$. Therefore,
		\begin{align*}
			\mathbb{E} \, \lambda^\ell(C(J)) & \, \geq \, \Big(\frac{\pi^{d/2}}{2^d \Gamma(d/2+1)}\Big)^K \cdot \rho^K \cdot \alpha_n^{Kd} \cdot \idotsint_{{\cal S}_n} \mathbb{P}\big({\cal E}_n(x)\big) \, dx_1 \cdots dx_\ell \\
			& \, \geq \, \Big(\frac{\pi^{d/2}}{2^d \Gamma(d/2+1)}\Big)^K\cdot \rho^K \cdot \alpha_n^{Kd} \cdot \inf_{x\in (S^*)^\ell} \mathbb{P}\big({\cal E}_n(x)\big) \cdot \idotsint_{{\cal S}_n} \, dx_1 \cdots dx_\ell\,.
		\end{align*}
		The $\ell d$-dimensional Lebesgue measure of ${\cal S}_n$ shall be bounded from below by $\mbox{const.}\cdot \alpha_n^{d(\ell-1)}$ as $n$ tends to $\infty$. Also note (\ref{eq:Rev.1}). Then, unfortunately, the $\ell$th moment of $\lambda(C(J))$ cannot attain the desired rate $n^{-\ell K}$ but only $n^{-\ell-K+1}$. 
	\end{proof}
	
	\section{Proofs for Section \ref{sec:applcompute}}\label{sec:proofapps}
	
	\begin{proof}[Proof of Theorem \ref{th:ratedeconv}]
		We show below that the estimators $\hat{\Psi}_{{\bf j}}$ of the Fourier coefficients $\langle h*\tilde{f}_\delta,\phi_{\bf j}\rangle$ resulting from \eqref{eq:estfinal} satisfy for $l=L$ that
		\begin{equation}\label{eq:boundFouriercoeff}
			\sup_{{\bf j}\in {\bf J}^{(n)}} \, \mathbb{E}\big[\big|\hat{\Psi}_{{\bf j}} - \Psi_{{\bf j}}\big|^2\big] \, \lesssim \, n^{-1} + (J_n \cdot n^{-1/d})^{2l}\,, 
		\end{equation} 
		since, under the condition $4(\alpha-1)(\alpha+\gamma) \geq d(2\alpha+2\gamma+d)$, for our choice of $J_n$ the first term dominates. By Parseval's identity, standard risk bounds for series estimators under the Sobolev condition \eqref{eq:sobolev} lead to 
		$$ \mathbb{E} \Big[\int_{[-\pi,\pi]^d} |\hat{h}(x) - h(x)|^2 \dd x\Big] \, \lesssim \, J_n^{-2\alpha} \, + \, J_n^d \cdot J_n^{2\gamma} \cdot n^{-1}\,, $$
		from which the rate \eqref{eq:boundberkson} follows upon inserting our choice of $J_n$. 
		
		\smallskip
		
		Concerning \eqref{eq:boundFouriercoeff}: Since $\|\phi_{{\bf j}}\|_\infty \leq 1$ the constant $C_V$ in Theorem \ref{th:theoremparrate} is set equal to the variance $\mbox{var}\, \epsilon$; whereas
		$$ \big\|\partial_{i_1}\cdots\partial_{i_l} G_{{\bf j}}\big\|_\infty \, \leq \, \sum_{M\subseteq\{1,\ldots,l\}} \big\|\partial_{i_M} \big(h*\tilde{f}_\delta\big)\big\|_\infty \cdot J_n^{l-\#M}\,, $$
		for all ${\bf j} \in {\bf J}^{(n)}:=\{-J_n,\ldots,J_n\}^d$ where $\partial_{i_1},\ldots,\partial_{i_l}$, $i_1,\ldots,i_l \in \{1,\ldots,d\}$, denote the partial differential operators with respect to the components $z_{i_1},\ldots,z_{i_l}$ where we allow for coincidence of the $i_1,\ldots,i_l$, and $\partial_{i_M}$ stands for the chaining of all $\partial_{i_k}$, $k\in M$. Therein note the symmetry of the partial derivatives. By the Cauchy-Schwarz inequality and the Fourier representation of the Sobolev norm, we consider that
		\begin{align*}
			\big\|\partial_{i_M} \big(h*\tilde{f}_\delta\big)\big\|_\infty & \, = \, \|(\partial_{i_M} h)*\tilde{f}_\delta\|_\infty \, \leq \, \Big(\int |\partial_{i_M} h|^2\Big)^{1/2}\cdot \Big(\int |f_\delta|^2\Big)^{1/2} \\
			& \, \leq \, (2\pi)^{-d/2} \cdot (C_\alpha^*)^{1/2}\cdot \Big(\int |f_\delta|^2\Big)^{1/2} \,,  
		\end{align*}
		when $l\leq \alpha$, so that 
		$$  \big\|\partial_{i_1}\cdots\partial_{i_l} G_{{\bf j}}\big\|_\infty \, \leq \, C_H \, = \, C_H({\bf j},l) \, := \, (2\pi)^{-d/2} \cdot (C_\alpha^*)^{1/2}\cdot \Big(\int |f_\delta|^2\Big)^{1/2} \cdot (1+J_n)^l\,, $$ 
		for ${\bf j}\in {\bf J}^{(n)}$ when $l \leq \lceil\alpha\rceil-1$. Then, $G_{{\bf j}} \in {\cal G}(l-1,1,C_H({\bf j},l))$ and Theorem \ref{th:theoremparrate} yields \eqref{eq:boundFouriercoeff}. 
		%
		%If we assume that $\alpha$ is 
		%
		%then under the 
	\end{proof}

			\section{Discussion of Asymptotic normality} \label{asno}
			
			In this section we discuss the asymptotic distribution of the estimator $\hat{\Psi}$. In the simple case when $\Psi$ is defined in (\ref{eq:simpleest}) we may consider the standardized sum
			$$ \hat{\Delta}_n \, := \, n^{1/2} \, \sum_{j=1}^n \big\{g(U_j,Z_j) - \mathbb{E}(g(U_j,Z_j)|\sigma_Z)\big\} \cdot \lambda(C_j)\,,$$
			and study this random variable by usual techniques from the proof of the Central Limit Theorem. The conditional characteristic function of $\hat{\Delta}_n$ given $\sigma_Z$ at some fixed $t\in \mathbb{R}$ may be decomposed as
			\begin{align*}
				\prod_{j=1}^n \Big\{1 - \frac{n}2 t^2 \cdot \mbox{var}\{g(U_j,Z_j)|Z_j\} \cdot \lambda^2(C_j) \pm {\cal O}\big(n^{3/2} \cdot \lambda^3(C_j)\big)\Big\}\,,      
			\end{align*}
			when the function $g$ is imposed to be bounded. Considering Remark \ref{rem:highmom} (in particular, $\mathbb{E} \lambda^3(C_j) \lesssim n^{-3}$) we deduce that, whenever $n \cdot \sum_{j=1}^n \mbox{var}\{g(U_j,Z_j)|Z_j\} \cdot \lambda^2(C_j)$ converges to some deterministic $\sigma^2>0$ in probability, then the random sequence $(\hat{\Delta}_n)_n$ converges weakly to a $N(0,\sigma^2)$-distributed random variable. 
			
			Of course, the question arises if a Gaussian limiting distribution can also be established for the advanced estimator $\hat{\Psi}$ from (\ref{eq:estfinal}). It is much harder to address since the summands in (\ref{eq:estfinal}) suffer from some specific weak dependence (only those summands are conditionally independent given $\sigma_Z$ for which the corresponding subsets $J$ are disjoint) and the higher order moments of $\lambda(C(J))$ show irregular behavior as explained in Remark \ref{rem:highmom}. Thus we have to leave the asymptotic distribution of the estimator (\ref{eq:estfinal}) open for future research.

			\begin{figure}   
				\centerline{
					\includegraphics[scale=0.4]{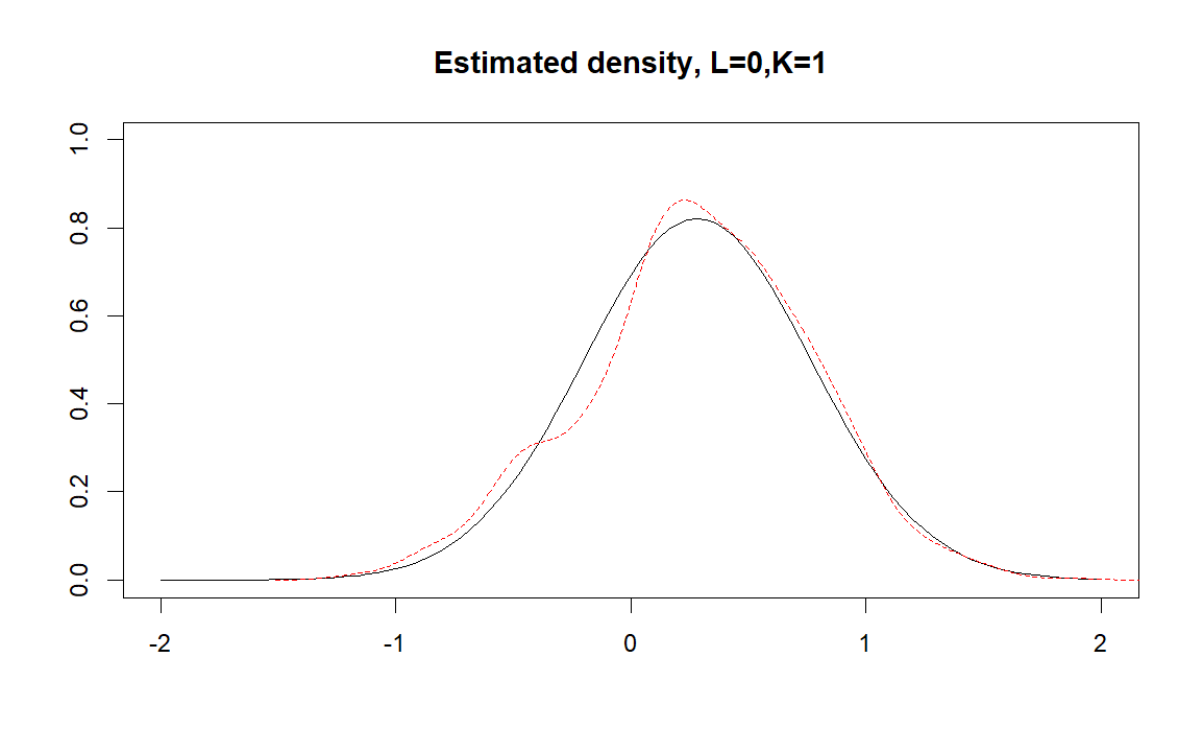}
					\includegraphics[scale=0.4]{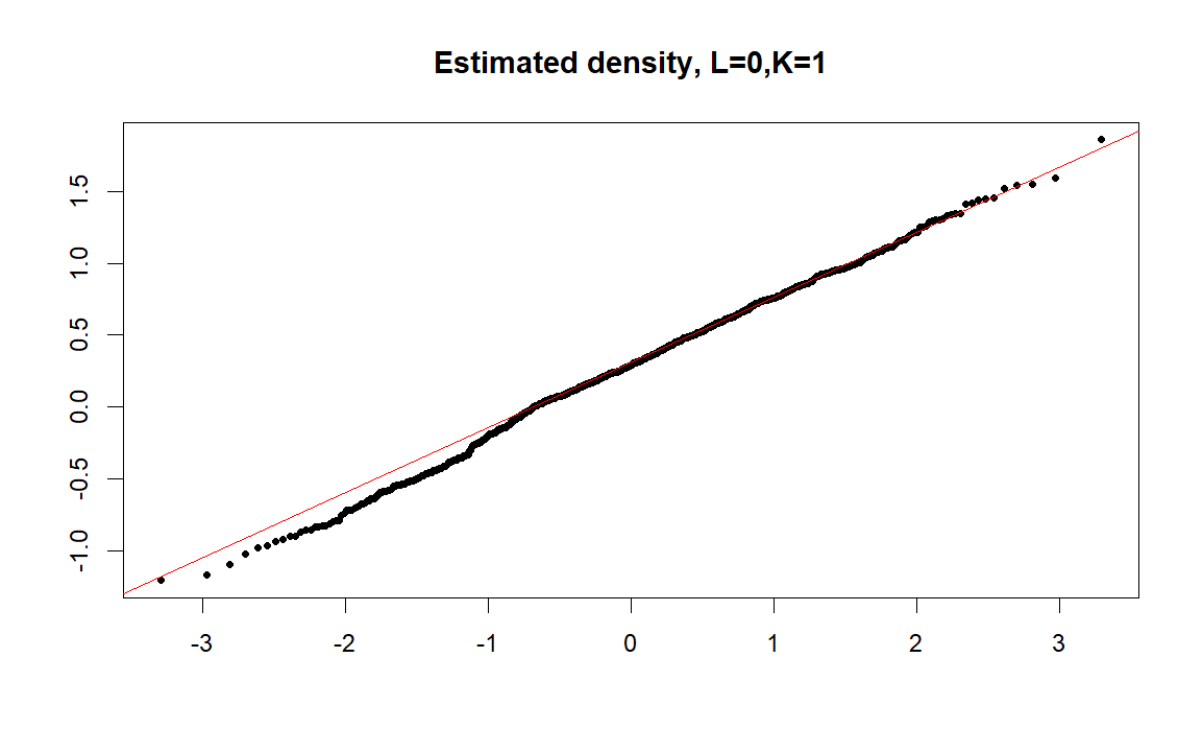}
				}
				\caption{Plots of $\sqrt{n}(\hat \Psi - \Psi)$ in simulation scenario $f_1$ with $L=0$, $K=1$, $n=1000$ over $N=1000$ repetitions. Left: Density plot (black curve) and normal density with estimated parameters (red curve). Right: QQ-Plot against standard normal, and qqline.  }  \label{fig:1}
			\end{figure}
			
			\begin{figure}   
				\centerline{
					\includegraphics[scale=0.4]{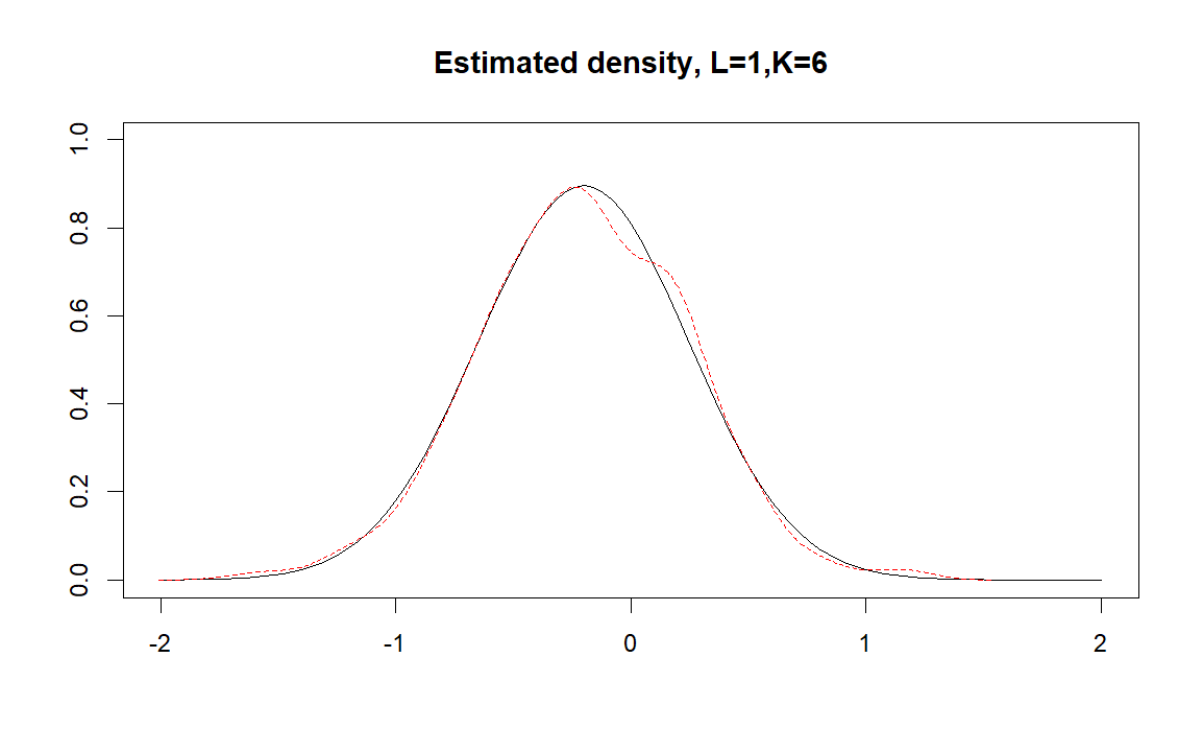}
					\includegraphics[scale=0.4]{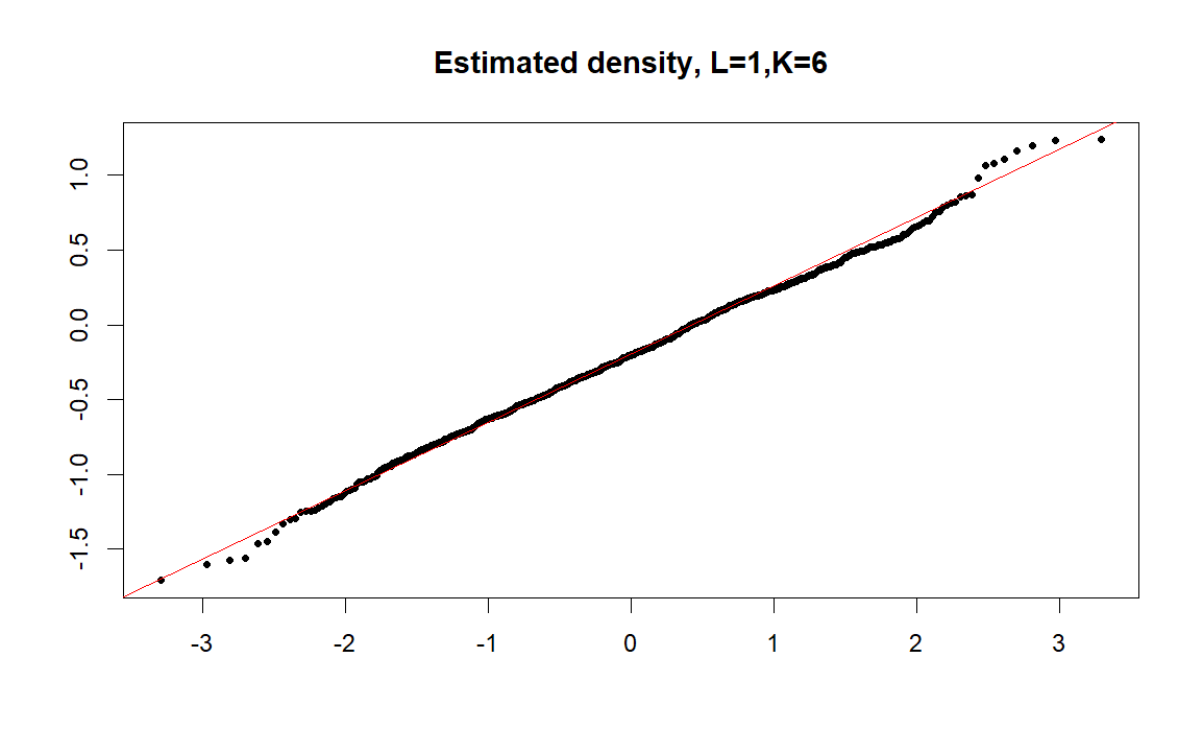}
				}
				\caption{Plots of $\sqrt{n}(\hat \Psi - \Psi)$ in simulation scenario $f_1$ with $L=1$, $K=6$, $n=1000$ over $N=1000$ repetitions. Left: Density plot (black curve) and normal density with estimated parameters (red curve). Right: QQ-Plot against standard normal, and qqline.}  \label{fig:2}
				
				%\caption{Plots of the optimal threshold $\delta^\ast$ as a function of $\pi$ for the different performance metrics used in Example \ref{ex:lda}}  \label{fig:lda1}
			\end{figure}

	\section{Additional simulations}\label{sec:addsim}
	
	 As a second scenario, in $d=3$ we now use the full support $[0,1]^3 =S^*=S$ to investigate potential boundary effects. We keep the function $f_1$, consider $n=1000$ (sample size), $m=10000$ (additional sample from uniform distribution on $S^*$), and $N=1000$ repetitions. The covariates are a sample from a mixture with weight $1/2$ of a uniform distribution on $[0,1]^3$, and a product of Beta-distributions with parameters $\alpha=\beta=3$ each.  The true parameter turns out to be $\Psi = 5.36$. 
		
		\setcounter{table}{2}
		
		\begin{table}[h]
			\begin{tabular}{lll|llll}
				$n=1000$ & $L=0$ & $K$ & 1 & 2 & 3 & 5  \\ \hline 
				& & $\sqrt n \cdot $ BIAS& 1.2945 & 1.5428 & 1.7177 & 1.8671  \\
				&& $\sqrt n \cdot $ STDV&  2.8676 & 3.0474 & 3.2148 & 3.4135  \\
				&& $\sqrt n \cdot $ RMSE&  {\bf 3.1463} & 3.4157 & 3.6449 & 3.8907 \\\hline
				& $L=1$ & $K$ & 5 & 6 & 7 & 8\\\hline
				&& $\sqrt n\cdot $ BIAS&  -0.4196 & -0.4585 & -0.5596 & -0.4818  \\
				&& $\sqrt n\cdot $ STDV&  3.4760 & 3.0252 & 3.0288 & 2.9639 \\
				&& $\sqrt n\cdot $ RMSE&  3.5012 & 3.0597 & 3.0801 & {\bf 3.0028} 
			\end{tabular}
			\caption{Simulation scenario with $f_1$ in $d=3$ and $[0,1]^3 =S=S^*$ (support of the covariates)}\label{tab:scen1}
		\end{table}
		
		We can still observe consistency and the overall reduction in bias from $L=0$ to  $L=1$ so that, at least in this scenario, the method seems to work for $S^*=S$.

	As a third scenario we consider a setting in $d=2$   with regression function 
	$$  f_2(z_1,z_2) =  z_1\, z_2^2-1+ \cos(z_1/z_2)$$
	and $Y_i = f_2(Z_{i,1},Z_{i,2}) + 0.2\, \varepsilon_i$, otherwise the parameters are the same as before. Results are displayed in Table \ref{tab:scen1}.  
	Higher values of $K$ perform better in this setting, but the method does not seem to depend sensitively on the choice of $K$ for $L=1$ as long as $K$ is not chosen too small. Again there is some reduction in bias from $L=0$ to $L=1$, in particular for $n=100$.

	\begin{table}[h]
		\begin{tabular}{lll|llllll}
			$n=100$ & $L=0$ & $K$ & 1 & 3 & 4 & 6&  8 & 10  \\ \hline
			& & $\sqrt n \cdot $ BIAS&  0.0156 & 0.0290 & 0.0351 & 0.0469 & 0.0574 & 0.0674  \\
			&& $\sqrt n \cdot $ STDV&  0.1000 & 0.0969 & 0.0990 & 0.1021 & 0.1067 & 0.1097  \\
			&& $\sqrt n \cdot $ RMSE&  0.1012 & {\bf 0.1011} & 0.1050 & 0.1124 & 0.1212 & 0.1287 \\\hline
			& $L=1$ & $K$ & 4 & 8 & 10 & 12 & 14 & 16\\\hline
			&& $\sqrt n\cdot $ BIAS&  -0.0002 & 0.0009  & 0.0003 & -0.0002 & -0.0032 & -0.0058\\
			&& $\sqrt n\cdot $ STDV&  0.1585 & 0.0972 & 0.0971 & 0.0981 & 0.0997 & 0.1002\\
			&& $\sqrt n\cdot $ RMSE&  0.1585 & 0.0972 & {\bf 0.0971} & 0.0981 & 0.0997 & 0.1004 \\\hline \hline
			$n=1000$ & $L=0$ & $K$ & 1 & 3 & 4 & 6&  8 & 10  \\ \hline 
			& & $\sqrt n \cdot $ BIAS& 0.0060 & 0.0163 & 0.0108 & 0.0216 & 0.0244 & 0.0285   \\
			&& $\sqrt n \cdot $ STDV&  0.1152 & 0.1064 & 0.1091 & 0.1030 & 0.1015 & 0.1042 \\
			&& $\sqrt n \cdot $ RMSE&  0.1153 & 0.1076 & 0.1096 & 0.1052 & {\bf 0.1044} & 0.1080  \\\hline
			& $L=1$ & $K$ & 4 & 8 & 10 & 12 & 14 & 16\\\hline
			&& $\sqrt n\cdot $ BIAS&  0.0026 & 0.0081 & 0.0115 & 0.0117 & 0.0134 & 0.0150 \\
			&& $\sqrt n\cdot $ STDV&  0.1664 & 0.1023 & 0.1030 & 0.1045 & 0.1005 & 0.1031\\
			&& $\sqrt n\cdot $ RMSE&  0.1664 & 0.1026 & 0.1036 & 0.1052 & {\bf 0.1014} & 0.1041  
		\end{tabular}
		\caption{Simulation scenario with $f_2$ in $d=2$}\label{tab:scen1}
	\end{table}

	%\begin{align*}	
	%	 \Psi \, &  := \, \int_{S^*} \int g(u,z) f_{U|Z}(u|z) du \, dz \,
	%	\end{align*} 

%$f_{U|Z}$ denotes the conditional Lebesgue density of $U_1$ given $Z_1$. The function $g$ is assumed to be known. 

%While the responses $U_j$ take their values in $\mathbb{R}$ the covariates $Z_j$ are $\mathbb{R}^d$-valued random variables with the Lebesgue density $f_Z$. 

%We construct a multivariate Priestley-Chao type estimator of $\Psi$ which avoids estimation of the density of $Z$ and division by this estimator; and incorporates local polynomial fitting to reduce the bias term in higher dimension.

%Let ${\cal J}_K$ be the collection of all subsets of $\{1,\ldots,n\}$ with exactly $K$ elements. 
%Then we define the $K$th order Voronoi cells
%$$ C(J) \, := \, \big\{z \in S^* \, : \, \max\{\|z-Z_{j}\| : j \in J\} < \|z - Z_{k}\| \, , \, \forall k \in \{1,\ldots,n\} \backslash J \big\}\,, \qquad J \in {\cal J}_K\,.$$
%For any fixed $z \in C(J)$ we introduce the notation
%for some integer $L\geq 0$; and base our estimator on the minimizer of the term

%\bibliographystyle{imsart-nameyear}
%\bibliography{biblio}

\end{document}